\documentclass[a4paper,11 pt]{amsart}

\usepackage[normalem]{ulem}
\usepackage{lineno}

\usepackage{comment}

\usepackage{xpatch}
\xpatchcmd{\sout}
  {\bgroup}
  {\bgroup}
  {}{}

\usepackage{amssymb}  
\usepackage{mathtools}      
\usepackage{mathabx}        
\usepackage[bb=fourier,cal=euler,scr=rsfs]{mathalfa}	
\usepackage{enumitem}       

\usepackage[colorlinks=true,backref=page]{hyperref} 

\usepackage[ansinew]{inputenc}

\usepackage{accents}
\newlength{\dhatheight}

\renewcommand*{\backref}[1]{}
\renewcommand*{\backrefalt}[4]{\quad \tiny
  \ifcase #1 (\textbf{NOT CITED.})%
  \or    (Cited on page~#2.)%
  \else   (Cited on pages~#2.)%
  \fi}

\hypersetup{bookmarksdepth = 3} 
\numberwithin{equation}{section}     

\setlist[enumerate,1]{label={\upshape(\roman*)},ref=\roman*}
\setlist[enumerate,2]{label={\upshape(\alph*)},ref=\alph*}

\newcommand{\N}{\mbox{$\mathbb{N}$}}

\newcommand{\R}{\mbox{$\mathbb{R}$}}

  \def\CC{{\mathbb C}} \def\DD{{\mathbb D}}

 \def\RR{{\mathbb R}}  
   
 \def\ZZ{{\mathbb Z}}

\def\cA{\mathcal{A}}    \def\cS{\mathcal{S}}
\def\cB{\mathcal{B}}  \def\cH{\mathcal{H}}  \def\cT{\mathcal{T}}
    
   \def\cP{\mathcal{P}} 
\def\cE{\mathcal{E}}   \def\cQ{\mathcal{Q}} \def\cW{\mathcal{W}}
   \def\cR{ \mathrm{Rat} } 
\def\cY{\mathcal{Y}}  

\newtheorem*{teo*}{Theorem}

\newtheorem{teo}{Theorem}[section]
\newtheorem{thm}[teo]{Theorem}
\newtheorem{theorem}[teo]{Theorem}

\newtheorem{cor}[teo]{Corollary}
\newtheorem{fact}[teo]{Fact}

\newtheorem{claim}[teo]{Claim}
\newtheorem{lema}[teo]{Lemma}
\newtheorem{lemma}[teo]{Lemma}
\newtheorem{prop}[teo]{Proposition}

\theoremstyle{definition}
\newtheorem{defi}[teo]{Definition}
\theoremstyle{remark}

\newtheorem{remark}[teo]{Remark}

\newtheorem{ej}[teo]{Example}

\newcommand{\finobs}{\par\hfill{$\diamondsuit$} \vspace*{.05in}}

\newcommand{\eps}{\varepsilon}

\newcommand{\supp}{\mathrm{supp}}

\newcommand{\ary}{\begin{eqnarray}}
\newcommand{\eary}{\end{eqnarray}}
\newcommand{\aryst}{\begin{eqnarray*}}
\newcommand{\earyst}{\end{eqnarray*}}
\newcommand{\enmt}{\begin{enumerate}}
\newcommand{\eenmt}{\end{enumerate}}

\newcommand{\clblue}{\color{blue}}
\newcommand{\clb}{\color{black}}

\title[Geometric properties of partially hyperbolic measures]{Geometric properties of partially hyperbolic measures and applications to measure rigidity}

\author[A. Eskin]{Alex Eskin} 
\address{Department of Mathematics, University of Chicago, Chicago, Illinois 60637, USA}
\email{eskin@math.uchicago.edu}
\urladdr{}
\author[R.~Potrie]{Rafael Potrie} 
\address{Centro de Matem\'atica, Universidad de la Rep\'ublica, Uruguay}
\email{rpotrie@cmat.edu.uy}
\urladdr{http://www.cmat.edu.uy/~rpotrie/}

\author[Z. Zhang]{Zhiyuan Zhang} 
\address{Department of Mathematics,  Imperial College London,  United Kingdom}
\email{ zhiyuan.zhang@imperial.ac.uk}
\address{Previous affiliation: CNRS, Institut Galile\'e, Universit\'e Paris 13. 99 Av. Jean Baptsite Cl\'ement, 93430, Villetaneuse}
\email{ zhiyuan.zhang@math.univ-paris13.fr}
\urladdr{}

\thanks{Rafael Potrie was partially supported by CSIC. Zhiyuan Zhang was supported by the National Science Foundation under Grant No. DMS-1638352. This work was started while the authors were members of IAS and they would like to thank the IAS for the excellent conditions offered. }
\begin{document}

\begin{abstract}
We give a geometric characterization of the {\it quantitative joint non-integrability}, introduced by  Katz in \cite{Katz}, of strong stable and unstable bundles of partially hyperbolic measures and sets in dimension 3. This is done via the use of higher order templates for the invariant bundles. 
Using the recent work of Katz, we derive some consequences, including the measure rigidity of $uu$-states and  the existence of physical measures. 
\end{abstract}

\maketitle


\section{Introduction}

Let $f: M \to M$ be a partially hyperbolic diffeomorphism of a closed 3-manifold: the tangent bundle $TM= E^u \oplus E^c \oplus E^s$ splits into $Df$-invariant one dimensional bundles with the property that there is some integer $N>0$ such that for every $x \in M$, we have
\aryst
 \|Df^N|_{E^s(x)} \| &\leq& \frac{1}{2} \min\{1, \|Df^N|_{E^c(x)} \|\}  \\
 &<& 2 \max\{1, \|Df^N|_{E^c(x)} \|\} \leq \|Df^N|_{E^u(x)}\|.
 \earyst
Any such diffeomorphism $f$ admits (uniquely defined) $f$-invariant foliations $\cW^s$ and $\cW^u$ tangent respectively to the bundles $E^s$ and $E^u$ (see e.g. \cite{CP}).

 Consider a lamination $\Lambda \subset M$ which is $f$-invariant and $\cW^u$-saturated. The geometric properties of its leaves,  when projected along stable holonomy, are very relevant to understanding several problems: ergodicity of conservative systems (e.g. \cite{BurnsWilkinson}), finiteness of attractors (e.g. \cite{CPS}), mixing properties (e.g. \cite{TZ}), among other properties.  More recently, some quantitative measures of joint non-integrability have been used  by Katz \cite{Katz} to obtain measure rigidity results based on ideas coming from homogeneous and Teichmuller dynamics \cite{EL,EM} (related progress is that of random dynamical systems \cite{BRH}, see also \cite{Obata} for its connection with partially hyperbolic dynamics). In this paper, 
 we intend to look into the notion of quantitative non-joint integrability (QNI) proposed by \cite{Katz}.  We consider here exclusively $C^\infty$ diffeomorphisms,  and obtain in this setting  equivalent notions that seem more conceptual and easier to verify and work with.

\begin{defi}\label{defi.jointorderell}
We say that a compact invariant set $\Lambda$ of a partially hyperbolic diffeomorphism $f: M \to M$ is \emph{jointly integrable up to order $\ell$} if there is $\rho>0$ and a continuous  family of $C^\ell$ smooth surfaces $\{  \cS_x  \}_{x \in \Lambda}$ which verifies that:  
\enmt
\item  $\cW^u_{\rho}(x) \cup \cW^s_{\rho}(x)\subset \cS_x$,
\item for every $x \in \Lambda$ and $y \in \cW^u_{\rho}(x) \cap \Lambda$ (resp. $y \in \cW^s_{\rho}(x) \cap \Lambda$) we have that $\cW^s_{loc}(y)$ is tangent to order $\ell$ to $S_x$ at $y$ (resp. $\cW^u_{loc}(y)$ is tangent to order $\ell$ to $\cS_x$ at $y$).
\eenmt  
\end{defi}

Here, when we say that the curve $\gamma$ is tangent to order $\ell$ to $\cS_x$ we mean that there is a constant $C > 0$ \footnote{Later in the paper we will also work with a measurable version of this, for partially hyperbolic measures, see Definition~\ref{defi.jointorderell-meas}.} such that when parametrized by arc-length the distance from a point $y \in \gamma$ to the surface $\cS_x$ is less than $C t^\ell$ where $t$ is the arc-length from $y$ to $x$.

Our main results concern the study of $uu$-states of partially hyperbolic systems. By definition, an ergodic $uu$-state is an ergodic invariant measure that is absolutely continuous with respect to strong unstable manifolds of the foliation $\cW^u$. These measures always exist (see e.g. \cite[\S 11]{BDV}) and are usually the place to look for \emph{physical measures} (i.e. those for which the statistical basin has positive Lebesgue measure).  

The results in this paper are also obtained in the more general setting of partially hyperbolic measures where analogous results hold. While very similar, the proofs require more careful analysis in some parts of the argument. We refer the reader to \S\ref{section-statements} for precise statements.

Our techincal result Theorem \ref{teo.uniformversion} in Section \ref{section-statements}, combined with the recent results of \cite{Katz}, immediately gives the following.
\begin{teo} \label{thm: main}
Let $f: M \to M$ be a $C^\infty$ partially hyperbolic diffeomorphism on a closed $3$-manifold and let $\mu$ be a $uu$-state  with positive center Lyapunov exponent, then, either $\mu$ is physical, or the support of $\mu$ is jointly integrable up to order $\ell$ for every $\ell>0$. 
\end{teo}

Note that in \cite{ABV} the physicality of $uu$-states is proved under the assumption that \emph{every} such measure has positive center exponents. 

In principle similar results may hold in higher dimensions which may be worth investigating. This may involve adapting some definitions to take care of some higher dimensional phenomena that may occur. We decided to restrict to the 3-dimensional case since it already presents some challenges and applications. We note that right now the results in \cite{Katz} require one-dimensional center, but there are extensions to higher dimensional centers in the work in progress \cite{BEFRH}. 

We also note that our results require very high regularity to compensate for the fact that we deal with the case where holonomies are not regular (which is the usual case). In some cases, there are reasons that force more regularity of holonomies, even in open sets, and in those cases recently arguments have been made to obtain similar results assuming less regularity of the map, see \cite{ALOS}.  
Theorem \ref{thm: main} will be used in \cite{ACEPWZ} to understand $uu$-states of partially hyperbolic Anosov diffeomorphisms in dimension 3 (addressing a conjecture of \cite{GKM}) and will be strengthened to show that if one assumes that the strong unstable foliation of a partially hyperbolic diffeomorphism of a 3-dimensional manifold fills center unstable disks\footnote{More precisely, a minimal subset of the strong unstable foliation verifies that it 'fills center unstable disks' if it contains open sets in some center unstable disk.} , then joint integrability up to order $\ell$ implies actual joint integrability.

The main technical contribution of this paper is to extend the notion of templates introduced in \cite{TZ} to partially hyperbolic dynamics, in particular, dealing with higher order templates to deduce quantitative forms of non-integrability of dynamically defined bundles.

\smallskip

{\small {\em Acknowledgements:} The authors have benefited from important input and suggestions by S. Alvarez, S. Crovisier, R. Elliot Smith, M. Leguil, D. Obata and B. Santiago that allowed to clarify some parts of the proofs and improve the presentation overall. }

\section{Context and main technical result}\label{section-statements}

Throughout this paper we let $f: M \to M$ be a $C^\infty$-diffeomorphism\footnote{All results hold in finite regularity which depends on the properties (Lyapunov exponents) of the measure one looks at   as well as some uniform constants of $f$ around the support of the measure.   We will not attempt to estimate the precise regularity since in any case it will be usually very high.} of a closed 3-manifold $M$.  
 We fix a smooth Riemannian metric $\| \cdot \|_0$ on $M$.

\subsection{Partially hyperbolic measures}   \label{subsec PHM}
 
An ergodic $f$-invariant measure  $\mu$ is  \emph{partially hyperbolic} if the following is true: 
\begin{itemize} 
\item $f$ has \emph{simple spectrum}. Namely, $f$ has three different Lyapunov exponents $\chi_1 > \chi_2 > \chi_3$, 
\item $\chi_1> 0 > \chi_3$. 
\end{itemize}

We denote by $E^1, E^2, E^3$ the Oseledets bundles for $\mu$ corresponding to $\chi_1, \chi_2, \chi_3$ respectively, and denote by $\cW^1$,  $\cW^3$ the Pesin laminations associated to $E^1, E^3$ respectively (see \cite{BP} and also \S\ref{ss.nf} for more properties of these laminations). 

Throughout the paper, we fix some $0 < \epsilon \ll \min_{i\in \{1,2\}} |\chi_i - \chi_{i+1}|$, and denote by  $\| \cdot \|$  the  Lyapunov norm (with parameter $\epsilon$)  for $\mu$ satisfying the following property:
For $\mu$-almost every $x \in M$ we have, for $i\in \{1,2,3\}$

\begin{equation}\label{eq:exponentpoint} 
 \|D_xf|_{E^i(x)}\| := \frac{\|D_xf(v)\|}{\|v\|} \in (e^{\chi_i - \epsilon}, e^{\chi_i + \epsilon}) \ , \ v \in E^i(x) \setminus \{0\}. 
\end{equation}

If we fix an orientation on each of the (one dimensional) bundles $E^i$ we get a vector $e^i(x)$ in $E^i(x)$  with unit Lyapunov norm $\|\cdot \|$ for almost every $x \in M$ and $i\in \{1,2,3\}$. We define $ \lambda_{i,x} \in \RR  $ by equation: 
\begin{equation}\label{eq:exponentpointsign} 
D_xf (e^i(x)) =  \lambda_{i,x} e^i(f(x)).
 \end{equation}
By definition, we have that $\lambda_{i,x} \in \{ \pm  \|D_xf|_{E^i(x)}\| \}$.

The general measure theory allows us to disintegrate the measure $\mu$ along the leaves of $\cW^1$ and $\cW^3$. We will denote by $\mu^i_x$ (with $i\in \{1,3\}$) the \emph{conditional measure along the leaves of $\cW^i$}, $i\in \{1,3\}$, (see \cite{BP}). 

\begin{defi}\label{standing-noatoms} 
An ergodic $f$-invariant partially hyperbolic measure $\mu$ will be called \emph{non-degenerate} if for almost every $x \in M$ the measures $\mu^1_x$ and $\mu^3_x$ are without atoms. 
\end{defi}

A measure $\mu$ is called a $uu$\emph{-state} if for $\mu$-a.e. $x$ the measure $\mu^1_x$ is absolutely continuous with respect to the length induced by the Riemannian metric on the leaves of $\cW^1$. Note that if $\chi_2>0$ then $E^1 \oplus E^2$ is $\mu$-a.e. tangent to the leaves of the Pesin unstable lamination which we denote by $\cW^{12}$. 
We denote the disintegration of $\mu$ along $\cW^{12}$ at a $\mu$-typical point $x$ by $\mu^{12}_x$. A measure $\mu$ is said to be \emph{SRB} (Sinai-Ruelle-Bowen) in this context if $\mu^{12}_x$ is absolutely continuous with respect to the Riemannian volume along the leaves of $\cW^{12}$.

\begin{remark}\label{rem.nondeg}
Note that by Ledrappier-Young's entropy formula in \cite{LY2}, any $uu$-state which has $\chi_2> -\chi_1$(in particular, when $\chi_2 > 0$) must be non-degenerate. 
\end{remark}

\subsection{Partially hyperbolic sets} 

A particularly important case in our discussion is the one where the diffeomorphism $f: M \to M$ is \emph{partially hyperbolic}.   
More generally, we let $f: M \to M$ be a smooth diffeomorphism of a closed 3-manifold and let $\Lambda$ be a compact $f$-invariant set admitting a \emph{partially hyperbolic splitting} $T_\Lambda M := TM |_{\Lambda} = E^u \oplus E^c \oplus E^s$ which is, by definition, $Df$-invariant and verifies that there is an integer $N>0$  so that for every $x\in \Lambda$ we have:
\aryst
 \|Df^N|_{E^s(x)} \|_0 &<& \min\{1, \|Df^N|_{E^c(x)} \|_0\}   \\
  &\leq&  \max\{1, \|Df^N|_{E^c(x)} \|_0\} < \|Df^N|_{E^u(x)}\|_0.
 \earyst
In this case, we call $\Lambda$ a  (uniformly) \emph{partially hyperbolic set} for $f$. 
Note that every ergodic $f$-invariant measure supported in $\Lambda$ is partially hyperbolic. (See e.g. \cite{BDV,CP} for more properties of these objects.)

Replacing $\| \cdot \|_0$ by an appropriate smooth Riemannian metric adapted to the dynamics, 
we can always assume that $N=1$ in the above inequalities (see \cite{CP}).
By a slight abuse of notation, we will denote such a metric by $\| \cdot \|$ in analogy with the Lyapunov metric in the case of partially hyperbolic measures.  
However we stress that these two metrics are usually not the same: 
The Lyapunov norm of a given measure supported on $\Lambda$ assigns an inner product to almost every point in a  measurable way, but it needs not be continuous, or everywhere defined on $\Lambda$.

It is known that every partially hyperbolic diffeomorphism admits at least one ergodic $uu$-state, but the existence of SRB measures is not clear in general (see \cite[Chapter 11]{BDV}). We state the following for later reference: 

\begin{fact}
Let $f: M \to M$ be a partially hyperbolic diffeomorphism. Then there exists a partially hyperbolic measure $\mu$ which is a $uu$-state. 
\end{fact}

We note that the same holds if there is a partially hyperbolic attractor (i.e. there is an open set $U$ such that $f(\overline{U}) \subset U$ and the set $\Lambda = \bigcap_{n>0} f^n(U)$ is partially hyperbolic). 

\subsection{Normal Forms}\label{ss.nf}
We refer the reader to \cite[\S 3.1]{KK} for more details and \cite{KS} for more general results. 

In the following, for $\mu$-a.e. $x$, we identify $T_x\cW^i(x)$ with $\RR$ so that the unit vector $e^i(x)$ corresponds to $1$. 

\begin{prop}\label{prop-normalforms} 
Let $\mu$ be a partially hyperbolic measure. Then for $i\in \{1,3\}$ and $\mu$-almost every $x\in M$ there exists $\Phi^i_{x}: T_x\cW^i(x) \to \cW^i(x)$ a smooth diffeomorphism such that: 
\begin{enumerate}
\item $x \mapsto \Phi^i_x$ varies measurably,
\item $\Phi^i_x(0) = x$ and $D_0\Phi^i_x = \mathrm{id}$, 
\item $f(\Phi^i_x(t)) = \Phi^i_{f(x)}(\lambda_{i,x} t)$ for every $t \in \R$,
\item if $y \in \cW^i(x)$ then $(\Phi^i_y)^{-1} \circ \Phi^i_x$ is an affine map. 
\end{enumerate}
\end{prop}

Note that in (ii) we have identified $\Phi^{i}_x$ with a diffeomorphism from $\RR$ to $\cW^i(x)$ through the above identification between $\R$ and $T_x\cW^i(x)$. 
From now on, we will fix a collection of maps $\Phi^i_x$, $i = 1, 3$, given by Proposition \ref{prop-normalforms}.

 \begin{remark}\label{remark.sign}
 The sign of $\lambda_{i,x}$ depends on the chosen orientation of the bundles $E^i$ at $x$. It is sometimes impossible to find a continuous orientation of the bundles, so it cannot be made so that the values are always positive even after taking iterates or finite covering. For the purposes of this paper, this is not an issue, so we will sometimes assume that $\lambda_{i,x}$ is always positive  to simplify the exposition when it is possible to treat the general case in a similar way. 
 \end{remark}

 We denote $W^i_r(x) = \Phi^i_x((-r,r))$, $i \in \{1, 3\}$.  
 We denote by $W^i_{loc}(x)$ a neighborhood of $x$ in $\cW^i(x)$ whose size may vary from line to line. 
 Since we will use some dynamically defined scales, we introduce the following notation for each integer $k>0$ and each $\rho > 0$: 
\begin{equation}\label{eq:notationscale}  
W^{1,k}_\rho(x) := f^{-k}(W^1_\rho(f^k(x))), \ \  W^{3,k}_\rho(x) : = f^k (W^3_\rho(f^{-k}(x))) . 
\end{equation}

Let $i\in \{1,3\}$.
Recall that $\mu^i_x$ is defined in Subsection \ref{subsec PHM}. We denote
\begin{equation}\label{eq:measurecoordinate}
\hat \mu^i_x = [(\Phi^i_x)^{-1} ]_\ast \mu^i_x.
\end{equation} 
The above conditions determine $\hat \mu^i_x$ as a Radon measure on $\RR$ up to a multiple. Given $i\in \{1,3\}$, we have $\hat \mu^i_{f(x)} = c f_\ast \hat \mu^i_x$ for some $c  > 0$. 
 In the following we normalise $\hat \mu^i_x$ so that its restriction to $(-1,1)$ is a probability measure. 
 With a slight abuse of notation, we use $\hat \mu^i_x$ to denote the probability measure restricted to $(-1,1)$.  

The following is an alternative way to characterize $uu$-states.
\begin{prop}\label{p.ustate}
The measure $\mu$ is an $uu$-state if and only if the measures $\hat \mu^1_x$ defined in \eqref{eq:measurecoordinate} are Lebesgue. 
\end{prop}

 See \cite[\S 6.5]{BRH} for a proof based on the rigidity result of Ledrappier-Young \cite{LY}.

In a similar way as in \S\ref{ss.nf} one can find continuous\footnote{Note that here the one-dimensionality of the bundle is crucial for this result. Here we have ignored the issue with orientability for simplicity. See Footnote \ref{footnote7} for explanation.} normal form coordinates for partially hyperbolic set, in dimension 3 (see \cite{KK}):

\begin{prop}\label{prop-normalforms-uniform} 
Let $f: M \to M$ be a smooth diffeomorphism of a 3-manifold $M$ and let $\Lambda$ be a compact $f$-invariant partially hyperbolic set. For every $x\in \Lambda$ there exists $\Phi^i_{x}: T_x\cW^i(x) \to \cW^i(x)$ a smooth diffeomorphism such that: 
\begin{enumerate}
\item $x \mapsto \Phi^i_x$ varies  continuously,
\item $\Phi^i_x(0) = x$ and $D_0\Phi^i_x = \mathrm{id}$,  \item $f(\Phi^i_x(t)) = \Phi^i_{f(x)}(\lambda_{i,x}t)$ for every $t \in \R$,
\item if $y \in \cW^i(x)$ then $(\Phi^i_y)^{-1} \circ \Phi^i_x$ is an affine map. 
\end{enumerate}
\end{prop}
 
 Remark \ref{remark.sign} applies to this proposition too. 
 
\subsection{Quantitative non-integrability} 

Recently, in \cite{Katz} the author proposed a geometric condition on $uu$-states that allows one to apply the scheme introduced in \cite{EM,EL}. Let us recall the following crucial definition in \cite{Katz}  (although this notion is only defined for $uu$-states in \cite{Katz},  it can be stated for partially hyperbolic measure considered here): 

\begin{defi}\label{def.QNI} 
A partially hyperbolic measure $\mu$ has the \emph{quantitative non-integrability} property (QNI) for $f$  if: 
\begin{itemize}
\item there is $\alpha>0$ and, 
\item for every $\eps>0$ a subset $\cP \subset M$ of measure $\mu(\cP)>1-\eps$ and,
\item for every $\nu > 0$ constants $C=C(\nu,\eps) > 0$ and $k_0 = k_0(\nu,\eps) > 0$ such that:
\end{itemize}
\qquad  if an integer $k \geq k_0$ and  $x \in \cP$ satisfy $f^k(x), f^{-k}(x) \in \cP$ then
\begin{itemize}
\item
there is a subset $S_x \subset W^{3,k}_1(x)$\footnote{Recall notation \eqref{eq:notationscale}.}  with $\mu_x^3(S_x) > (1-\nu) \mu_x^3(W^{3,k}_1(x))$ satisfying the following property: 

\item  for every $y \in S_x$ there exists $U_y \subset  W^{1,k}_1(x)$ with $\mu_x^1(U_y) > (1-\nu) \mu_x^1(W^{1,k}_1(x))$ so that if $z \in U_y$ then 
\begin{equation}\label{eq:QNI}
d(W^1_1(y), W^3_1(z)) > C  e^{-\alpha k}.   
\end{equation}
\end{itemize}
\end{defi}

We do not assume that the measure is a $uu$-state because this allows us to define the notion in a more general setting; and even though our main application is for $uu$-states we wish to make the arguments symmetric:  

\begin{prop}\label{prop.symmetry}
A partially hyperbolic measure $\mu$ has QNI for $f$ if and only if it has QNI for $f^{-1}$.  
\end{prop}

The proof is a simple Fubini argument that we postpone to Appendix \ref{app.symmetry}. In Appendix \ref{app.symmetry} we also discuss this definition as well as other formulations and compare them with the ones in the work of Katz \cite{Katz}.

\begin{remark}
The main difference between our definition and that of \cite{Katz} is the notion of local stable and local unstable manifolds. For notational simplicity (helped by the fact that we are working with one dimensional stable and unstable strong manifolds) we consider subsets of $W^{i, k}_1$, while in \cite{Katz} the local stable and unstable manifolds are considered with respect to a measurable partition of the stable/unstable measurable (Pesin) lamination. The consideration in \cite{Katz} is more natural and it extends better to higher dimensions. We could have chosen to use this formalism, but some arguments where we reduce to cocycles defined on fixed intervals would be more cumbersome to write. We explain the equivalence of the definitions in more detail in Appendix \ref{app.symmetry}. 
\end{remark}

\subsection{Cocycle normal forms and good charts} 

We will consider good coordinate charts that incorporate the normal coordinates as in \cite[\S 4]{TZ}. 

\begin{defi}[$0$-good unstable charts]\label{def.normalcharts}
Let $\mu$ be a partially hyperbolic measure. A measurable collection of smooth diffeomorphisms  $\{\imath_x:  (- \| Df \|,  \| Df \| )^3  \to M\}_{x \in M}$  is a family of \emph{unstable charts}  if it verifies that for $\mu$-almost every $x \in M$ we have that $\imath_x(t_1,0,0)= \Phi^1_x(t_1)$, $\imath_x(0,0,t_3)= \Phi^3_x(t_3)$ for $t_1,t_3 \in (-1,1)$, $\partial_2 \imath_x(0,0,0)$ is the unit vector $e^2(x)$ in $E^2(x)$. Moreover, if we write $F_x := \imath_{f(x)}^{-1} \circ f \circ \imath_x = (F_{x,1},F_{x,2}, F_{x,3})$ then the map $F_x:  ( - \| Df \|,  \| Df \| )^3   \to \RR^3$ verifies that: 
\begin{enumerate}
\item $\partial_2 F_{x,2}(t,0,0) =  \lambda_{2, x}$ for all $t \in (-1,1)$,
\item  $\partial_3 F_{x,3}(t,0,0) =  \lambda_{3, x}$ for all $t \in (-1,1)$,
\item $\partial_2 F_{x,3}(t,0,0) = 0$ for all $t \in (-1,1)$.  
\end{enumerate}

We say that a family of  unstable charts is  {\em $0$-good}
if for some constant $d$ (independent of $x$) we have that 
\begin{equation}\label{eq:goodcoordinate0}
\partial_3 F_{x,2}(t,0,0)  \text{ is a polynomial  of degree} \leq d \text{ in } t \in (-1,1).
\end{equation}

\end{defi}

Note that in \cite[\S 4]{TZ} similar charts are constructed for Anosov flows. It is not hard to adapt the argument to our case.  We will prove the following in \S\ref{s.cocyclenormalforms}.

\begin{prop}\label{p.existnormalcoord}
For every partially hyperbolic measure $\mu$ there is a family of $0$-good unstable charts.
\end{prop}

 Given a family of unstable charts, for $\mu$-a.e. $x$, the map $F_x$ satisfies that for every $t_1 \in (-1,1)$ we have that $F_x(t_1,0,0)= (\lambda_{1,x}t_1,0,0)$ and 
\begin{equation}\label{eq:derivativenormalchart}
\begin{pmatrix}
 \frac{\partial F_{x, 2}}{\partial x_2}  & \frac{\partial F_{x, 2}}{\partial x_3}  \\[6pt] \frac{\partial F_{x, 3}}{\partial x_2} & \frac{\partial F_{x, 3}}{\partial x_3}  \end{pmatrix} (t_1, 0, 0) \clb = \begin{pmatrix}  \lambda_{2, x} & r_x(t_1) \\ 0 &  \lambda_{3, x}\end{pmatrix}.
\end{equation} 
 We may think of \eqref{eq:derivativenormalchart} as representing a $2$-dimensional \emph{linear cocycle} over $f$ in a family of unstable charts. 
The construction of this linear cocycle will be detailed in \S\ref{s.cocyclenormalforms}.

In equation \eqref{eq:derivativenormalchart} the function $r_x$ is a smooth function. It follows from the general theory of {\it cocycle normal forms}, developed in \cite{BEFRH},  that one can change coordinates in order to make $r_x$ a polynomial of degree depending only on the values of the functions $\lambda_{2,x}$ and $\lambda_{3,x}$ (see Proposition~\ref{p.CNF} below). This is how Proposition \ref{p.existnormalcoord} is proven.

 Note that the strong stable bundle along the strong unstable manifold can be modeled as a section of  this cocycle (cf. \S\ref{ss.CNF}), what will be referred to as a \emph{template} (see Definition \ref{def.lgoodcharts} below).  Since the cocycle is 2 dimensional and can be taken smoothly into an upper triangular form (cf. equation \eqref{eq:derivativenormalchart}) we can think of this template, under the normal form coordinates, as a function on  the strong unstable manifold. Therefore this reduction allows one to distinguish between the case where such template is a polynomial or not. This is a reformulation of one of the main observations from \cite{Tsujii,TZ} (see \cite[Remark 1.2]{Tsujii}). 
 
We will show that whenever the template is not a polynomial, then the QNI condition is verified. Else, one can continue doing this for higher order $\ell+1$-good charts of the stable manifolds along a strong unstable manifold, see Theorem \ref{teo.maintechnical1}.  Figure \ref{fig-flowchart1} illustrates schematically the way Theorem \ref{teo.maintechnical1} works.  

\begin{figure}[ht]
\begin{center}
\includegraphics[scale=0.84]{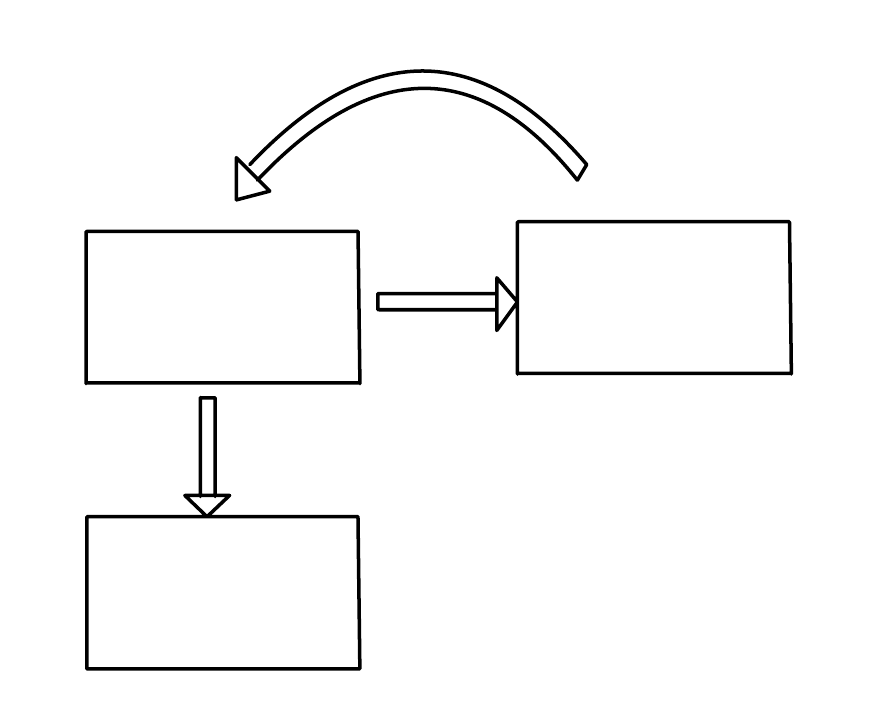}
\begin{picture}(0,0)
\put(-189,104){$\exists \ell-$ \text{good}} 
\put(-179,90){\text{charts}} 
\put(-92,104){$\exists (\ell+1)$-\text{good}}
\put(-75,90){\text{charts}}
\put(-135,140){{\tiny \text{$\ell \mapsto \ell+1$}}}
\put(-125,92){{\tiny \text{smooth}}}
\put(-128,86){{\tiny \text{template}}}
\put(-164,72){{\tiny \text{non-smooth}}}
\put(-160,66){{\tiny \text{template}}}
\put(-177,29){\text{QNI}}

\end{picture}
\end{center}
\vspace{-0.5cm}
\caption{{\small Schematic illustration of the way QNI is obtained if to a certain order there are no $\ell$-good (unstable) charts .}}\label{fig-flowchart1}
\end{figure}
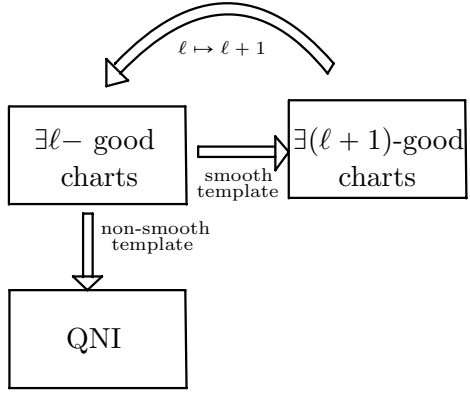

We introduce the following notion.

\begin{defi}[$\ell$-good unstable charts]\label{def.lgoodcharts}
Let $\{\imath_x\}_{x \in M}$ be a family of unstable charts for a partially hyperbolic measure $\mu$. Let $\ell$ be a positive integer.  We say that the family is $\ell$-\emph{good} if for $\mu$-almost every $x\in M$ there is a unique  \footnote{It is unique almost everywhere and up to zero  $\hat \mu_x^1$  measure.} collection of measurable  functions $\cT^\ell_{x}: (-1,1) \to \RR$, $a_x, b_x: (-1,1)^2 \to \RR$
such that for  $\hat \mu_x^1$-almost every $t \in (-1,1)$  we have that: 
\begin{equation}\label{eq:highordertemplate} 
\imath_x^{-1}(W^3_{ loc }(\Phi^1_x(t))) = \{ (t  + a_x(t,s)s, \cT^\ell_{x}(t)  s^{\ell + 1}  + b_x(t,s)  s^{\ell+2},s )  :  s \in (-1,1) \}
\end{equation} 
and for some constant $d:=d(\ell,f, \mu)$ (independent of $x$) we have that  
\begin{equation}\label{eq:goodcoordinate}
 \partial_3^{\ell+1} F_{x,2}(t,0,0)  \text{ is a polynomial of degree } \leq d \text{ in } t \in (-1,1).
\end{equation}
In this case, we call $\cT^\ell_x$ in equation \eqref{eq:highordertemplate} a  \emph{stable template of} $(\ell+1)$-\emph{jets} at $x$.

One can define in a similar way $\ell$-good stable charts for $\mu$. 
\end{defi}

\begin{remark}
We point out again the fact that the relevant conditions about $\ell$-good unstable charts at a point $x$ all concern information that can be read in arbitrarily small neighborhoods of $W^1_{1}(x)$ and therefore to analyze the existence of such charts it is enough to understand the associated linear cocycles along the unstable manifold. This will be expanded in \S
\ref{s.cocyclenormalforms}.  In particular, in Proposition \ref{p.normalform} (see in particular Lemma \ref{lem.derivativezero}) it is shown that under the assumptions of the definition, the lower order derivatives vanish so that condition \eqref{eq:goodcoordinate} makes sense. 
\end{remark}

We note that since the leaves of the invariant laminations are smooth, the functions $a_x(t, s)$ and $b_x(t,s)$ are smooth in $s$ for $\hat \mu^1_x$-almost every $t \in (-1,1)$. In particular, there is a measurable function $c_x: (-1,1) \to \R_+$ such that  for $\hat \mu^1_x$-almost every $t \in (-1,1)$ and for any $|s|< 1$ we have
\begin{equation}\label{eq:boundab}
|a_x(t,s)|, |b_x(t,s)| < c_x(t). 
\end{equation}
Moreover, we have the following, which will be used in Section \ref{s.compatible}.
\begin{lemma} \label{lem uniformboundab}
For every $\eps > 0$, there exist a constant $C > 1$ and a subset $\cP \subset M$ with $\mu(\cP) > 1 - \eps$, and for every $x \in \cP$, 
for every $\nu > 0$, there exists an integer $m_0 > 0$ such that for every integer $m > m_0$, the set of $t \in (-1, 1)$ satisfying $|c_x(t)| < C$  and $\Phi^1_x(t) \in W^{1, m}_1(x)$  has $\hat \mu^1_x$-measure at least $(1-  \nu ) \mu^1_x(W^{1, m}_1(x))$. The same statement holds if we consider $\mu^3_x$, $W^{3, m}_1(x)$  in place of $\mu^1_x$, $W^{1, m}_1(x)$.
\end{lemma}
\begin{proof}
Fix an arbitrary $\eps > 0$. By Lusin's lemma, we can find a compact subset $\cQ \subset M$ such that $\mu(\cQ) > 1 - \eps^2$, and $W^1_{loc}(x)$, $W^3_{loc}(x)$ as well as the chart $\imath_x$ depend continuously on $x$ restricted to $\cQ$. Then by definition, for any $x \in \cQ$ and $t \in (-1,1)$ such that $\Phi^1_x(t) \in \cQ$, we see that $|c_x(t)|$ can be chosen uniformly bounded from above. Then the lemma follows immediately from Proposition \ref{prop.measurable}.
\end{proof}

\begin{remark}
Note that the stable templates depend on the charts.  In \cite{TZ} the stable templates at $x$ are taken to be the family of all possible $\cT^0_x$ as we change the underlying $0$-good unstable charts. We emphasize that we usually expect to have non-smooth $\cT^\ell_x$. Indeed, one of the main points here is that if $\cT^\ell_x$ is smooth in some regions, then one can produce a higher order approximation. 
\end{remark}

\begin{remark}\label{rem-HB} 
 The existence of $\ell$-good charts implies that the {\em stable Hopf brush} at a point $x$, by which we mean $\cH_x^s = \bigcup_{t \in (-1,1)} W^3_1(\Phi^1_x(t))$, is more regular than expected: it can be approximated to order $\ell$ by the stable templates of $\ell$-jet.           
One has a similar approximation for the \emph{unstable Hopf brush} defined by $\cH_x^u= \bigcup_{t \in (-1,1)} W^1_1(\Phi^3_x(t))$.  Note that the regularity of $\cH_x^s$ and $\cH_x^u$ may be pretty bad, but the templates used to approximate these sets to high order may have good regularity.
\end{remark}

One useful consequence of \eqref{eq:goodcoordinate} is the following simple computation: 

\begin{remark}\label{rem.dynamicstemplate}
Note that the condition \eqref{eq:highordertemplate} together with the properties of unstable normal coordinate charts imply that: 
\begin{equation}\label{eq:dynamicstemplate} 
\frac{\lambda_{2,x}}{\lambda_{3,x}^{\ell+1}} \cT^\ell_x(t) + \frac{1}{\lambda_{3,x}^{\ell+1}} \frac{ \partial_3^{\ell+1} F_{x,2}(t,0,0)}{(\ell+1) !} = \cT^\ell_{f(x)}(\lambda_{1,x} t). 
\end{equation}
If \eqref{eq:goodcoordinate} is verified, we know that $\frac{1}{(\ell+1) ! \lambda_{3,x}^{\ell+1}} \partial_3^{\ell+1} F_{x,2}(t,0,0)$ is a polynomial in $t$ which depends only on the coordinates we have chosen. Consequently, the property that $\cT^\ell_x$ is a polynomial of degree $\leq d$ is independent of the choice of the $\ell$-good chart. See Proposition \ref{prop.cocycle} for more details. 
\end{remark}

 Before we state the main inductive step for proving Theorem \ref{thm: main}, we recall the notion about Whitney smoothness.   
 
For a function $\varphi: (-1,1) \to \RR$ and $K \subset (-1,1)$ a compact set, we say that $\varphi$ is {\it $C^r$ in the sense of Whitney  on $K$ }  if 
 there exists a $C^r$ function $\tilde\varphi$ on an open neighborhood of $K$ such that $\tilde\varphi|_{K} = \varphi$.
Another equivalent condition (cf.  Whitney's extension theorem, see \cite{W}) is given by the existence of continuous functions $a_i : K \to \R$, $1 \leq i \leq r$, satisfying a family of compatibility conditions (see \cite{W}). In particular,  for any $t, s \in K$ we have
\begin{equation}\label{eq:whitney} 
\left| \varphi(s)- (\varphi(t) + a_1(t)(s-t) + \ldots + a_r(t)(s-t)^r ) \right| = o( |s-t|^{r} ). 
\end{equation}
We say that $\varphi$ is \emph{smooth in the sense of Whitney on $K$} if it is $C^r$  in the sense of Whitney  on $K$  for every integer $r>0$.

We will prove in \S\ref{s.dichotomy} the following proposition.
\begin{prop}[Dichotomy]\label{p.dichotomy}
Let $\mu$ be a partially hyperbolic measure with $\ell$-good   unstable charts. Then either there are $(\ell+1)$-good unstable charts, or, for almost every $x \in M$ we have that $\cT^\ell_x$ as defined in \eqref{eq:highordertemplate} is not smooth in the sense of Whitney restricted to any subset of $\cW^1(x)$ with positive $\mu^1_x$-measure (in particular, it is not a polynomial of degree $\leq d$).
\end{prop}

 We can see from the above proposition that  the smoothness of $\mathcal{T}^{\ell}_x$ (an intrinsic property about $(f, \mu)$) can be expressed naturally using normal coordinates  (see Proposition \ref{p.ustate} for another application of such an idea). 
 It says that if these $\ell$-order approximations of the strong stable lamination are smooth along the strong unstable direction then one has the a priori stronger condition that they are polynomial in the normal coordinates.

\subsection{Compatible charts} 
Note that the time one map of the geodesic flow on a constant negatively curvatured surface admits $\ell$-good  stable and unstable charts for every $\ell>0$, and, at the same time, verifies a strong form of quantitative non-integrability due to the contact structure. Thus, to be able to extract more information from the existence of $\ell$-good stable and unstable charts, we will show that there is some compatibility between these charts assuming that the QNI condition is not verified. 

\begin{defi}[Compatible charts]\label{def.compatible}
For a partially hyperbolic measure $\mu$ we say that it admits $\ell$-\emph{compatible good charts} if there exist some $L \geq \ell$,  $L$-good stable charts $\{\imath_x\}_{x \in M}$ and $L$-good unstable charts $\{\imath'_x\}_{x\in M}$ such that for $\mu$ almost every $x\in M$ we have that: for all $(t_1,t_3)$ close to $(0,0)$, 
\begin{equation}\label{eq:compatible}
(\imath'_x)^{-1} \circ \imath_x (t_1, 0 , t_3) = ( s_1 , O(  (|s_1| + |s_3|)^\ell),  s_3). 
\end{equation} 
We say that $\mu$ admits \emph{compatible good charts} if it admits $\ell$-compatible good charts for every $\ell>0$. 
\end{defi}

The existence of compatible good charts implies that the measure is \emph{jointly integrable up to order $\ell$}, a notion defined below:  
\begin{defi}\label{defi.jointorderell-meas}
We say that a partially hyperbolic measure $\mu$ is \emph{jointly integrable up to order $\ell$} if there is a measurable family of $C^{\ell}$ smooth surfaces (with boundaries) $\{\cS_x \}_{x \in M}$ in $M$ such that for $\mu$ almost every $x \in M$, there is $\rho_x > 0$ such that: 
\begin{enumerate}
\item $W^1_{\rho_x}(x) \cup W^3_{\rho_x}(x) \subset \cS_x$, and, 
\item for $\mu^1_x$ almost every $y \in W^1_{\rho_x}(x)$ (resp. $\mu^3_x$ almost every $y$ in $W^3_{\rho_x}(x)$) we have that $W^3_1(y)$ is tangent  to $\cS_x$ to order $\ell$ at $y$ (resp. $W^1_1(y)$ is tangent to $\cS_x$  to order $\ell$ at $y$). 
\end{enumerate}
\end{defi}

It is natural to compare the above definition with Definition \ref{defi.jointorderell}.

\begin{prop}\label{prop.compatibleimpliesjoint}
Let $\mu$ be a partially hyperbolic measure with compatible good charts. Then $\mu$ is jointly integrable up to order $\ell$ for every $\ell>0$. 
\end{prop}

\begin{proof} 
We fix an arbitrary integer $\ell > 0$. 
By hypothesis, there exist an integer $L > 10 \ell$, a collection of $L$-good stable charts  $\{\imath_x\}_{x \in M}$,  and a collection of $L$-good unstable charts  $\{\imath'_x\}_{x\in M}$  such that \eqref{eq:compatible} is satisfied for $10 \ell$ in place of $\ell$. Let $x$ be a $\mu$-typical point such that $\imath_x$ and $\imath'_x$ are defined at $x$, and \eqref{eq:compatible} holds.

Let $\rho_x > 0$ be a small constant such that for all $(t_1, t_3) \in (-\rho_x, \rho_x)^2$, we may write
\aryst
(\imath'_x)^{-1} \circ \imath_x(t_1, 0, t_3) = (h^x_1(t_1, t_3), h^x_2(t_1, t_3), h^x_3(t_1, t_3)).
\earyst
Here $h^x_1, h^x_2, h^x_3$ are smooth functions. 
We have $h^x_1(t_1, 0) = t_1$ and $h^x_3(0, t_3) = t_3$.
Then we have $| h^x_1(t_1, t_3)|, |h^x_3(t_1, t_3)| \leq C_x( |t_1| + |t_3| )$ for some $C_x > 0$. 
By \eqref{eq:compatible}, by enlarging $C_x$, and by letting $\rho_x$ be small  if necessary, we have
\ary \label{eq h2bound}
 | h^x_2(t_1, t_3) | \leq C_x( |t_1|^{10 \ell} + |t_3|^{10 \ell}).
\eary

We will  construct some function $\varphi_x : \R^2 \to \R$ whose derivatives along the axes have desired properties. The following statement of Whitney's extension theorem (see \cite{W}) is \cite[Theorem 2.3.6]{Ho} (in dimension $2$):
\begin{thm} \label{thm WET}
	Let $E$ be a compact set in $\R^2$ and let $u_{i, j}$ be a continuous functions on $E$ for any $i,j \geq 0$ with  $i + j \leq k$.
	If the function $U_{i, j} : E \times E \to \R$ given by   $U_{i, j} = 0$ on the diagonal of $E \times E$, and given by
	\aryst
		U_{i, j}((t_1, t_3), (s_1, s_3)) :=  u_{i, j}(t_1, t_3) -  \sum_{i' + j' \leq \ell - i - j} \frac{u_{i + i', j + j'}(s_1, s_3)}{(i')! (j')!}(t_1 - s_1)^{i'}(t_3 - s_3)^{j'} 
	\earyst
	for distinct $(t_1, t_3), (s_1, s_3)$ on $E$, 
	 is continuous, then there exists $v \in C^k(\R^2)$ with $\partial_{t_1}^{i} \partial_{t_3}^{j} v = u_{i, j}$ for any $i,j \geq 0$ with $i + j \leq k$, and satisfies that
	\aryst
	\sum_{i, j \geq 0, i+ j \leq k}  \| \partial^i_{t_1} \partial^j_{t_3} v \| \leq C (\sum_{i, j \geq 0, i+ j \leq k} \sup_{K \times K}  | U_{i, j} | + \sum_{i, j \geq 0, i+j \leq k} \sup_K | u_{i, j} | ).
	\earyst
\end{thm}

 We set $E = [ - \rho_x, \rho_x ] \times \{0\} \cup \{0\} \times  [ -\rho_x, \rho_x ]$.
 For any integers $i, j \geq 0$ such that $i + j \leq \ell$, we define
 \ary \label{eq varphiijx}
 u^x_{i, j}(t_1, t_3) = 
 \begin{cases}
\partial_{t_1}^i \partial_{t_3}^j h^x_2(0, t_3),  & t_1 = 0, t_3 \in [ -\rho_x, \rho_x ] \setminus \{0\}, \\
0, & t_3 = 0, t_1 \in [ -\rho_x, \rho_x ].
 \end{cases}
 \eary
 By  \eqref{eq h2bound}, the above formula gives a collection of continuous functions on $E$.
 Moreover, by Taylor's expansion of $h_2$ at the origin, we see that for the above $i,j$ and some $C'_x > 0$
 \ary \label{eq boundforvarphiij}
 | u^x_{i,j}(t_1, t_3)| \leq  C'_x (|t_1| + |t_3|)^{\ell + 1}.
 \eary 
We claim that  for any $i, j$ as above, for any $(t_1, t_3)$ and $(s_1, s_3)$ on $E$, we have
 \ary 
 	U^x_{i, j}((t_1, t_3), (s_1, s_3))  =  O((| s_1 - t_1 | + | s_3 - t_3 |)^{\ell + 1 - i - j}) \label{eq varphiijt1t3} 
 \eary
 where $U^x_{i, j}$ is given by the expression of $U_{i, j}$ in Theorem \ref{thm WET} for $u^x_{i,j}$ in place of $u_{i, j}$.
    To prove the claim, it suffices to assume either $t_1 = s_3 = 0$ or $t_3 = s_1 = 0$, for otherwise the equality is either trivial or follows from the smoothness of $h^x_2$. Note that in both cases we have $|t_1|+|s_1| = |t_1-s_1|$ and $|t_3| + |s_3| = |t_3 - s_3|$. Then the claim follows immediately from \eqref{eq boundforvarphiij}.

 Now we may apply Theorem \ref{thm WET} to obtain a $C^{\ell}$ function $\tilde\varphi_x$ defined on $\R^2$ such that $\partial_{t_1}^i \partial_{t_3}^j \tilde\varphi_x|_{E} = u^x_{i,j}$  for any integers $i, j  \geq 0$ such that $i + j \leq \ell$. 
 We claim that this extension can be chosen to depend measurably on $x$ on a full measure set. Indeed, by Lusin's lemma we can take an increasing sequence $( \Omega_n \subset M )_{n \geq 0}$ so that their union is a full measure set, and for each $n$ the coordinate charts  $\imath_x$ and $\imath'_x$ depend continuously on $x$; $\rho_x$ is uniformly lower bounded; and $C'_x$ in \eqref{eq boundforvarphiij} is uniformly upper bounded, for $x$ restricted to $\Omega_n$.  As a result, the map $h^x_2$ depends continuously on $x$ restricted to $\Omega_n$.  By construction, the  collection of  functions $u^x_{i, j}$ given in \eqref{eq varphiijx} depend continuously on $x$ restricted to $\Omega_n$. Consequently, for any sequence $(x_m)_{m \geq 0}$ in $\Omega_n$ converging to $x \in \Omega_n$,  we have $ \lim_{m \to \infty} \sup_{K}| u^{x_m}_{i, j} - u^x_{i, j}| = 0$.
 Moreover, using the fact that  the implicit constant in \eqref{eq varphiijt1t3} is uniform, we may deduce that $ \lim_{m \to \infty} \sup_{K \times K} | U^{x_m}_{i, j} - U^{x}_{i, j} | = 0$.  By Theorem \ref{thm WET}, we see that $ \lim_{m \to \infty} \| \tilde \varphi_{x_m} - \tilde \varphi_{x} \|_{C^{\ell}(E)} = 0$ where $\| \varphi \|_{C^{\ell}(E)}$ is understood as the infimum of $\| F  \|_{C^\ell}$ over the collection of function $F$ extending $f$ (see \cite{Fe}).
 At this point, we can apply the main result in \cite{Fe}  to construct $\varphi_x$ so that it depends continuously on $x$ restricted to $\Omega_n$.\footnote{Alternatively, we may argue by following the proof in \cite{Ho} to see that the construction of $\tilde\varphi_x$ can be made linear in the data $(\varphi^x_{i, j})_{i, j \geq 0, i+j \leq k}$.}   As $n$ is arbitrary, we conclude the proof of the claim.


We define $H(t_1, t_2, t_3) = (t_1, t_2 + \varphi(t_1, t_3), t_3)$, and define a map $\imath^{\flat}_x : (-  \rho_x,   \rho_x)^3 \to M$ by $\imath^{\flat}_x = \imath'_x \circ H$.
We denote $T_0 = \{  (t_1, t_2, t_3)  : t_2 = 0 \}$ and define $\cS_x = (\imath^{\flat}_{x})^{-1}(T_0)$.
Then it is straightforward to see that Definition \ref{defi.jointorderell-meas}(i), (ii) are satisfied at $x$ for $\cS_x$. Consequently $\mu$ is jointly integrable up to order $\ell$.
\end{proof}

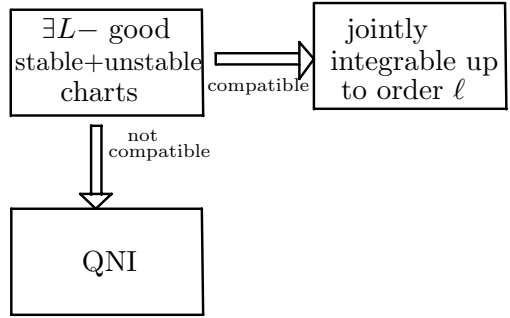
\begin{figure}[ht]
\begin{center}
\includegraphics[scale=0.91]{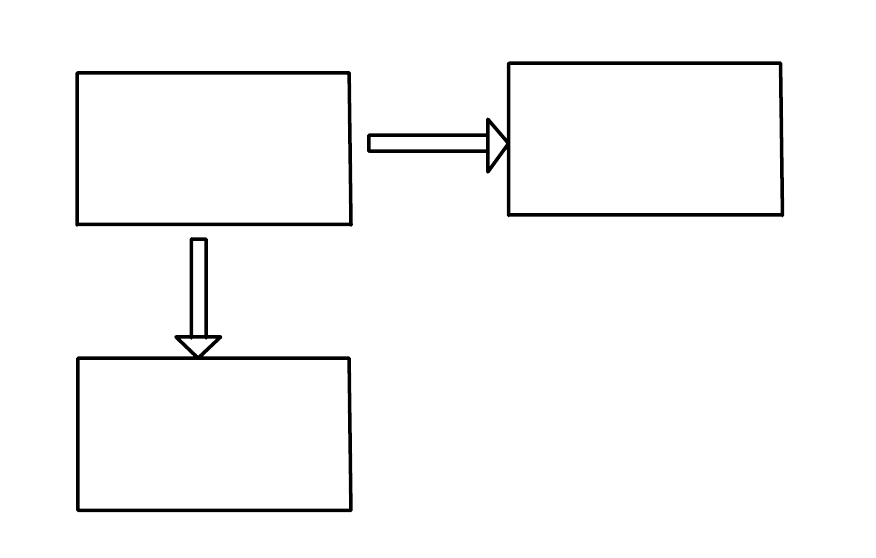}
\begin{picture}(0,0)
\put(-202,117){$\exists L-$ \text{good}} 
\put(-212,105){{\small \text{stable}+\text{unstable}}} 
\put(-195,93){\text{charts}} 
\put(-88,117){\text{jointly}}
\put(-94,105){\text{integrable up}}
\put(-90,94){\text{to order }$\ell$}
\put(-170,78){{\tiny \text{not}}}
\put(-140,97){{\tiny \text{compatible}}}
\put(-177,72){{\tiny \text{compatible}}}
\put(-187,29){\text{QNI}}

\end{picture}
\end{center}
\vspace{-0.5cm}
\caption{{\small Schematic illustration of the way QNI is obtained if the measure is not jointly integrable up to high order.}}\label{fig-flowchart2}
\end{figure}

\subsection{Main technical statement} 

We have the following dichotomy, which proposes a more geometric way to deal with the QNI condition (at least when the diffeomorphism is sufficiently smooth).

\begin{teo}\label{teo.maintechnical}
Let $\mu$ be a  non-degenerate  partially hyperbolic measure for a $C^\infty$ smooth diffeomorphism $f$ of a closed 3-manifold. Then, $\mu$ has the QNI property if and only if it does not admit compatible good charts (cf. Definition \ref{def.compatible}). 
\end{teo}

It is easy to check that if $\mu$ admits compatible good charts then it cannot verify QNI, so the main point of Theorem \ref{teo.maintechnical} is to establish that if $\mu$ does not admit compatible good charts, then it has to have the QNI property.
 We divide the proof into two natural steps. The first was illustrated in  Figure \ref{fig-flowchart1} and a more detailed scheme can be found in Figure \ref{fig-flowchart3}. 

\begin{teo}\label{teo.maintechnical1}
Let $\mu$ be a non-degenerate partially hyperbolic measure and assume that it does not admit $\ell$-good unstable charts for some integer $\ell>0$, then $\mu$ has the QNI property. 
\end{teo}
The proof of this theorem will be given in \S\ref{s.computations}. 
The symmetric statement holds for the existence of $\ell$-good stable charts (cf. Proposition \ref{prop.symmetry}).

The second part is to show that the $\ell$-good charts must be compatible unless QNI holds (see Figure \ref{fig-flowchart2}): 

\begin{teo}\label{teo.maintechnical2}
 Let $\ell \geq 1$ and let $\mu$ be a  non-degenerate  partially hyperbolic measure. Then there is an integer $L > 0$ such that if $\mu$  admits $L$-good stable charts and $L$-good unstable charts and $\mu$ does not have QNI then there is a family of $\ell$-compatible good charts.
\end{teo}

This will be shown in \S\ref{s.compatible}. In \S\ref{ss.uniformversions} we discuss and prove some uniform versions of these results when the diffeomorphism is (uniformly) partially hyperbolic. 
\subsection{Applications} 

We restate here a consequence of the main result of \cite{Katz}.  

\begin{theorem}[Katz \cite{Katz}] \label{thm Katz}
Assume that $\mu$ is an ergodic partially hyperbolic measure with $\chi_2>0$ which is a $uu$-state and verifies the QNI property. Then, $\mu$ is SRB. 
\end{theorem}

 The measure $\mu$ in Theorem \ref{thm Katz} is  clearly non-degenerate by Ledrappier-Young's entropy formula (cf. Remark \ref{rem.nondeg}). 
It is worth pointing out that in \cite{Katz} the flow case is treated. Note that for diffeomorphisms one can take the suspension flow and the results from \cite{Katz} will apply and thus give the statement we just quoted. 

Therefore, our main technical statement has the following consequence: 

\begin{cor}\label{cor-dich}
Assume that an ergodic partially hyperbolic measure $\mu$  is a $uu$-state with $\chi_2>0$, then either $\mu$ is SRB or  $\mu$ is jointly integrable up to order $\ell$ for every $\ell>0$. 
\end{cor}

\subsection{Uniform versions of the results}\label{ss.uniformversions}
We come back to the context of the introduction. 

\begin{remark}\label{rem.nf}
We will use normal form coordinates for points in $\Lambda$ as in \S\ref{ss.nf}. In the setting of one dimensional stable and unstable manifolds we are working on, it turns out that one can choose the normal form coordinates to vary continuously on the point $x \in \Lambda$. This will be relevant for our statements, and in higher dimensions presents a challenge to generalize our results. See \cite{KS} and references therein. 
\end{remark}

The results announced in the introduction are not a direct consequence of their measurable counterparts stated in the previous subsections due to the fact that the estimates are measurable functions instead of continuous ones. In order to obtain the continuous version, it is just necessary to check that the arguments in the measurable case do give uniform estimates when necessary since there is a continuous invariant splitting to start with. We will explain this in \S\ref{s.uniformversions} (pointing out how the arguments simplify in some places for the uniform case). Here we will provide the corresponding definitions and main statements for the convenience of the reader. 

We will consider a continuous orientation on $E^i$ up to finite cover\footnote{\label{footnote7} Note that if $\Lambda \subset M$ is not everything, there many not be a finite cover of $M$ that orients the bundles (e.g. the Plykin attractor). However, we are only interested on the dynamics in a neighborhood of $\Lambda$ and one can always find a finite cover of such neighborhood with the desired properties. Note that this is just a notational issue, to avoid having to add $\pm$ signs in each equation.} and the induced unit vector fields $e^i(x)$. For $x \in \Lambda$ and $i \in  \{1,2, 3\}$, we define $ \lambda_x^i \in \pm \|D_xf|_{E^i(x)}\|$ by equation:
\begin{equation}\label{eq:exponentpointuni}
 D_xf(  e^i(x) ) =  \lambda_x^i e^i(f(x)),
\end{equation}
\noindent where $E^1 = E^u$, $E^2=E^c$ and $E^3=E^s$. We know by the choice of the Riemannian metric that these are continuous functions which verify that $|\lambda_x^1 |> |\lambda_x^2| > |\lambda_x^3|$ as well as $|\lambda_x^1|>1>|\lambda_x^3|$. We consider the laminations $\cW^1, \cW^3$ tangent respectively to $E^1$ and $E^3$ given by the stable manifold theorem with their corresponding normal form coordinates (cf. Remark \ref{rem.nf} or Proposition \ref{prop-normalforms-uniform} below).

The notion of quantitative non-joint integrability which one obtains in the uniform case is also a bit stronger due to the uniform assumptions.

We have the following result: 

\begin{teo}\label{teo.uniformversion}
Let $\Lambda$ be a partially hyperbolic set of a smooth diffeomorphism $f$. Then, if there is a fully supported  non-degenerate measure  $\mu$ on $\Lambda$ which does not have the QNI property, then the set $\Lambda$ is jointly integrable up to order $\ell$ for every $\ell>0$ (cf. Definition \ref{defi.jointorderell}). 
\end{teo}

One consequence of this theorem is that having one measure without QNI forces every measure with the same support to have this property:

\begin{defi}\label{defi.topoQNI}
We say that a partially hyperbolic set $\Lambda$ has \emph{topological QNI} if every measure which is fully supported on $\Lambda$ has QNI.  
\end{defi}

Theorem \ref{teo.uniformversion} implies that either $\Lambda$ has topological QNI, or every fully supported measure is degenerate, or $\Lambda$ is jointly integrable up to order $\ell$ for every $\ell$. The second case happens for instance when the set $\Lambda$ is contained in a normally hyperbolic surface tangent to $E^u \oplus E^c$.   

\begin{remark}
We note that it has been proved in \cite{BC} that if $\Lambda$ has no \emph{strong connections} (i.e. for every $x \in \Lambda$ we have that $W^s_{loc}(x) \cap \Lambda = \{x\}$) then  $\Lambda$ is contained in a locally invariant surface, that is, there is a compact surface with boundary $S$ containing $\Lambda$ in its interior and an open neighborhood $U$ of $\Lambda$ in $S$ such that $f(U) \subset S$. In this case, every fully supported measure in $\Lambda$ is degenerate. 
\end{remark}


\begin{figure}[ht]
\begin{center}
\includegraphics[scale=0.71]{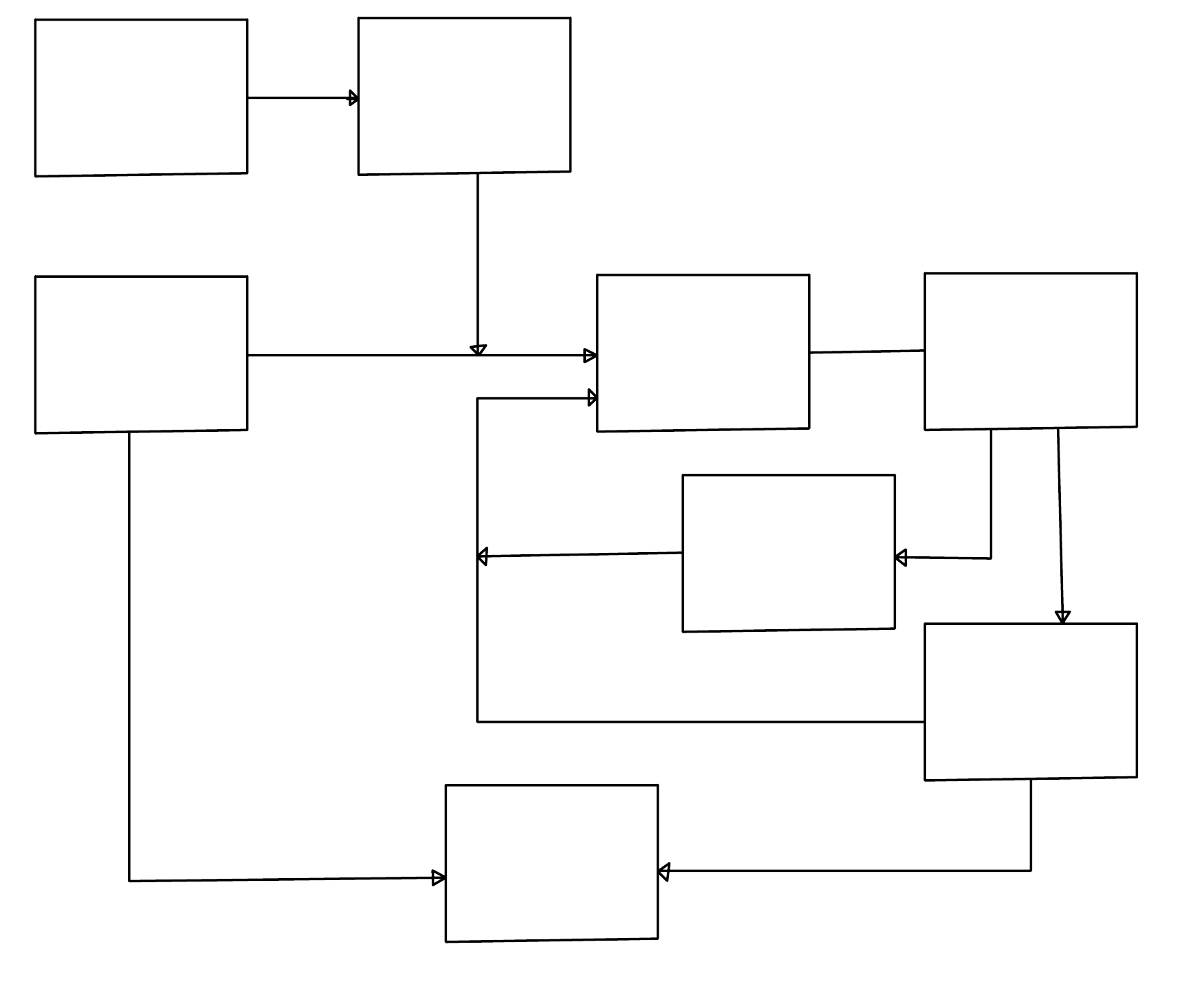}
\begin{picture}(0,0)
\put(-164,292){{\tiny \text{Cocycle NF of}}} 
\put(-162,282){{\tiny $Df$ \text{acting on}}}
\put(-167,272){{\tiny $TM/_{E^1}$ \text{over} $W^1$}} 
\put(-158,262){{\tiny Example \ref{ex.derivative}}} 
\put(-167,215){{\tiny $\cT^0_x$ \text{invariant sect.}}} 
\put(-162,205){{\tiny \text{associated to}}}
\put(-163,195){{\tiny $(E^1\oplus E^3)/_{E^1}$}} 
\put(-164,184){{\tiny \text{Equation} \eqref{eq:highordertemplate}}} 
\put(-95,282){{\tiny \S~\ref{s.cocyclenormalforms}}}
\put(-65,292){{\tiny \text{Polynomial}}} 
\put(-68,282){{\tiny \text{transformation}}}
\put(-64,272){{\tiny \text{law for coord.}}}
\put(-67,262){{\tiny \text{Equation} \eqref{eq:dynamicstemplate}}}
\put(-98,213){{\tiny \text{close to rational}}}
\put(-98,203){{\tiny \text{at some scale}}}
\put(-98,192){{\tiny \text{Proposition} \ref{p.dichotomytech}}}
\put(-28,203){{\tiny $\ell=1$}}
\put(-31,180){{\tiny \text{add} $+1$}}
\put(-28,172){{\tiny \text{to }$\ell$}}
\put(-120,48){{\tiny \text{far from rational}}}
\put(-107,36){{\tiny \S \ref{s.rational} \& \S \ref{s.computations}}}
\put(30,50){{\tiny \text{far from rational}}}
\put(43,38){{\tiny \S \ref{s.rational} \& \S \ref{s.computations}}}
\put(-20,95){{\tiny \text{close to rational at some scale}}}
\put(-5,82){{\tiny \text{Proposition} \ref{p.dichotomytech}}}
\put(-29,44){{\Large \text{QNI}}}
\put(3,215){{\tiny \text{Poly. change of}}} 
\put(8,205){{\tiny \text{coord. to take}}}
 \put(0,195){{\tiny $\cT^{\ell-1}_x$ \text{to zero sect.}}}
 \put(1,185){{\tiny \text{Proposition} \ref{prop.cocycle}}}
 \put(75,205){{\tiny \S \ref{s.cocyclenormalforms}}}
 \put(108,215){{\tiny \text{Cocycle NF}}} 
\put(110,205){{\tiny \text{for} $(\ell+1)$}}
 \put(107,195){{\tiny \text{jets over} $W^1$}}
  \put(103,185){{\tiny \text{Proposition} \ref{p.normalform}}}
   \put(140,165){{\tiny \text{construct}}} 
    \put(140,155){{\tiny \text{good}}}
    \put(140,145){{\tiny \text{charts}}}
\put(140,135){{\tiny \text{Prop.} \ref{p.varphipolyell+1}}}     
\put(99,112){{\tiny $\cT^\ell_x$ \text{invariant sect.}}} 
\put(103,102){{\tiny \text{associated to}}} 
\put(97,92){{\tiny $(\ell+1)$-\text{jets of} $W^3$ }} 
\put(101,82){{\tiny \text{Equation} \eqref{eq:highordertemplate}}}
\put(35,155){{\tiny \text{Polynomial}}}
\put(30,145){{\tiny \text{transf. law for}}}
\put(35,135){{\tiny $(\ell+1)$-\text{jets}}}
\put(30,125){{\tiny  \text{Equation} \eqref{eq:dynamicstemplate}}}

\end{picture}
\end{center}
\vspace{-0.5cm}
\caption{{\small Schematic illustration of the proof of Theorem \ref{teo.maintechnical1}}}\label{fig-flowchart3}
\end{figure}
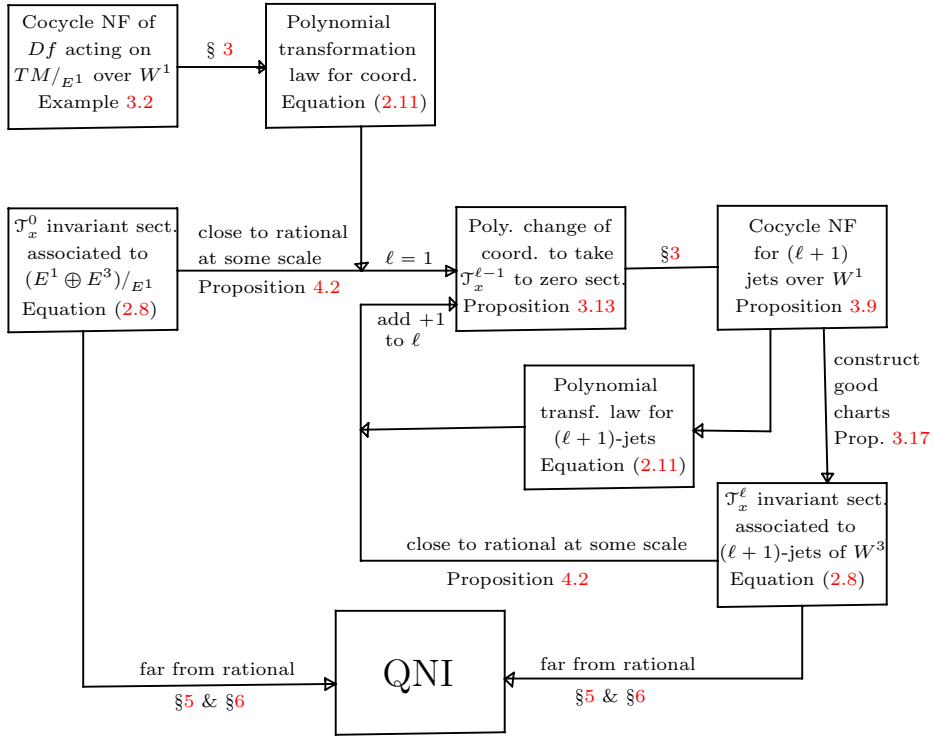

%
\section{Existence of normal coordinate charts and cocycle normal forms}\label{s.cocyclenormalforms}
In this section we prove Proposition \ref{p.existnormalcoord} and Proposition \ref{p.normalform}.  We restate some results which are done in more generality in \cite[Appendix A]{BEFRH} but in a somewhat different form. 

\subsection{Cocycle normal forms}\label{ss.CNF}

Let $f: M \to M$ be a smooth diffeomorphism preserving a partially hyperbolic measure $\mu$. We let $\cE \to M$ be a (measurable) vector bundle over $(M,\mu)$ and $A: \cE \to \cE$ a vector bundle automorphism that lifts $(f,\mu)$ (i.e. for $\mu$-a.e. $x \in M$ we have that $A_x:=A|_{\cE_x} : \cE_x \to \cE_{f(x)}$ is a linear automorphism). 

We will be concerned only with two dimensional vector bundles (i.e. $\mathrm{dim}(\cE_x) = 2$ for $\mu$-a.e. $x \in M$). We refer to \cite[Appendix A]{BEFRH} for more general results. 

We say that a bundle $\cE$ is \emph{smooth along unstable manifolds} if for $\mu$-a.e. $x$, the restriction of $\cE$ to $\cW^1(x)$  is smooth.
In this case,  a \emph{smooth trivialization along unstable manifolds} of $\cE$ is a family of pairs  $\cY = \{ \cY_x= ( \xi_x,\xi^\perp_x ) \}_{x \in M}$  such that for $\mu$-a.e. $x \in M$, $\xi_x, \xi_x^{\perp}: (-\|Df\|, \|Df\|) \to \cE$ are smooth maps such that $\xi_x(t), \xi^\perp_x(t)$ are linearly independent vectors in $\cE_{\Phi^1_x(t)}$. 

 \begin{remark}\label{rem.smoothunstables}
We note that for a partially hyperbolic measure $\mu$ almost every point has a well defined strong unstable manifold, however, not every point in this manifold is generic with respect to $\mu$.
\end{remark}

We start by presenting an example which corresponds to the first step of our induction.

\begin{ej}\label{ex.derivative}
Consider the two dimensional vector bundle $\cE \to M$ defined for $\mu$-a.e. $x\in M$ as the quotient $\cE_x = T_xM/_{E^1(x)}$.  We fix a non-degenerate inner product on $\cE_x$ on each $x$ which we denote by  $\langle \cdot, \cdot \rangle_{\cE}$ and which is smooth along unstable manifolds $W^1_1$. 
Clearly, since the cocycle $Df$ preserves $E^1$, it induces a vector bundle automorphism $A$ given by:
$$A [v] = [D_xf (v)] = D_{x} f( v ) + E^{1}(f(x)) \in T_{f(x)} M /_{E^{1}(f(x))}$$
  where $v \in T_xM$ and $[v] \in T_xM/_{E^1(x)}$ denotes $v + E^1(x)$ the representative of $v$ in the quotient.

We note that for $\mu$-a.e. $x \in M$, the restriction $\cE_x$ of the bundle $\cE$ to $W^1_1(x)$ is a smooth vector bundle because the local unstable manifold is smooth.

We choose a trivialization of $\cE$ as follows.
We choose a smooth map  $\xi_x: (-\|Df\|, \|Df\|) \to \cE$ such that 
\aryst
\Phi^1_x(t) \mapsto \xi_x(t) \in (E^{1}(\Phi^1_x(t)) \oplus E^2(\Phi^1_x(t)))/_{E^1(\Phi^1_x(t))}
\earyst
 is a section of the bundle $\cE \to M$.  The existence of such smooth map $\xi_x$ is guaranteed by the fact that the weak-unstable bundle $E^1 \oplus E^2$ of the Oseledets decomposition is smooth along strong unstable manifolds (see Proposition \ref{p.oseledetssmooth}). 
We let $\xi^\perp_x: (-\|Df\|, \|Df\|) \to \cE$ be a smooth map such that 
$\xi^\perp_x(t)$ is a unit vector in $\cE_x$, and $\langle \xi^\perp_x(t),   \xi_x(t) \rangle_{\cE} = 0$.

 In this way, if we set $\cY_x= ( \xi_x, \xi_x^\perp )$, then we can write the matrix corresponding to the action of $Df$ from $\cE_{\Phi^1_x(t)}$ to $\cE_{\Phi^1_{f(x)}(\lambda_{1,x}t)}$ in the basis $\cY_x(t)= ( \xi_x(t), \xi_x^\perp(t) )$ and $\cY_{f(x)}(\lambda_{1,x}t) = (  \xi_{f(x)}(\lambda_{1,x} t), \xi_{f(x)}^\perp(\lambda_{1,x} t) )$, for $t \in (-1,1)$,  as
$$A^\cY (t)  := \begin{pmatrix} \alpha_x(t) & r_x(t) \\  0 & \beta_x(t) \end{pmatrix} $$
\noindent where $\alpha_x, r_x, \beta_x$ are smooth functions.   \finobs \end{ej}

We can write the vector bundle automorphism $A$ in a smooth trivialization as a measurable function
\aryst
A^\cY : M \times (-1,1) \to \mathrm{GL}_2(\RR) 
\earyst
such that $A^\cY(x, t)$ denotes the matrix associated to $A_{\Phi^1_x(t)}$ from the ordered basis $( \xi_x(t), \xi^\perp_x(t) )$ to the ordered basis $( \xi_{f(x)}(\lambda_{1,x} t), \xi_{f(x)}^\perp (\lambda_{1,x}t) )$. 

We say that the vector bundle automorphism $A$ is \emph{smooth along unstable manifolds} if there is a smooth trivialization of $\cE$ such that for $\mu$-a.e. $x\in M$ the entries of $A^\cY(x, \cdot)$ are smooth functions of $t$. Note that if there is one such trivialization, the same holds for all smooth trivializations.

The following is the main result from normal forms for cocycles which are smooth along unstable manifolds that we will need. We refer the reader to Appendix \ref{ss.apcocycles} for more discussion on the notions of smoothness along strong unstables and exponents of cocycles. 

\begin{prop}\label{p.CNF} 
Let $\mu$ be a partially hyperbolic measure of a diffeomorphism $f: M \to M$, let $\cE \to M$ be a $\mu$-measurable two dimensional vector bundle which is smooth along unstable manifolds, and let $A$ be a linear cocycle over $(f,\mu)$ which is smooth along unstable manifolds and has exponents $\alpha > \beta$.  Then there exists $\cY = \{\cY_x= (\xi_x, \xi^\perp_x ) \}_{x \in M}$, a smooth trivialization along unstable manifolds of $\cE$, satisfying that for $\mu$-a.e. $x\in M$ we have that for $t\in (-1,1)$
\begin{equation}\label{eq:CNF}
 A^\cY (x, t) = \begin{pmatrix} \alpha_x & p_x(t) \\ 0 & \beta_x \end{pmatrix}
\end{equation}
\noindent where $\alpha_x$ and $\beta_x$ depend measurably on $x$ such that $\int \log \alpha_x d\mu(x)= \alpha$ and $\int \log \beta_x d\mu(x) = \beta$ and $p_x: (-1,1) \to M$ is a polynomial of degree at most $d$, where $d$  depends only on $\chi_1, \alpha, \beta$.
\end{prop}

\begin{proof}
Being smooth along unstable manifolds, we can define a  measurable non-degenerate inner product $\| \cdot \|_{\cE}$ on the fibers of $\cE$ which is smooth along unstable manifolds.  

By Proposition \ref{p.oseledetssmooth},
there exists a family of smooth trivializations $\cY = \{ \cY_{0,x}= (\xi_{0,x}, \xi_{0,x}^\perp) \}_{x \in M}$ such that the cocycle $A^{\cY_0}$ is upper triangular along unstable manifolds. That is, there are smooth functions $\alpha_x, \beta_x, r_x: (-1,1) \to \RR$ such that:  
\begin{equation}\label{eq:CNF0}
 A^{\cY_0} (x, t) = \begin{pmatrix} \alpha_x(t) & r^0_x(t) \\ 0 & \beta_x(t) \end{pmatrix}.
\end{equation} 
We remark that, by uniqueness, $\xi_{0,x}(t)$ belongs to the  Oseledets bundle associated to the exponent $\alpha$ in the fiber $\cE_{\Phi^1_x(t)}$ (whenever it is defined\footnote{Note that actually, the bundle associated to $\alpha$ is defined on backward regular points, so it would make sense to say it is well defined for all $t$, though we will not use this fact in the proof.}). 

We now make the diagonals to be constant. 

\begin{claim}\label{claim-firstcoordinate}  
There exist a measurable family of smooth functions $\{a_x: (-1,1)\to \RR_{>0} \}_{x \in M}$, and a   measurable  family of smooth functions $\{a^\perp_x : (-1,1) \to \R_{>0} \}_{x \in M}$ such that the following is true. 
Let $\xi_{1,x}(t) = a_x(t) \xi_{0,x}(t)$ and $\xi_{1,x}^\perp(t) = a^\perp_x(t) \xi_{0,x}^\perp(t)$. Then $\cY_1 = \{\cY_{1,x}= ( \xi_{1,x}, \xi_{1,x}^\perp) \}_{x\in M}$ is  a smooth trivialization along unstable manifolds of $\cE$ such that
\begin{equation}\label{eq:CNF1}
	A^{\cY_1} (x, t) = \begin{pmatrix} \alpha_x(0) & r_x(t) \\ 0 & \beta_x(0) \end{pmatrix}
\end{equation}
where $\{ r_x: (-1,1) \to \RR \}_{x \in M}$ is a measurable family of smooth functions. Moreover the choice of $\{ a_x \}_{x \in M}$ is unique if we require $a_x(0)=1$ in addition.
\end{claim}

\begin{proof}

The proof is similar to that of the Stable Manifold Theorem. Let us spell out the computations. 

First we construct $\xi_{1,x}$ from $\xi_{0,x}$. We can put coordinates $(t,s)_x$ on the one-dimensional bundle $\RR \xi_{0,x}$ so that $(t,s)_x$ represents the vector $v:= s \xi_{0,x}(t) \in \cE_{\Phi^1_x(t)}$. This way, we can write, for $t \in (- 1,  1)$ and $s \in \RR$: 
$$ \Psi_x (t,s)_x:=  (\lambda_{1,x}t , \alpha_x(t) s)_{f(x)}$$

\noindent that corresponds to the action of $A$ on vectors in the chosen coordinates (from now on we will remove the subindex of the point where the coordinates are chosen in the notation). 

Write $\alpha_x(t) = \pm \mathrm{exp}(b_x(t))$ for some positive smooth function $b_x: (-1,1)\to \RR$ (we will assume from now on that $\alpha_x(t)$ is positive for simplicity).   

We need to find a family of smooth functions $\{  c_x: (-1,1) \to \RR \}_{x \in M}$ so that $c_x(0)=0$ with the property that $\Psi_x(t, \exp(c_x(t))) = (\lambda_{1,x} t, \alpha_x(0) \exp(c_{f(x)}(\lambda_{1,x}t)))$ which we can write as: 
\begin{equation}\label{eq:fiberbunch}
c_{f(x)}(\lambda_{1,x}t ) +  b_x(0) = b_x (t) + c_x(t).
\end{equation}
This holds for almost every $x \in M$, so we can solve $c_x$ as follows: Denote $b_m= b_{f^{-m}(x)}$, $c_m = c_{f^{-m}(x)}$,  $\lambda^-_{1,m}  = \|D_{f^{-m}(x)}f^{m}|_{E^1}\|^{-1} = (\lambda_{1, f^{-1}(x)}^1 \cdots \lambda_{1, f^{-m}(x)}^1)^{-1}$. Then we get that for every $k \in \ZZ_{>0}$ that
\begin{equation}\label{eq:convergence} 
c_{x}(t) = c_{k}(\lambda^-_{1,k} t) + \sum_{j=1}^{k} \left(b_{j}(\lambda^-_{1,j} t) - b_{j}(0) \right). 
\end{equation}
Since $\lambda^-_{1,j}$ tends to $0$ exponentially fast and $c_k$ is smooth satisfying $c_k (0)= 0$, we get that $c_k(\lambda^-_{1,k}t)$ tends to $0$ exponentially fast for a $\mu$-typical $x$ by Birkhoff's ergodic theorem. Similarly, we have that the value of $b_j(\lambda^-_{1,j}t) -b_{j}(0)$ is also exponentially small so that the sum converges uniformly in $t$. Thus for a typical $x$ we have
\aryst
c_{x}(t) = \sum_{j=1}^{\infty} \left(b_{j}(\lambda^-_{1,j} t) - b_{j}(0) \right). 
\earyst
It is clear that the above expression gives the unique solution of equation \eqref{eq:convergence}.
Notice also that the derivatives of $c_x$ can be solved by an analogous computation. Consequently we can show that the functions $c_x$  are $C^\infty$ smooth and derivatives of all orders vary measurably on the point $x$.

To get the family of sections $\xi_{1,x}^\perp$ one argues in the same spirit by looking at the bundle $\cE_{\Phi^1_x(t)}$ quotiented by $\RR\xi_{1,x}(t)$, the same considerations on the smoothness apply. 
\end{proof}

Finally, we will need to add some component of $\xi_{1, x}$ to $\xi_{1, x}^\perp$ in order to change the function $r_x$ in \eqref{eq:CNF1} into a polynomial. In the following let us abbreviate $\alpha_x(0)$ and $\beta_x(0)$ as $\alpha_x$ and $\beta_x$ respectively.

\begin{claim}\label{c.polinomial}
There exists a measurable family of smooth functions $\{u_x: (-1,1) \to \RR\}_{x \in M}$ such that if we take $\xi_x(t) = \xi_{1,x}$ and $\xi_x'(t) = u_x(t) \xi_x(t) + \xi_{1,x}^\perp(t)$ and $\cY = \{\cY_x = (\xi_x, \xi'_x ) \}_{x \in M}$ we get that: 
\begin{equation}\label{eq:CNF2}
 A^{\cY}(x, t) = \begin{pmatrix} \alpha_x & p_x(t) \\ 0 & \beta_x \end{pmatrix}
\end{equation}
\noindent where $p_x$ is a polynomial of degree $\leq d_0$ where\footnote{Caution, here $\alpha$ and $\beta$ denote the Lyapunov exponents of the cocycle, which are the integral of the functions $x \mapsto \log \alpha_x$ and $x \mapsto \log \beta_x$. Similarly, $\chi_1$ denotes the first Lyapunov exponent of $(f,\mu)$ and can be computed as the integral of $x \mapsto \log \lambda_{1,x}$.}  $d_0= \lfloor{\frac{\alpha- \beta}{\chi_1}}\rfloor +1$. 
\end{claim} 

\begin{proof}
Let us then compute the map $A$ in the coordinates $\cY_1$. We have
$$\xi'_x(t) =  u_x(t) \xi_x(t) + \xi^\perp_x(t) \mapsto  (r_x(t) + \alpha_x u_x(t)) \xi_{f(x)} (\lambda_{1,x}t) + \beta_x \xi^\perp_{f(x)}(\lambda_{1,x}t). $$

We can write $r_x(t) = p_x(t) + \hat r_x(t)$ where $p_x(t)$ is a polynomial of degree $\leq d_0$ and $\hat r_x(t) = O(t^{d_0+1})$. We need to solve: 
\begin{equation}\label{eq:solveu}  u_{f(x)}(\lambda_{1,x} t) = \frac{1}{\beta_x} (\hat r_x(t) + \alpha_{x} u_x(t)). 
\end{equation}

Let us then solve $u_x$ formally to see that one can only get a solution for sufficiently large $d_0$. This is why one can only get to remove $r_x(t)$ up to a polynomial of that degree. Write $u_n(t) = u_{f^{-n}(x)}(t)$, $\hat r_n(t) = \hat r_{f^{-n}(x)}(t)$, $\lambda_{1,m}^-=  \|D_{f^{-m}(x)}f^{m}|_{E^1}\|^{-1} = (\lambda_{f^{-1}(x)}^1 \cdots \lambda_{f^{-m}(x)}^1)^{-1}$, $\alpha_n= \alpha_{f^{-1}(x)} \cdots \alpha_{f^{-n}(x)}$ and $\beta_n = \beta_{f^{-1}(x)} \cdots \beta_{f^{-n}(x)}$ so that:
\begin{equation}\label{eq:solveu20} 
u_x(t) =  \frac{\alpha_k}{\beta_k} u_{k}(\lambda_{1,k}^{-} t)  + \sum_{j=1}^{k} \left( \frac{\alpha_{j-1}}{\beta_{j}} \hat r_j(\lambda_{1,j}^- t) \right). 
\end{equation} 

Note that $\frac{1}{j} \log (\lambda_{1,j}^-)^{d}$ converges to $-d \chi_1$ while $\frac{1}{j} \log \frac{\alpha_j}{\beta_j}$ converges to $\alpha-\beta$ we can then choose $d_0$ so that for every $x$ and $d\geq d_0$ we have that $(\lambda_{1,j}^{-})^d \beta_j^{-1} \alpha_{j-1}$ converges exponentially fast to $0$  (uniformly in $t$) as $j \to +\infty$. This happen as long as $d_0$ verifies that $\alpha-\beta - d_0\chi_1 < 0$. 

Using that $\hat r_j(\lambda_{1,j}^-t) = (\lambda_{1,j}^-)^{d_0+1} O(t^{d_0+1})$ we deduce that the function 
\begin{align}  \label{eq:solveu2}
u_x(t) = \sum_{j=1}^{\infty} \left( \frac{\alpha_{j-1}}{\beta_{j}} \hat r_j(\lambda_{1,j}^- t) \right)
\end{align}
is well defined and smooth along unstable manifolds.
\end{proof}

This completes the proof of the proposition. 
\end{proof}

%
%
%
%
%

Using Proposition \ref{p.oseledetssmooth.uniform} instead of Proposition \ref{p.oseledetssmooth} in the above proof, we have the following parallel statement.  We will omit its proof since it is in close parallel to that of Proposition \ref{p.CNF}.

\begin{prop}\label{p.CNF.uniform} 
	Let $f : M \to M$ be a $C^{\infty}$ smooth diffeomorphism preserving a uniform partially hyperbolic set $\Lambda$. Let $\cE \to \Lambda$ be a two-dimensional vector bundle over $\Lambda$, and let $A : \cE \to \cE$ be a bundle automorphism, both of which are smooth along the unstable manifolds. Assume that $A|_{\cE}$ admits a continuous dominated splitting $\cE = E' \oplus E''$, i.e.,  $\| A |_{E'} \|  > \| A|_{E''} \|$ pointwise\footnote{When $E'$ and $E''$ are of higher dimension, the condition for dominated splitting writes $\| A^{-1}|_{E'} \|^{-1} > \| A |_{E''} \|$.}.

Then there exists a continuous family of smooth trivializations
$\cY_0 = \{ \cY_{0,x}= (\xi_{0,x}, \xi_{0,x}^\perp) \}_{x \in \Lambda}$ such that for every $x \in \Lambda$,
\aryst
A^{\cY_0} (x, t) = \begin{pmatrix} \alpha_x & p_x(t) \\ 0 & \beta_x \end{pmatrix}
\earyst
where  $|\alpha_x | = \| A |_{E'(x)} \| $, $| \beta_x| =  \| A |_{E''(x)} \| $, and  $p_x: (-1,1) \to M$ is a polynomial of degree at most $d$, where $d$  depends  only on $f$ and $A$ (but not on $x$).

\end{prop}


\subsection{Construction of $0$-good unstable coordinate charts} 
In this subsection we complete the proof of Proposition \ref{p.existnormalcoord}. 
Let us recall the statement.

\begin{prop}\label{p.exist1goodchart} 
Every partially hyperbolic measure $\mu$ admits a family of $0$-good unstable charts. 
\end{prop}

\begin{proof}
	First we fix  a family of Pesin charts $\{\imath^0_x\}_{x}$ (which are smooth charts varying measurably) from $(-100\lambda_{1,x}, 100 \lambda_{1,x})^3 \to M$ as in \cite{BP}. These are chosen to verify: 

\begin{itemize}
\item $\imath^0_x(0,0,0)= x$, 
\item for $i \in \{1,2,3\}$, $\partial_i \imath^0_x(0,0,0) \in E^i_x$ is unit vector. 
\end{itemize}

Using the normal form coordinates, we can make a coordinate change (which we still call $\{\imath^0_x\}_x$) and further assume that the charts verify: 

\begin{itemize}
\item $\imath^0_x(t,0,0) = \Phi^1_x(t)$, 
\item $\imath^0_x(0,0,t) = \Phi^3_x(t)$.
\end{itemize}

Recall the construction in  Example \ref{ex.derivative}.
The derivative map $Df$ on $TM$ descends to a vector bundle automorphism on $\cE = TM / E^1$ over $f$. This vector bundle automorphism is clearly smooth along the unstable manifolds. 
We now write $F^0_x := (\imath^0_{f(x)})^{-1} \circ f \circ (\imath_x^0) = (F_{x,1}^0, F_{x,2}^0, F_{x,3}^0)$. 
Then  $\{\imath^0_x\}_x$ gives us a smooth trivialization of $\cE$ along unstable manifolds under which the bundle map takes form 
$$ t \mapsto \begin{pmatrix}  \partial_2 F_{x,2}^0(t,0,0) & \partial_3 F_{x,2}^0 (t,0,0) \\  \partial_2 F_{x,3}^0(t,0,0) & \partial_3 F_{x,3}^0 (t,0,0) \end{pmatrix}. $$

Now, if we apply Proposition \ref{p.CNF} to this cocycle we can find a change of coordinates of the form: 
$$ (x_1, x_2, x_3) \mapsto (x_1, x_2 + a_2(x_1)x_2 + a_3(x_1) x_3, x_3 + b_2(x_1) x_2 + b_3(x_1)x_3) $$

\noindent that produces new charts $\imath_x$ for which the conditions of $0$-good charts are verified since it takes the derivative cocycle along unstable manifolds to normal form.  
\end{proof}

\begin{remark}\label{rem.0goodchartuniform} 
As in  Proposition  \ref{p.CNF.uniform}, the $0$-good unstable chart in Proposition \ref{p.exist1goodchart} depends H\"older continuously on the base point near any predetermined point in $M$.
\end{remark}

\subsection{Two-dimensional cocycles for the $\ell$-jets of the stable manifolds} \label{subsec TwoDimCoc}


We have the following. 
\begin{prop}\label{p.normalform}
If there is a family of $(\ell-1)$-good unstable charts $\{ \imath_x \}_{x \in M}$ which moreover verify \eqref{eq:highordertemplate} for $\mu$-almost every $x \in M$ and some functions $\cT_x^{\ell}$, $a_x$, $b_x$, then:
\begin{enumerate}
\item\label{it.1propNF} the derivatives $\partial_3^k F_{x,2}(t,0,0) = 0$ for $\hat\mu^1_x$-a.e.  $t \in (-1,1)$ and $1\leq k \leq \ell$,
\item\label{it.2propNF} there is a $\ell$-good family of unstable charts, i.e. so that \eqref{eq:goodcoordinate} also holds. 
\end{enumerate}
\end{prop}

 The proof of this proposition relies on the study of cocycle normal forms. Related results have appeared in  \cite[Appendix A]{BEFRH} and \cite[Section 4]{TZ}.


In the rest of this subsection, we will assume that $(f,\mu)$ admits a family of $(\ell-1)$-good unstable charts verifying property \eqref{eq:highordertemplate} for some $\ell \geq 1$.  Our goal here is to construct a two-dimensional cocycle in order to apply Proposition \ref{p.CNF} to obtain  Proposition \ref{p.normalform}(ii).

Given a $\mu$-typical $x \in M$ and set $y=f(x)$.  
We are going to work in the $(\ell-1)$-good charts centered at $x$ and $y$. Let $F_x=(F_1,F_2,F_3)$ be $f$ written out in the $(\ell-1)$-good chart coordinates as in equation \eqref{eq:derivativenormalchart}, then we have
\begin{itemize}
\item $F_x(0,0,0) = (0,0,0)$,
\item  and
\begin{equation}\label{eq:F_23} F_x(x_1, 0, 0) = (\lambda_{1,x}  x_1, 0, 0),
\end{equation}
\item  and for $\hat\mu^1_x$-a.e. $x_1$ we have that
$$\begin{pmatrix} \frac{\partial F_2}{\partial x_2} & \frac{\partial F_2}{\partial x_3} \\[6pt] \frac{\partial F_3}{\partial x_2} & \frac{\partial F_3}{\partial x_3}  \end{pmatrix}  = \begin{pmatrix} \lambda_{2,x} & 0 \\ 0 & \lambda_{3,x} \end{pmatrix}$$ 
where derivatives in the above expression are evaluated at $p=(x_1, 0, 0)$.
\end{itemize}

Let $x_1 \in (-1,1)$ be a $\hat\mu^1_x$-typical value. In other words, $\Phi^1_x(x_1)$ is a $\mu$-typical point. In particular, we may assume that $W^3_1(\Phi^1_x(x_1))$ is defined. 
By \eqref{eq:highordertemplate}, we may define $\hat a(x_1, s), \hat b(x_1, s)$  by 
\aryst
 \imath_x^{-1}(W^3_{loc}(\Phi^1_x(x_1))) 
= \{ (x_1  +  \hat a(x_1, s), \hat b(x_1, s) s^{\ell + 1} , s) :  s \in (-1,1) \}.
\earyst

We can deduce Proposition \ref{p.normalform}(\ref{it.1propNF}) from the following lemma.
\begin{lemma}\label{lem.derivativezero}
 If a family of unstable normal coordinate charts verifies \eqref{eq:highordertemplate} then for any  $j \geq 0$ and any $0 \leq i \leq \ell$,  we have that $  \partial_1^{j} \partial_3^{i} F_2(x_1,0,0)=0$ for every $x_1 \in \supp(\hat\mu^1_x)$.
 \end{lemma}
\begin{proof}
Since $F_2(x_1, 0, 0) = 0$, then, for every $i \geq 0$, we have that $\partial_1^i F_2(x_1,0,0)= 0$ for all $i \geq 0$. Also, by \eqref{eq:highordertemplate} and the $f$-invariance of $\cW^3$ we have that for $\hat \mu_x^1$-a.e. $x_1 \in (-1,1)$, 
\aryst
F_2(x_1  + \hat a(x_1, s), \hat b(x_1, s)  s^{\ell+1}, s) = O(s^{\ell+1}).
\earyst
Here we allow the implicit constant in $O(\cdot)$ on the right hand side above to depend on $F$ and $x_1$, but of course independent of $s$.
 We deduce that for every $i \in \{1, \cdots, \ell \}$, and $\hat \mu^1_x$-almost every $x_1 \in (-1,1)$ we have
\aryst
0 &=& \partial^{i}_{s} \{F_2(x_1  + \hat a(x_1, s),  \hat b(x_1, s)  s^{\ell+1}, s)\}|_{s = 0} \\
&=& \partial^{i}_3 F_2(x_1, 0, 0).
\earyst
By our hypothesis that $\hat\mu^1_x$ is not atomic, there is a subset of $x_1$ with full $\hat\mu^1_x$-measure and no isolated points.
Then we deduce for every $i$ as above and every $j \geq 0$ and $\hat\mu^1_x$-a.e. $x_1$ that
$$ \frac{\partial^{i+j}F_{2}}{\partial x_1^j \partial x_3^i}(x_1,0,0)=0. $$
This concludes the proof.
\end{proof}

For every $x_1 \in (-1, 1)$, we consider the collection of germs of curves  of the form:  
\ary \label{eq formofcurve}
t \mapsto (x_1 +   O(t), bt^{\ell + 1} + O(t^{\ell+2}), ct + O(t^2)). 
\eary
Within this collection, we declare two curves to be equivalent if they have the same $b$-value (resp.  $c$-value) in the above expression. 
Then for each $x_1 \in (-1,1)$, we may use $(b, c)$ in $\R^2$ to parametrize the equivalence classes of smooth curves through $(x_1, 0, 0)$ in chart $\imath_x$.

We now construct a $\R^2$-bundle $\cE_{\ell}$ over a $\mu$-full measure set of $M$.

Let $x$ be a $\mu$-typical point. 
The trivial $\R^2$-bundle over $W^1_1(x)$ can be identified with $(-1,1) \times \R^2$ using $\Phi^1_x$.
We let $(x_1, (b, \hat{c})) \in (-1,1) \times \R^2$  represent the union of all the equivalence classes of curves through $(0, 0, 0)$ in  chart $\imath_x$ parametrized by some $(b, c)$ satisfying $\hat{c} = c^{\ell+1}$.  By definition, it is clear that for a fixed $x_1$, each $(b, \hat{c})$ corresponds to the union of $0$, $1$ or $2$ equivalence classes.
 
The following proposition shows that the $\R^2$-bundles over $W^1_1(x)$ and $W^1_1(y)$ given above for different $x, y$ can be naturally glued together via a smooth bundle automorphism on their intersection.
We let $\cE_{\ell}$ denote the resulting $\R^2$-bundle over a $\mu$-full measure set.

\begin{prop}\label{prop.cocycle2}
	Given a $\mu$-typical $x$  and a $\mu^1_x$-typical $y$, such that $I = (\Phi^1_x)^{-1}(W^1_1(x) \cap W^1_1(y))$ has positive $\hat\mu^1_x$-measure.  
	We denote $H =  \imath_y^{-1} \circ \imath_x$.
	Then there is a smooth one-parameter family of upper triangular matrices $\Big \{ \begin{bmatrix}  \alpha(x_1) & r(x_1) \\ 0 & \beta(x_1) \end{bmatrix}  \in GL(2, \R)  \Big \}_{x_1 \in I}$ such that for every $x_1 \in I$, $H$ maps the equivalence class of curves through $p = (x_1, 0, 0)$ in   chart $\imath_x$ parametrized by $(b, c)$ to the equivalence class of curves through $H(p)$ in   chart $\imath_y$ parametrized by 
	$(  \alpha(x_1) b +  r(x_1) c^{\ell + 1},  \beta(x_1) c)$.
\end{prop}

\begin{proof}

	We write $H = (H_1, H_2, H_3)$. By definition, there exist $a, b \in \R$ such that for every $s \in (-1,1)$, we have
	\ary \label{eq valueofH}
	H_1(s, 0, 0) = a + b s, \quad H_2(s, 0, 0) = H_3(s, 0, 0) = 0.
	\eary
	
	We note that the following statement, analogous to Lemma \ref{lem.derivativezero}, holds.
	\begin{lemma}\label{lem.derivativezero2}
		For any  $j \geq 0$ and any $0 \leq i \leq \ell$,  we have that $  \partial_1^{j} \partial_3^{i} H_2(x_1,0,0)=0$ for 
		every $x_1 \in \supp(\hat\mu^1_x)$.
	\end{lemma}
	We omit the proof of Lemma \ref{lem.derivativezero2}, which is almost identical to that of Lemma \ref{lem.derivativezero}.

	Take a curve $\gamma(t) = (x_1(t), x_2(t), x_3(t))$ of form
	\aryst
	\gamma(t) = (x_1 + O(t), bt^{\ell + 1} + O(t^{\ell + 2}), ct + O(t^2)).
	\earyst
	Then by Lemma \ref{lem.derivativezero2} and \eqref{eq valueofH}, we can write $H \circ \gamma(t) $ as
	\aryst
	&& ( a + b x_1 + O(t),  \\
	&& [  \frac{\partial H_2}{\partial x_{2}}(p) b + \frac{1}{(\ell+1)!} \frac{\partial^{\ell + 1} H_2}{\partial x_3^{\ell + 1}}(p)  c^{\ell + 1}] t^{\ell + 1}  + O(t^{\ell + 2}),     \frac{\partial H_3}{\partial x_3}(p) c t + O(t^2)).
	\earyst
	From the above expression it is straightforward to conclude the proof.
\end{proof}

Clearly, the map $f$ induces a map  $A : \cE_{\ell} \to \cE_{\ell}$,  through its action on the level of curves along with the identification above (here we have assumed for simplicity the orientability of the invariant bundles).
 
We now show that $A$ is a vector bundle automorphism, which is smooth along $W^1$. This is an immediate consequence of the following.

\begin{prop}\label{prop.cocycle}
     Given a $\mu$-typical $x \in M$. Recall that $F_x = (F_1, F_2, F_3)$. Then  we have
	\begin{equation}\label{eq:elljetNF}
		\begin{pmatrix} \partial^{(\ell + 1)} y_2 \\ (\partial y_3)^{\ell + 1} \end{pmatrix}(0) = \begin{pmatrix} \lambda_{2,x} &  \frac{\partial^{\ell + 1} F_2}{\partial x_3^{\ell + 1}}(x_1, 0, 0) \\[6pt] 0 & \lambda_{3,x}^{\ell + 1} \end{pmatrix} \begin{pmatrix} \partial^{(\ell+1)}x_2 \\ (\partial x_3)^{\ell + 1} \end{pmatrix}(0)
	\end{equation}
	where $t \mapsto (x_1(t), x_2(t), x_3(t))$ represents a curve of form \eqref{eq formofcurve} through $(x_1, 0, 0)$ in chart $\imath_x$, and  $t \mapsto (y_1(t), y_2(t), y_3(t)) = F_x(x_1(t), x_2(t), x_3(t))$ represents a curve of form \eqref{eq formofcurve} through $( \lambda_{1,x} x_1, 0, 0)$ in chart $\imath_{f(x)}$.
\end{prop}

\begin{proof} 
 We look at the image by $F_x$ of a curve $\gamma:(-\eps,\eps) \to (-1,1)^3$ through $(x_1, 0, 0)$ of the form \eqref{eq formofcurve} for some values of $b, c \in \RR$.

	Using Lemma \ref{lem.derivativezero} and Taylor's expansion we get that the map $F_x\circ \gamma$  is of the form: 
	\[
	\begin{aligned}
		t \mapsto & (   \lambda_{1,x } x_1 +  O( t), \\ & \lambda_{2,x} b t^{\ell + 1} + \frac{1}{(\ell+1)!} \frac{\partial^{\ell + 1} F_2}{\partial x_3^{\ell + 1}}(x_1, 0, 0) c^{\ell+1} t^{\ell + 1} + O(t^{\ell+2}),  \lambda_{3,x} c t + O(t^2)). 
	\end{aligned}
	\]
	By a substitution $(b, \hat c) =(b,c^{\ell + 1})$, we have 
	$$(b,\hat c)\mapsto (\lambda_{2,x} b + \frac{1}{(\ell + 1)!} \frac{\partial^{\ell + 1} F_2}{\partial x_3^{\ell+1}}(x_1,0,0) \hat c, \lambda_{3,x}^{\ell + 1} \hat c). $$
	Since we have $(b, c) = (\frac{1}{(\ell+1) !} \partial^{(\ell+1)} x_2(t)|_{t = 0}, \partial x_3(t)|_{t = 0})$, this completes the proof. 
\end{proof}

%

%

%
 
\begin{remark}
The bundles defined in this section correspond to some components of the $(\ell + 1)$-jet bundle of curves through typical points in unstable manifolds of generic points. 
\end{remark}

\subsection{Construction of $\ell$-good charts}  \label{subsec Conlgoodchart}
Proposition \ref{p.normalform} is a consequence of the following proposition. 
\begin{prop}\label{p.changepoly}
Assume that $f$ admits $(\ell-1)$-good unstable charts $\{\imath_x\}_x$ and that equation \eqref{eq:highordertemplate} is verified, then there is a smooth change of coordinates which produces $\ell$-good unstable charts for $f$. 
\end{prop}

Besides proving Proposition \ref{p.normalform}, the proof of this proposition allows us to obtain a formula for the change of coordinates. 

\begin{proof}
We are in the situation of Proposition \ref{prop.cocycle} and thus we can write the action on $(\ell + 1)$-jets as a cocycle as in formula \eqref{eq:elljetNF}. Applying Proposition \ref{p.CNF} one can obtain a smooth change of  coordinates of the form 
\begin{equation}\label{eq:changepoly}
 (x_1, x_2, x_3) \mapsto (x_1, x_2 + u_{x,\ell}(x_1) x_3^{\ell+1} , x_3) 
\end{equation}

\noindent giving that the action on $(\ell + 1)$-jets is polynomial and thus providing $\ell$-good unstable charts as desired.  
  \end{proof}

\begin{remark}\label{rem.HolderJet} 
Assume that $f$ is uniformly partially hyperbolic, 
then we can inductively show that $\ell$-good unstable charts, if they exist, can be made to depend H\"older continuously in a neighborhood of any predetermined $x \in M$.  

When $\ell = 0$, this is the content of Remark \ref{rem.0goodchartuniform}. 
Now we consider the general case. Given an arbitrary $x \in M$, and a family of $(\ell-1)$-good unstable charts $\{ \iota_x \}_{x \in M}$ depending H\"older continuously on the base point near $x$, both the bundle $\tilde \cE_{\ell}$ and $\tilde F$ constructed above depend H\"older continuously on the base point near $x$. Since $x$ is arbitrary, the bundle $\tilde \cE_{\ell}$ and $\tilde F$ are H\"older, and smooth along the unstable manifolds.
Then by Proposition \ref{p.CNF.uniform}, the chart we obtained by applying Proposition \ref{p.changepoly} satisfies the inductive hypothesis: they can be made H\"older, possibly after a coordinate change, in a neighborhood of any predetermined $x$.
\end{remark}

\subsection{Improvement of charts}\label{s.improvechart}

Here we prove the following proposition that is the starting point of the proof of Theorem \ref{teo.maintechnical1}.

\begin{prop}\label{p.varphipolyell+1}
Let $\mu$ be a partially hyperbolic measure and $\{\imath_x\}$ a family of $\ell$-good unstable charts. If there exists an integer $d_0 > 0$ such that the stable templates of $(\ell+1)$-jets $\cT^\ell_x$ (given by \eqref{eq:highordertemplate}) are polynomials of degree at most $d_0$ for almost every $x \in M$, then $\mu$ admits $(\ell+1)$-good unstable charts. The symmetric statement holds for $\ell$-good stable charts. 
\end{prop}

\begin{proof}
By equation \eqref{eq:highordertemplate} we have:  
$$\imath_x^{-1}(\cW^3_{loc}(\Phi^1_x(t)))  = \{ (t  + O(s),\cT^\ell_{x}(t) s^{\ell+1} + O(s^{\ell+2}),s ) \ : \ s \in (-1,1) \} $$
for some $\rho' > 0$ depending only on $f$ and $\ell$. 

Since we know by assumption that $\cT^\ell_x$ is a polynomial, we can consider the new smooth charts $\imath'_x = \imath_x \circ \psi_x$ where: 
\begin{equation}\label{eq:psichange}
\psi_x(t,u,s) = (t, u + \cT^\ell_x(t)s^{\ell+1}, s). 
\end{equation}
\noindent 
We have $\psi_x^{-1}(t,u,s) = (t, u - \cT^\ell_x(t)s^{\ell+1}, s)$. Then 
$$(\imath_x')^{-1}(\cW^3_{\rho'}(\Phi^1_x(t)))  = \psi_x^{-1} \circ \imath_x^{-1}(\cW^3_{\rho'}(\Phi^1_x(t)))= \{ (t  + O(s), O(s^{\ell+2}),s ) \ : \ s \in (-1,1) \}. $$

Thus the new charts verify condition \eqref{eq:highordertemplate}. Using Proposition \ref{p.normalform} we complete the proof of the proposition. 
\end{proof}

\subsection{The uniform case} 
The results in this section extend to the uniform setting with minor modifications. Let us state the results we will use and discuss briefly the adaptations needed to obtain such statements.

We first need the notion of $\ell$-good uniform unstable charts, parallel to Definition \ref{def.normalcharts} as the measurability there will be replaced by continuity. This makes sense in view of the uniformity and the fact that normal form coordinates vary continuously with the point (in higher dimension, the existence of analogous uniform charts in general remains obscure). 

\begin{defi}[$0$-good uniform unstable charts]\label{def.normalchartsunif}
Let $\Lambda$ be a partially hyperbolic set of a smooth diffeomorphism $f$ on a closed 3-manifold. A continuous collection of smooth diffeomorphisms $\{\imath_x:  (-2 \|Df\|, 2 \|Df\|)^3  \to M\}_{x \in M}$ is a family of \emph{ uniform unstable coordinate charts} if it verifies that for every $x \in \Lambda$ we have that $\imath_x(t_1,0,0)= \Phi^1_x(t_1)$, $\imath_x(0,0,t_3)= \Phi^3_x(t_3)$ for $t_1,t_3 \in (-1,1)$, $\partial_2 \imath_x(0,0,0)$ is a unit vector in $E^2(x)$ and if we write $F_x := \imath_{f(x)}^{-1} \circ f \circ \imath_x = (F_{x,1},F_{x,2}, F_{x,3})$, then $F_x: (-1,1)^3 \to \RR^3$ verifies: 
\begin{enumerate}
\item $\partial_2 F_{x,2}(t,0,0) =  \lambda_{2,x}$ for all $t \in (-1,1)$,
\item  $\partial_3 F_{x,3}(t,0,0) =  \lambda_{3,x}$ for all $t \in (-1,1)$,
\item $\partial_2 F_{x,3}(t,0,0) = 0$ for all $t \in (-1,1)$. 
\end{enumerate}
A family of unstable uniform coordinate charts is called $0$-\emph{good} if moreover, there is some $d\geq 1$(independent of $x \in \Lambda$) such that
\begin{equation}
\partial_3F_{x,2}(t,0,0) \text{ is a polynomial of degree } \leq d \text{ in } t \in (-1,1).
\end{equation}
\end{defi} 

\begin{remark}\label{remark.continuityuniformcharts}
Technically, since $\Lambda$ may have some non-trivial topology, it is possible that the tangent space $T_\Lambda M$ which splits in 3-bundles $E^1 \oplus E^2 \oplus E^3$ cannot be coherently oriented.  This imposes an obstruction for the existence of uniform coordinate charts. There are several solutions for this issue. One is to take a finite cover of (a neighborhood of) $\Lambda$ and work there. Note that our results are independent of this finite cover and thus this will not result in a loss of generality.  Taking charts defined on a fixed square is convenient to avoid charging the notation. We will thus implicitly assume throughout that the bundles are orientable and therefore this obstruction is not existing. The reader not comfortable with this assumption can consider either local families of smooth diffeomorphisms or directly parametrize the charts in cubes defined in the tangent space of each point.  
\end{remark}

The following definition is parallel to Definition \ref{def.lgoodcharts}.

\begin{defi}[$\ell$-good uniform unstable charts]\label{def.lgoodchartsuni}
Let $\{\imath_x\}_{x \in \Lambda}$ be a family of $0$-good unstable charts  a partially hyperbolic set $\Lambda$.  We say the family is $\ell$-\emph{good} if for  every $x\in \Lambda$ there are (unique) continuous functions $\cT^\ell_{x}: (-1,1) \to \RR$, $a_x: (-1,1)^2 \to \RR$ and $b_x: (-1,1)^2 \to \RR$ such that for every $t \in (-1,1)$ so that $\Phi_x^1(t) \in \Lambda$ we have that: 
\begin{equation}\label{eq:highordertemplateuni}
\imath_x^{-1}(W^3_{ loc }(\Phi^1_x(t))) = \{ (t  + a_x(t,s)s, \cT^\ell_{x}(t) s^{\ell+1} + b_x(t,s)s^{\ell+2},s ) \ : \ s \in (-1,1) \}. 
\end{equation}
and for some uniform constant $d:=d(\ell,f)$ (independent of $x$) we have that 
\begin{equation}\label{eq:goodcoordinateuni} 
\partial_3^\ell F_{x,2}(t,0,0)  \text{ is a polynomial of degree } \leq d \text{ in } t \in (-1,1).
\end{equation}
\end{defi}

We will need the following result whose proof is omitted as it is in close parallel with whose of Proposition \ref{p.exist1goodchart} and  Proposition \ref{p.varphipolyell+1}  (In particular,  the same steps  as explained in Remark \ref{rem.HolderJet} can be used).   

\begin{prop}\label{p.CNFuniform}
Every partially hyperbolic set $\Lambda$ admits a family of $0$-good uniform unstable charts. Moreover, if $\Lambda$ admits a family of $\ell$-good uniform unstable charts and the template $\cT^\ell_x$ given by equation \eqref{eq:highordertemplateuni} is a polynomial for every $x$, then $\Lambda$ admits $(\ell+1)$-good uniform unstable charts.  Moreover, the $(\ell+1)$-good uniform unstable charts can be chosen to depend continuously on  base point if the $\ell$-good uniform unstable charts depend continuously on base point.
\end{prop}


\section{Proof of the dichotomy: Proposition \ref{p.dichotomy}}\label{s.dichotomy} 

We let $f: M \to M$ be a smooth diffeomorphism and let $\mu$ be a non-degenerate partially hyperbolic ergodic measure (cf. Definition \ref{standing-noatoms}) with $\ell$-good unstable charts (cf. Definition \ref{def.lgoodcharts}).  Let $\chi_1 > \chi_2 > \chi_3$ be the Lyapunov exponents of $\mu$.

For a compact set $K \subset M$ and $x \in K$ we denote $\hat K_x = (\Phi^1_x|_{(-1,1)})^{-1}(K)$. 
 Under our non-degeneracy assumption we have: 
\begin{lema}\label{l.compactintersectmanypoints}
For every compact subset $K \subset M$  with $\mu(K)>0$ we have that for $\mu$-almost every $x \in K$ the set $W^{1}_1(x) \cap K$ is infinite. 
\end{lema}

\begin{proof}
Let $\cA= \{ x \in M  :  \hat \mu_x^1 \ \text{has at least one atom} \}$.  Since $\cA$ is $f$-invariant, by ergodicity it either has zero or full $\mu$-measure. Since $\mu$ is non-degenerate, $\cA$ has zero measure. 
For almost every $x \in K$, we have $x \notin \cA$ and $\mu^1_x(W^1_1(x) \cap K) > 0$ since $\mu(K) > 0$.
For any such $x$, $W^1_1(x) \cap K$ is infinite.
\end{proof}

Proposition \ref{p.dichotomy} is a consequence of the following: 

\begin{prop}\label{p.dichotomytech} 
Let $\ell \geq 0$ and let $\mu$ be a partially hyperbolic measure admitting $\ell$-good unstable charts. Then there exists an integer $d:=d(\ell, f, \mu)  > 0$ such that:
\begin{enumerate}
\item either for $\mu$-a.e. $x \in M$ we have that $\cT^\ell_x$ is a polynomial of degree $d$  when restricted to 
 a full measure set with respect to  $\hat \mu^1_x$, 
\item\label{it.2section4} or for $\mu$-a.e. $x \in M$ if $S_x \subset (-1,1)$ is a subset with positive $\hat \mu_x^1$ measure, then $\cT^\ell_x|_{S_x}$ is not smooth in the sense of Whitney. 
\end{enumerate}
\end{prop}

\begin{proof} 

We write $\lambda_{i,x}^{(n)} := \lambda_{i, f^{n-1}(x)} \cdots \lambda_{i, x} \in \{ \pm \|D_xf^n |_{E^i(x)}\| \}$ for $i \in \{1, 3\}$, for each integer $n \geq 0$ and for $\mu$-a.e. $x$. We will use the notation $J_x^{(n)}:= (- (\lambda_{1,x}^{(n)})^{-1},  (\lambda_{1,x}^{(n)})^{-1})$.

Since we have $\ell$-good unstable charts,  by definition there is an integer $d$ such that \eqref{eq:goodcoordinate} holds.
Iterating \eqref{eq:dynamicstemplate} we get the following formula for $t \in J_x^{(n)}$:  
\begin{equation}\label{eq:templatentimes}
\cT^\ell_{f^n(x)}(\lambda_{1,x}^{(n)} t) = \frac{\lambda_{2,x}^{(n)}}{(\lambda_{3,x}^{(n)})^{\ell+1}} \cT^\ell_x(t) + P^{(n)}_x(t) 
\end{equation}

\noindent where $P^{(n)}_x$ is a polynomial of degree $\leq d$.
After a change of variables in \eqref{eq:templatentimes} we get: 
\begin{equation}\label{eq:templatentimes2}
\cT^\ell_{x}(t) = \alpha_x^{(n)} \cT^\ell_{f^{-n}(x)}(\beta_x^{(n)} t) + Q^{(n)}_x(t) 
\end{equation}
\noindent where $$\alpha^{(n)}_x= \frac{\lambda_{2,f^{-n}(x)}^{(n)}}{(\lambda_{3,f^{-n}(x)}^{(n)})^{\ell+1}} \ \ , \ \ \beta_x^{(n)}=(\lambda_{1,f^{-n}(x)}^{(n)})^{-1}$$ \noindent and $Q^{(n)}_x(t) = P_{f^{-n}(x)}^{(n)}(\beta_x^{(n)} t )$ is also a polynomial in $t$ of degree $\leq d$. By enlarging $d$ if necessary, we may assume without loss of generality that
\begin{equation}\label{eq:choiceofd} 
\lim_{n} \frac{1}{n} \log( \alpha_x^{(n)} (\beta_x^{(n)})^m) < 0  \quad \text{ for every } m\geq d \text{ and } \mu\text{-a.e. } x \in M.
\end{equation}
Note that it suffices to take
\aryst 
d > \frac{\chi_2 - (\ell+1)  \chi_3 + (\ell + 2)\epsilon}{ \chi_1 - \epsilon}.
\earyst

We denote by $\cA \subset M$ the set of $x \in M$ with the following property: there is a compact set $S_x$ of positive $\hat \mu_x^1$-measure such that $\cT^\ell_x$ is smooth in the sense of Whitney on $S_x$. We assume that $\mu(\cA)>0$, for otherwise we already have (ii). Then by ergodicity and by \eqref{eq:dynamicstemplate}, we have $\mu(\cA) = 1$.
By definition, it is clear that for $\mu$-almost every $x \in \cA$ and for almost every $y$ (with respect to $\mu^1_x$) in a neighborhood of $x$,  we have $y \in \cA$.

We may upgrade the set $\cA$ in the following the way.
We denote by $\cB \subset M$ the subset of $x \in M$ such that for every $x \in  \cB$, there is a compact set $S_x$ of positive $\hat\mu^1_x$-measure such that $\cT^{\ell}_x$ is smooth in the sense of Whitney on $S_x$ and moreover {\it $x$ is a density point of $S_x$ with respect to $\hat\mu^1_x$}. By definition, we see that for every $x \in \cA$, the set $\cB \cap W^1_{loc}(x)$ has positive $\hat\mu^1_x$-measure. Since we have seen that $\cA$ is a full measure set. This means that $\mu(\cB) > 0$. Then by ergodicity, we have $\mu(\cB) = 1$.

 We fix some small constant $\eps > 0$.
 By Lusin's lemma, there is a compact subset $\cQ \subset \cB$ such that $\mu(\cQ) > 1 - \frac{\eps}{100}$, and the conditional measure $\mu^1_x$ depend continuously on $x \in \cQ$.  Moreover, by slightly reducing the size of $\cQ$ if necessary, we may assume in addition to the above that $\cT^{\ell}_x$, as a function defined $\hat\mu^1_x$-almost everywhere, depends continuously on $x \in \cQ$, in the following sense. 
 For every Cauchy sequence $\{ x_n \}_{n \geq 0}$ in $\cQ$ converging to some $x \in \cQ$, 
 there exists a compact subset $E_n \subset (-1,1)$ for each $n \geq 0$ such that, as $n$ tends to infinity, $\hat\mu^1_{x_n}(E_n)$ converges to $1$, and $E_n$ converges in Hausdorff's distance to a compact subset $E$ of $\hat\mu^1_x$-measure $1$, such that for every sequence $\{ t_n \in E_n \}_{n \geq 0}$ converging to $t \in E$, we have that $\cT^{\ell}_{x_n}(t_n)$ converges to $\cT^{\ell}_x(t)$.

Summarizing the above, we deduce that  there is \clb a point $x \in \cQ$ with the following properties: 
\begin{itemize}
\item one has that  $\mu(\overline{\{f^n(x)\}_{n \geq 0}  \cap \cQ } ) = \mu(\cQ)$,
\item there is a compact set $\hat S_x \subset \mathrm{supp}(\hat \mu_x^1)$ so that $\cT^\ell_x|_{\hat S_x}$ is smooth in the sense of Whitney, and
\item $\frac{\hat \mu_x^1(\hat S_x \cap  J_x^{(n)} )}{\hat \mu_x^1( J_x^{(n)})}$ tends to $1$ as $n$ tends to infinity. 
\end{itemize}

We can write (cf. \eqref{eq:whitney}) for some $c > 1$ that: 
\begin{equation}\label{eq:taylorphi}
\cT^\ell_x(t) = a_{x, 1} t + \ldots + a_{x, d}t^d + \hat \cT^\ell_x(t) \ \text{where} \ |\hat \cT^\ell_x(t)| \leq c |t|^{d+1} \text{ if }  t \in \left(-\frac{1}{c}, \frac{1}{c}\right) \cap \hat  S_x.  
\end{equation}

Pick $y \in \cQ$ so that $\mu(B_\eps(y) \cap \cQ)>0$ for all $\eps>0$, and a sequence $n_i \to \infty$ so that $f^{n_i}(x) \in \cQ$, and converges to $y$. Notice that we can deduce from \eqref{eq:templatentimes2} and \eqref{eq:taylorphi}:
\begin{equation} \label{eq:ctellxandqn}
\cT^\ell_{f^{n_i}(x)}(t) = \hat Q^{(n_i)}_{f^{n_i}(x)}(t) + \alpha_{f^{n_i}(x)}^{(n_i)} \hat \cT^\ell_{x}( \beta_{f^{n_i}(x)}^{(n_i)} t) 
\end{equation}

\noindent where $\hat Q^{(n_i)}_x(t)$ is a polynomial of degree $\leq d$. If $\beta_{f^{n_i}(x)}^{(n_i)} t \in \hat S_x \cap J_x^{(n_i)}$ we have that
\begin{equation}
 \alpha_{f^{n_i}(x)}^{(n_i)} \hat \cT^\ell_{x}(\beta_{f^{n_i}(x)}^{(n_i)} t) \leq \alpha_{f^{n_i}(x)}^{(n_i)} c (\beta_{f^{n_i}(x)}^{(n_i)})^{d+1} |t|^{d+1}.
\end{equation}
Notice that the $\hat\mu^1_{f^{n_i}(x)}$-measure of the set of $t$ satisfying $\beta_{f^{n_i}(x)}^{(n_i)} t \in \hat S_x \cap J_x^{(n_i)}$ is at least
\aryst
\hat\mu^1_x( \hat S_x \cap J_x^{(n_i)}) / \hat\mu^1_x( J_x^{(n_i)} )
\earyst
which tends to $1$ as $i$ tends to infinity.

Up to passing to a subsequence of $(n_i)_{i \geq 0}$, we have that for every $i \geq 0$ there exist a polynomial $R_i$ of degree $\leq d$, and a subset $E_i \subset (-1,1)$ such that:  as $i$ tends to infinity,
$\hat\mu^1_{f^{n_i}(x)}(E_i)$ converges to $1$;
$E_i$ converges in the Hausdorff's distance to a subset $E \subset (-1,1)$ of full $\hat\mu^1_y$-measure; and for every sequence $\{t_i \in E_i\}_{i \geq 0}$, we have $R_i(t_i)$ converges to $\cT^{\ell}_y(t)$.

By Lemma \ref{l.compactintersectmanypoints}, $E$ contains infinitely many points.
 Since a polynomial with degree $\leq d$ is determined by its values at $d+1$ points,  we deduce that $\cT^\ell_y$ is a polynomial of degree $d$ on a full $\hat \mu^1_y$-measure set for every $y \in \cQ$.
 By letting $\eps$ tend to $0$ we deduce that $\cT^\ell_y$ is a polynomial of degree $d$ when restricted to the support of $\hat \mu^1_x$ for $\mu$-a.e. $x \in M$.  
\end{proof}

\section{Polynomials and rational functions}\label{s.rational}

We consider the collections of functions $\mathrm{Poly}^{d}= \{ p: [-1,1] \to \RR \ : \ p \text{ is a polynomial of degree } \leq d \}$ and $\cR^d= \{ \frac{q}{p}:[-1,1] \to \RR \ : \ p,q \in \mathrm{Poly}^d \ \mbox{ and } \ p(t)\neq 0 \ \forall t \in [-1,1] \}$. Clearly we have that $\mathrm{Poly}^d \subset \cR^d$.  
 We note that $\mathrm{Poly}^d$ is a linear subspace of $C^0([-1,1])$, but $\cR^d$ is not. 
  
We will need a compactness result which is standard for polynomials. We first give a definition. Given constants $k \in \ZZ_{>0}$, $\sigma,\eta>0$, we say that a subset $E \subset [-1,1]$ is $(k,\sigma,\eta)$-\emph{spread} if for any intervals $I_0, \ldots, I_k$ such that $\sum_i |I_i|< \eta$ we have that $E \setminus \bigcup I_i$ has at least $k+1$ points with pairwise distances strictly larger than $\sigma$. 

\begin{prop}\label{p.compactnessrational} 
For every $d \in \ZZ_{>0}$, $\sigma, \eta>0$ there is $C:=C(d,\sigma,\eta) > 0$  such that for any $(d,\sigma,\eta)$-spread subset $E\subset [-1,1]$, and any $R \in \cR^d$ satisfying $\sup_{t \in E} |R(t)| = 1$, the following is true:
\begin{enumerate}
\item\label{it1prop}  there are intervals $I_0, \ldots, I_d$ such that $\sum_i |I_i| < \eta$ and $|R'(t)|< C$ for every $t \in [-1,1] \setminus \bigcup I_i$,
\item\label{it2prop}  there are intervals $J_0, \ldots, J_{2d}$ such that $\sum_i |J_i| < \eta$ and $|R(t)|> C^{-1}$ for every $t \in [-1,1] \setminus \bigcup J_i$.
\end{enumerate}
\end{prop}

 \begin{remark}\label{remark-rescale}
We will use this result in intervals of varying length (not always $[-1,1]$) and for rational functions with possibly different normalizations (not always $\sup_{t \in E} |R(t)| = 1$). Assume that the rational function $R$ is defined on $[a,b]$, and $\sup_{t \in E} |R(t)|=A$ for some $E \subset [a,b]$ such that $\xi(E)$ is $(d,\sigma,\eta)$-spread, where $\xi: [a,b] \to [-1,1]$ is the unique  affine bijection. Then we can apply the result to $\hat R(t) = \frac{1}{A} R(\xi^{-1}(t))$ which is a rational function of the same degree defined on $[-1,1]$. We obtain that the derivative of $\hat R$ is less than $C=C(d,\sigma,\eta)$ except in a finite family of intervals which cover a small proportion (less than $\eta$) and thus the derivative of $R$ is less than $\frac{C}{A(b-a)}$ by the chain rule. In the same way, the lower bound for $|R(t)|$ in (\ref{it2prop}) becomes $\frac{A}{C}$. 
\end{remark} 

To prove this proposition we will need the following elementary result that will also serve other purposes:

\begin{lemma}\label{l.subsequencerational}
Let $C_0, \sigma > 0$ and 
let $( R_n \in \cR^d )_{n \geq 1}$ be a sequence of rational functions such that for every $n \geq 1$ there exist  points $t_{0,n}, \ldots, t_{d, n} \in [-1,1]$ with pairwise distances strictly larger than $\sigma$ verifying that $\sup_{i} |R_n(t_{i,n})| \leq C_0$. Then, there exist a subsequence $n_j \to \infty$, points $s_1, \ldots, s_d \in \DD_2$, and a rational function $R_\infty \in \cR^d$ (whose poles are contained in $\{s_1, \ldots, s_d\}$) such that $R_{n_j}$ converges to $R_{\infty}$ uniformly on compact subsets of $\DD_2 \setminus \{s_1, \ldots, s_d\}$ (where $\DD_2 = \{ z\in \CC \ : \ |z| \leq 2 \}$).
\end{lemma} 

To see the need to take out some points from the interval, we may consider the sequence $\{ R_n(z) = \frac{1}{n z^2 + 1} \}_{n \geq 1}$. 

\begin{proof}
We can write 
\begin{equation}\label{eq:rational1}
R_n(z) = c_n \frac{\prod_{i=1}^{k_n} (z-a_{i, n})}{\prod_{j=1}^{m_n}(z-b_{j, n})}. 
\end{equation}
\noindent where $c_n, a_{i, n}, b_{j, n} \in \CC $ and $0 \leq k_n, m_n \leq d$\footnote{We use the convention that $\prod_{i=1}^0 (z-\gamma_i) =1$.}. 
Up to considering a subsequence, we can assume that $k_n = k$ and $m_n=m$ are constant for all $n$ and that $c_n \to c_\infty$, $a_{i, n} \to a_{i, \infty}, b_{j, n} \to b_{j, \infty}$ all converge in $\overline{\CC}= \CC \cup \{\infty\}$.

We order $a_{i, n}$ and $b_{j, n}$ so that they decrease in modulus. We let $\hat k \in \{1,\ldots, k+1\}$ and $\hat m \in \{1,\ldots m+1\}$ the smallest integers so that $a_{\hat k, \infty}, b_{\hat m, \infty} \in \CC$ (so that $a_{i, \infty} = \infty$ if $i<\hat k$ and $b_{j, \infty} = \infty$ if $j < \hat m$; if $\hat k = k+1$ or $\hat m=m+1$ means that all coefficients diverge).  

We have the following.
\begin{claim}
Up to taking  further a subsequence we have that the sequence of functions 
$$\hat c_n(z):= c_n \frac{\prod_{i=1}^{\hat k-1} (z- a_{i, n})}{\prod_{j=1}^{\hat m -1} (z-b_{j, n})} $$
\noindent converges uniformly in $\DD_2$ to a constant function $\hat c_\infty \in \CC$. 
\end{claim}

\begin{proof}
Up to taking a subsequence we can assume that the tuple of points $(t_{i, n})_{i=0}^{d}$ converge to the tuple of points $(t_{i, \infty})_{i=0}^{d} \in [-1,1]^d$ which are pairwise at distance $\geq \sigma$. By  $k- \hat k \leq d$, we can assume without loss of generality that $t_{0, n}$ is uniformly far from $a_{i, n}$ for all $\hat k \leq i \leq k$ (and therefore for all $1\leq i \leq k$ as $a_{i, n} \to \infty$ if $i< \hat k$).

It is enough to show that the functions $\hat c_n(z)$ are bounded uniformly in some point of $\DD_2$ since one can compute the logarithmic derivative as: 

$$ \frac{\hat c_n'(z)}{\hat c_n(z)} = \sum_{i=1}^{\hat k-1} \frac{1}{z - a_{i, n}} - \sum_{j=1}^{\hat m-1} \frac{1}{z- b_{j, n}}, $$
\noindent which converges uniformly to $0$ in $\DD_2$ because the coefficients $a_{i, n}$ and $b_{j, n}$ diverge. 

To get the uniform boundedness, we compute the value of $\hat c_n$ in the point $t_{0, n} \in [-1,1] \subset \DD_2$. Notice that

$$R_n(t_{0, n}) = \hat c_n(t_{0, n})\frac{\prod_{i=\hat k}^k (t_{0, n} - a_{i, n})}{\prod_{i = \hat m}^m (t_{0, n} - b_{i, n})}$$ \noindent is uniformly bounded. Since the product $ \prod_{i=\hat k}^k (t_{0, n} - a_{i, n})$ is uniformly bounded from below  and $R_n(t_{0, n})$ is uniformly bounded from above, we get the desired result.  
\end{proof}

Now it is easy to show that outside any given neighborhood of $\{b_{\hat m, \infty}, \ldots, b_{m, \infty} \}$ in $\DD_2$ the sequence $(R_n)_{n \geq 1}$ converges uniformly to 
\begin{equation}\label{eq:rinfty}
R_\infty (z) = \hat c_\infty  \frac{\prod_{i=\hat k}^{k} (z-a_{i, \infty})}{\prod_{i = \hat m}^{m}(z-b_{i, \infty} )}.
\end{equation}

The rational function $R_\infty$ verifies the desired properties. 
\end{proof}

\begin{proof}[Proof of Proposition \ref{p.compactnessrational}]

We detail the proof fo  (\ref{it1prop}), and give some sketch for the proof of (\ref{it2prop}) as it is similar.

Assuming to the contrary that (\ref{it1prop}) fails.
Then there is a sequence $R_n = \frac{Q_n}{P_n} \in \cR^d$ so that: (1) there are $(d,\delta,\eta)$-spread sets $E_n \subset [-1,1]$  with $|R_n(t)| = 1$ for all $t \in E_n$; and (2)  there is some $t \in [-1,1] \setminus \bigcup I_i$ so that $|R'_n(t)|>n$  for every family of intervals $I_1, \ldots, I_{d}$ whose sum of lengths do not excede $\delta$.

Using Lemma \ref{l.subsequencerational} and our hypotheses on $E_n$, we can find a rational function $R_\infty \in \cR^d$, a subsequence $n_j \to \infty$ and points $s_1, \ldots, s_d \in \CC$ containing the poles of $R_\infty$ such that $R_{n_j} \to R_\infty$ on every compact subset of $\DD_2 \setminus \{s_1, \ldots, s_d\}$. In particular, on every compact subset of $\DD_2 \setminus \{s_1, \ldots, s_d\}$ we have that $R_{n_j}' \to R'_\infty$ uniformly.   By covering the set $\{ s_1, \ldots, s_d \} \cap \RR$ by small  open intervals whose lengths add up to less than $\sigma$ we find a contradiction since $R'_\infty$ is bounded away from those intervals.  This proves (\ref{it1prop}).

To prove (\ref{it2prop}), we may construct $R_n$, $E_n$ as before, but instead of (2) we assume that 
there is some $t \in [-1,1] \setminus \bigcup J_i$ so that $|R_n(t)| < 1/n$  for every family of intervals $J_1, \ldots, J_{2d}$ whose sum of lengths do not excede $\delta$.  
 Then one can use that $R_\infty$ has the form given by equation \eqref{eq:rinfty} and use logarithmic derivatives (i.e. consider the derivative of $\log (R_\infty)$) to see that 
$$ \frac{ R'_\infty (z)}{R_\infty(z)} = \sum_i \frac{1}{z- a_i} - \sum_j \frac{1}{z-b_j}, \ \ \ \ a_i, b_j \in \CC. $$
We let $\cup_{i = 1}^{2d} J_i$ cover a neighborhood of $\R \cap \{ a_i, b_j \}$ in $\R$. Then it is direct to deduce (\ref{it2prop}) by contradiction. 

This concludes the proof of the proposition. 
\end{proof}

\section{Distance to rational functions}\label{s.computations} 

This section is devoted to showing the following statement. 

\begin{prop}\label{p.qniellgood}
Let $\mu$ be a partially hyperbolic measure with $\ell$-good unstable charts $\{\imath_x\}$ such that  the stable templates $\cT^\ell_x$ (cf. \eqref{eq:highordertemplate}) are not in $\mathrm{Poly}^d$ for some $d = d(f, \ell) = O(\ell)$. Then $\mu$ has QNI. 
\end{prop}

As a consequence, using Proposition \ref{p.varphipolyell+1} we deduce the following. 
\begin{cor}
 If $\mu$ does not have QNI then it admits $\ell$-good stable and unstable charts for every integer $\ell\geq 1$. 
 \end{cor} 
 
 \begin{remark}\label{remarklipschitz}
 The proof of Proposition \ref{p.qniellgood} can be simplified if one knows that the center unstable direction is more regular, for instance, if $f$ where an Anosov diffeomorphism with expanding center direction, then it gets simpler as the full unstable foliation is of class $C^{1+}$. This allows us to consider only polynomials instead of general rational functions in the proof below. However, to apply the result for $f^{-1}$ one would need to deal with the lack of integrability and regularity of the center stable subspaces (note that it is very rare for both the center stable and the center unstable subspaces be more regular than H\"{o}lder). In the next subsection we treat a toy case with some artificial simplifications to show the idea more transparently. 
 \end{remark}
 
 \subsection{A toy case}
 
 In this section, we will show that the fact that the templates are not polynomials provides some kind of non-integrability. To avoid technicalities and show the key ideas in a transparent way, we will make some simplifying assumptions. 
 
 Let us consider $\Lambda \subset M$ be a partially hyperbolic set saturated by $\cW^1$-leaves and which is minimal (i.e. for every $x \in \Lambda$ we have that $\cW^1(x)$) and let $\mu$ be a fully supported invariant ergodic measure on $\Lambda$ so that it is non-degenerate (for instance, this is automatic if $\chi_2(\mu)> -\chi_1(\mu)$ by an entropy argument). We will assume that  Proposition \ref{p.dichotomytech} (ii) holds with $\ell = 0$.
 In other words, for a $\mu$-typical $x$,  $\cT^0_x$ is not polynomial on any subset with positive $\hat \mu^1_x$-measure. 

 We wish to show some form of non-integrability (compare with \cite{CPS}): 
 
 \begin{prop}
 Under these assumptions, given $x\in \Lambda$ and a connected neighborhood $I$ of $x$ in $W^1_1(x)$, we have that for every $y \in \Lambda \cap  W^3_1(x)  \setminus \{x\}$ sufficiently close to $x$ there is a point $z \in I$ such that $W^3_1(z) \cap W^1_1(y) = \emptyset$. 
 \end{prop}
 
 This statement clearly follows from Proposition \ref{p.qniellgood} but the proof here allows one to avoid some technical details  which makes the proof easier to follow. The full proof of Proposition \ref{p.qniellgood} will be given in \S \ref{s.proof6}.

 
Without loss of generality, let $I = \Phi^1_x((0,r_1))$ for some $0<r_1 \ll 1$, and 
let  $y = \Phi_x^3(s)$ for some $s \in (0,1)$. We may let $s$ be small depending on both $x$ and $I$. 

Whenever $s$ is sufficiently small, the statement $W^3_1(z) \cap W^1_1(y) = \emptyset$ for some $z \in I$ can be reduced to $W^3_1(f^{-1}(z)) \cap W^1_1(f^{-1}(y)) = \emptyset$. By pull-backing a few more times if necessary, we may assume that 
\ary \label{eq ratiologsr1}
\log s / \log r_1 \in (V/2, 2V)
\eary
where $V > 1$ is a large constant to be given in due course.

We fix $1$-good unstable charts $\{\iota_{x'}\}_{x' \in \Lambda}$ and we can thus write in coordinates $\iota_x: (-1,1)^3 \to M$ the following sets: 
\begin{equation}\label{eq:W1x}
\iota_x^{-1}(W^1_1(x)) = \{(t,0,0) \ : \ t\in (-1,1) \}
\end{equation}
and\footnote{We write $W^1_{loc}(y)$ instead of $W^1_1(y)$ because the charts may change slightly the parametrization, but of course these two sets are very close to one another.} 
\begin{equation}\label{eq:W1y}
\iota_x^{-1}(W^1_{loc}(y)) = \{ (t', \hat Q(t'), \hat P(t')) \ : \ t' \in (-1,1) \}. 
\end{equation}

We will make the following simplifying assumption:
 \begin{equation}\label{eq:lipass}
\hat P(t') \equiv s 
\end{equation}
This is unreasonable in general, but in some cases it is not far from what happens (see Remark \ref{remarklipschitz}). At the end of this subsection we will explain how to lift this assumption.

Consider now, for small $t \in (0,r_1)$ the point $z = \Phi_x^1(t) \in I$ and we denote:
\begin{equation}\label{eq:W3z}
\iota_x^{-1}(W^3_{loc}(z)) = \{ (t + a(t)u + O(u^2), \cT^0_x(t)u + O(u^2), u) \ : \ u \in (-1,1) \}. 
\end{equation}

Under the assumption of \eqref{eq:lipass} we have
\begin{equation}\label{eqstimate1}
d(W^1_1(y), W^3_1(z)) \approx  |\hat Q(t') - \cT^0_x(t) s | + O(s^2)
\end{equation}
where $t' = t + a(t)s + O(s^2)$. Therefore,  it suffices to show for some $\rho \in (0, 1)$
\begin{equation}\label{eqstimate11}
\sup_{t \in (0,r_1)} \left|\frac{\hat Q(t')}{s} - \cT^0_x(t)\right| >  s^{\rho}.
\end{equation} 

Let $D=\sup_{t'\in(0, r_1)} | \hat{Q}(t')|$.
By Taylor's expansion, there is a polynomial $Q$ of degree $d_1$ (large, to be chosen)  and some $C_0$ depending only on $d_1$   such that  
\begin{equation}\label{eqstimate2}
\left|\frac{\hat Q(t')}{s} - Q(t')\right| <  \frac{ C_0 D r_1^{d_1}}{s}.
\end{equation}

By  \eqref{eq:dynamicstemplate}, and the hypothesis that $\cT^0_x$ does not coincide with a polynomial on any interval intersecting $\supp(\hat\mu^1_x)$,
 we deduce that there exist a constant $c_0$  depending on $d_1$ but independent of   $r_1$ (compare with Proposition \ref{prop.estimatepoly} below), and some $\alpha > 0$ depending only on the (uniform) expansion and contraction rates of $f$ ( in particular, it is independent of   $V$ and $d_1$), such that 
\begin{equation}\label{eqstimate3}
\inf_{P \in Poly^{d_1}} \sup_{t \in (0,r_1) \cap \supp(\hat\mu^1_x)} |P(t)  - \cT^0_x(t)| > c_0 r_1^{\alpha}. 
\end{equation}
See Proposition \ref{prop.estimatepoly2} for a detailed proof of \eqref{eqstimate3} in a more general setting.
Let $\hat D = \sup_{t \in (0,r_1)} |\cT^0_x(t)|$. Then \eqref{eqstimate3} implies that $| \hat D | \gtrsim r_1^{\alpha} \gtrsim s^{2\alpha / V}$.

We first assume that $D >  2 \hat D$. Then
there exists $t \in (-r_1, r_1)$ with $|Q(t') - \cT^0_x(t)| > D/2$. Then by \eqref{eqstimate2} and by letting $d_1 \gg V$, we have 
\aryst
\left|\frac{\hat Q(t')}{s} - \cT^0_x(t)\right| \geq | Q(t')  - \cT^0_x(t)|   - \left|\frac{\hat Q(t')}{s} - Q(t')\right|  \gtrsim  D -  \frac{C_0 D r_1^{d_1}}{s}  \gtrsim s^{2\alpha / V}.  
\earyst
 An important point is that here $\alpha$ can be chosen to be independent of the degree $d_1$.
 Indeed, in the following we will let $d_1$ to be large when needed, while keeping $\alpha$ unchanged.

Now we assume that $D < 2 \hat D$.
Then $ \sup_{t'\in(0, r_1)} | \partial Q(t') | < C' D$ for some $C'$ depending only on $d_1$.
By \eqref{eq ratiologsr1}, we have
\begin{equation}\label{eqstimate4} 
	|Q(t') - Q(t)| <  C' D |t-t'| \lesssim  C' D s.
\end{equation} 
Putting together   \eqref{eqstimate2} to \eqref{eqstimate4} we see that there exists $t \in (-r_1, r_1) \cap \supp(\hat\mu^1_x)$ with
 \ary
&&\left|\frac{\hat Q(t')}{s} - \cT^0_x(t)\right| \geq | Q(t)  - \cT^0_x(t)| - |Q(t') - Q(t)| - \left|\frac{\hat Q(t')}{s} - Q(t')\right|   \nonumber \\
&& \qquad \gtrsim \quad c_0 r_1^{\alpha}  -   C' D s -  \frac{C_0 D r_1^{d_1}}{s}.   \label{eqstimate5} 
\eary
By letting $d_1 \gg V \gg 1$, we deduce  \eqref{eqstimate11}. 
This completes the proof under the simplifying assumption \eqref{eq:lipass}. 

In the following we sketch the proof without assuming \eqref{eq:lipass}.  
In this case, equation \eqref{eqstimate1} becomes
\begin{equation}\label{eqstimate1'}
d(W^1_1(y), W^3_1(z)) \approx  |\hat Q(t') - \cT^0_x(t) \hat P(t')| + O(|\hat P(t')|^2). 
\end{equation} 
Instead of \eqref{eqstimate11}, we need a lower bound for  $|\cT^0_x(t) - \frac{\hat Q(t')}{\hat P(t')}|$.

Following the same strategy as above, we use Taylor's expansion to get $\hat Q(t')= Q(t')  + \eps_Q(t')$ and $\hat P(t') = P(t') + \eps_P(t')$. This gives 
\begin{equation}\label{eqstimate2'}
\left| \frac{\hat Q(t')}{\hat P(t')} - \frac{Q(t')}{\hat P(t')} \right| < C_0 D t^{d_1} |\hat P(t')|^{-2}. 
\end{equation}
The term $|\hat P(t')|^{-2}$ could be a problem if it is not bounded from below.
To overcome this, we need to use  H\"{o}lder holonomies of the strong unstable lamination (see e.g. \cite{PSW}):
There exist $c_1,c_2>0$ and $\gamma_1>1>\gamma_2>0$ such that:
\begin{equation}\label{eq:holderholon}
c_1 s^{\gamma_1} < |\hat P(t')| + |\hat Q(t')| < c_2 s^{\gamma_2}. 
\end{equation}

We separate the argument in two cases. Recall that $\cT^0_x$ is uniformly bounded by $\hat D \leq 1$. We first assume that $|\hat P(t')| \leq \frac{c_1}{10} s^{\gamma_1}$. In this case $|\hat Q(t')|$ must be larger than $\frac{9c_1}{10} s^{\gamma_1}$ and therefore, using \eqref{eqstimate1'} we get that for some constant $c_3>0$ 
\begin{equation}\label{eqstimate7'}
d(W^1_1(y), W^3_1(z)) \geq |\hat Q(t') - \cT^0_x(t) \hat P(t')| - c_3 |\hat P(t')|^2 \geq \frac{c_1}{2} s^{\gamma_1} >0
\end{equation}

\noindent which gives the desired non-integrability. 

Now assume that  $|\hat P(t')| \geq \frac{c_1}{10} s^{\gamma_1}$.  In this case  the right hand side of \eqref{eqstimate2'} is bounded by $C C_0 D  c_1^{-2} t^{d_1} s^{ - 2 \gamma_1} \ll s$ when $d_1 \gg V_1$. Denote $R(t') = \frac{Q(t')}{P(t')}$. We will use 
\begin{equation}\label{eqstimate8'}
|R(t')- \cT^0_x(t)| \geq |R(t) - \cT^0_x(t)| - |R(t')-R(t)|.  \nonumber
\end{equation}
We can bound $|R(t')-R(t)|$ from above following the similar strategy in the polynomial case.
 As in \eqref{eqstimate4} we treat the case where $R$ attains very large values with respect to $\cT^0_x$ (in which case there is no need to estimate $|R(t')-R(t)|$), and we treat the complementary case using Proposition \ref{p.compactnessrational}, which provides good bounds for the derivative of $R$. The conclusion of the proof is now very similar to the simplified case.

The rest of this section will be devoted to refine the above argument in order to obtain a more quantitative version of non-joint integrability.

\subsection{Some uniform distance} 
 We now deduce some consequences from the hypothesis that the functions $\cT^\ell_x$ are not polynomials.

Let $d_1$ be a sufficiently large positive integer depending only on $f, \ell$, to be determined later.
The underlying assumption of this section is that $\mu$ is a  non-degenerate  partially hyperbolic measure with $\ell$-good charts and the functions $\cT^\ell_x$ are not polynomials  of degree $\leq d_1$  restricted to the support of $\hat \mu^1_x$ for almost every $x$.  By ergodicity, we see that item (\ref{it.2section4}) of Proposition \ref{p.dichotomytech} holds.

\begin{prop}\label{prop.estimatepoly}
For every $\eps>0$,  for every integer $d > 0$, there is a compact set $K \subset M$ with $\mu(K)> 1-\eps$ such that for every $\nu>0$ there is $c:=c(d, \nu,\eps) >0$ such that for any $x \in K$ and any polynomials $P, Q$ of degree $\leq d$, the set 
\begin{equation}\label{eq:polydist}
I^{P,Q}_{x,c} = \left \{ t \in (-1,1)  \ : \  \left|\cT^\ell_x(t) - \frac{Q(t)}{P(t)} \right| \leq c \right\}
  \end{equation}
satisfies that $\hat \mu^1_x(I^{P,Q}_{x,c}) \leq \nu$.  
\end{prop}

Note that since $\cT^\ell_x$ is defined on a $\hat \mu^1_x$-full measure set, the set $I^{P, Q}_{x,c}$ is also only  defined up to a $\hat \mu^1_x$-null measure set (also recall that $\hat \mu^1_x$ is normalized so that it is a probability measure in $(-1,1)$).

\begin{proof} 
We proceed by contradiction. We notice that  if $c<c'$ then $I^{P,Q}_{x,c} \subset I^{P,Q}_{x,c'}$.
If the result does not hold then there exist an integer $d > 0$, a constant $\nu>0$, and a compact set $K_0 \subset M$ with $\mu(K_0)>0$ such that for every $x \in K_0$, for every integer $n > 0$ there exist polynomials $P_n, Q_n : (-1,1) \to \RR$ of degree $\leq d$ such that $\hat \mu^1_x(I^{P_n,Q_n}_{x, 1/n}) > \nu$. By reducing the size of $K_0$ if necessary, we may assume in addition that  all objects we will consider vary continuously on $K_0$ (cf. Proposition \ref{prop.measurable}). 

 We now show that there exists $\delta>0$ such that each set $I^{P_n, Q_n}_{x,1/n}$ contains $d +1$ points with pairwise distances larger than $\delta$.  
 
   Since by assumption $\mu$ is non-degenerate,  we may assume that $\hat \mu_x^1$ are non-atomic probabilities varying continuously on $x$ restricted to the compact set $K_0$, for any $\nu>0$ there exists $\delta > 0$ (which depends on $\mu$, $\nu$, $d$ and $K_0$) so that for every $x \in K_0$, every subset of $(-1,1)$ with $\hat \mu^1_x$-measure larger than $\nu$ must contain $d+1$-points with pairwise distances larger than $\delta$.

Let us fix an arbitrary $x \in K_0$.
Up to taking some subsequence, we can apply Lemma \ref{l.subsequencerational} to obtain a rational function $R_\infty \in {\cR}^{d}$ such that $\frac{Q_n}{P_n} \to R_\infty$ uniformly away from finitely many points in $[-1,1]$. We deduce that $\cT^\ell_{x}$ coincides with a rational function in a set of positive $\hat \mu^1_x$ measure (this is because we can remove intervals of uniform size around the points where the convergence is not uniform, and this will cover no more than half the measure of $\hat \mu_x^1$, so there is a positive measure set where the template coincides with a rational function, in particular smooth).  Since $\mu(K_0) > 0$ and $x$ is arbitrary, we can apply Proposition \ref{p.dichotomytech} to get a contradiction.
\end{proof}

\begin{remark}
Note that we cannot ensure with the limiting process that the template will coincide with a smooth function in some open set of the support a priori. This is why we need to deal with density points and apply Proposition \ref{p.dichotomytech}. 
\end{remark}

Before stating the next proposition, we recall the notation \eqref{eq:notationscale}: $W^{1,k}_1(x)= f^{-k}(W^1_1(f^k(x)))$ and $W^{3,k}_1(x) = f^k(W^3_1(f^{-k}))$. 
\begin{prop}\label{prop.estimatepoly2}
 There exists $\delta:=\delta(f, \ell, \mu)>0$ such that for every integer $d_1 > 0$,  every $\eps > 0$ there is  $\cQ \subset M$ with $\mu(\cQ) > 1-\frac{\eps}{10}$
  such that for every $\nu >0$, there is $c:= c(f, \mu, d_1, \eps, \nu)>0$ such that for every $x \in \cQ$, every $k >0$ such that $f^{k}(x) \in \cQ$ and every pair of  polynomials $Q,P$ of degree $\leq d_1$ there is a set $U_{Q,P,x,k} \subset W^{1,k}_1(x)$ such that $\mu_x^1(U_{Q,P,x,k} \cap W^{1,k}_1(x)) >(1-\nu) \mu_{x}^1(W^{1,k}_1(x))$ and 
   \begin{equation}\label{eq:estimatepoly2} 
 \left|\cT^\ell_{x}(t) -\frac{Q(t)}{ P(t)}\right|>  c e^{-\delta k}, \quad \forall t \in (\Phi^1_x)^{-1}(U_{Q,P,x,k}).
 \end{equation}
\end{prop} 

\begin{proof}
Let $\cQ_0$ be a compact set such that $\mu(\cQ_0)> 1-\frac{\eps}{100}$ and  every object we will consider varies continuously as in Proposition \ref{prop.measurable}. 

We apply Proposition \ref{prop.estimatepoly} to $\eps/100$ and get a set $\cQ_1$ verifying Proposition \ref{prop.estimatepoly} in place of $K$ (in particular, $\mu(\cQ_1)>1 - \frac{\eps}{100}$). Then the set $\cQ= \cQ_0 \cap \cQ_1$ satisfies $\mu(\cQ)>1 - \frac{\eps}{10}$. Proposition \ref{prop.estimatepoly} gives a constant $c_0 > 0$ so that for every pair of polynomials $P_0,Q_0$ of degree $\leq d_1$ and a point $x \in M$ with $f^k(x) \in \cQ$, we have that $\hat\mu_{f^k(x)}^1(I^{P_0,Q_0}_{f^k(x), c_0}) < \nu \mu_{f^k(x)}^1(W^1_1(f^k(x)))=
\nu$ (recall that $\mu_{f^{k}(x)}^1$ is of unit mass restricted to $W^1_1(f^{k}(x))$).

Let $d_1$ be sufficiently large so that   $\partial^{\ell+1}_3 F_{x,2}(\cdot,0,0)$ is a polynomial of degree $\leq d_1$ for $\mu$-a.e. $x$.
Now fix some point $x \in \cQ$ such that $f^k(x) \in \cQ$ and polynomials $P,Q$ of degree $\leq d_1$. 
By formula \eqref{eq:dynamicstemplate} and the fact that $f$ has $\ell$-good charts, we see that there is a polynomial of degree $\leq d_1$, denoted by $R$, such that for every $t \in (-(\lambda_{1,x}^{(k)})^{-1}, (\lambda_{1,x}^{(k)})^{-1})$ we have that
\begin{equation} \label{eq tlxtequation}
\cT^\ell_x(t) = \frac{(\lambda_{3,x}^{(k)})^{\ell+1}}{\lambda_{2,x}^{(k)}} \cT^\ell_{f^k(x)}(\lambda_{1,x}^{(k)} t) + R(\lambda_{1,x}^{(k)}t). 
\end{equation}
Therefore, to estimate $ \left|\cT^\ell_{x}(t) -\frac{Q(t)}{ P(t)}\right|$ for $t  \in (-(\lambda_{1,x}^{(k)})^{-1}, (\lambda_{1,x}^{(k)})^{-1})$ it is enough to estimate: 
\begin{equation} \label{eq distancefromtlfkxtorational}
 \left|  \frac{(\lambda_{3,x}^{(k)})^{\ell+1}}{\lambda_{2,x}^{(k)}}(  \cT^\ell_{f^k(x)}(\lambda_{1,x}^{(k)} t) -\frac{ Q_0(t)}{ P_0(t)} ) \right|
\end{equation}
for some polynomials $P_0,  Q_0$ of degree at most $2d_1$. 

We let $U_{P,Q,x,k}$ be the set of points in $W^{1,k}_1(x)$ such that their images under $f^k$ do not belong to $I^{P_0,Q_0}_{f^k(x),c_0}$. Since the measure $\mu$ is invariant, we have that 

$$\frac{\mu_x^1(U_{P,Q,x,k})}{\mu_x^1(W^{1,k}_1(x) )} = 1- \frac{\mu_{f^k(x)}^1(\Phi^1_{f^k(x)}(I^{P_0,Q_0}_{f^k(x),c_0}) )}{\mu_{f_k(x)}^1(W^1_1(f^k(x)) )} \geq 1-\nu. $$

Since $x$ and $f^k(x)$ both belongs to $\cQ$, there exist $\delta = \delta(f, \ell, \mu) >0$ and $c_1 = c_1(f, \cQ) > 0$, so that $ \frac{\lambda_{2,x}^{(k)}}{(\lambda_{3,x}^{(k)})^{\ell+1}} \geq c_1 e^{-\delta k}$.  By \eqref{eq distancefromtlfkxtorational} and \eqref{eq tlxtequation}, we can choose $c = c_0 c_1$ so that \eqref{eq:estimatepoly2} holds for points in $U_{P,Q,x, k}$. 
\end{proof}

\subsection{Proof of Proposition \ref{p.qniellgood}}\label{s.proof6}
To show  that Definition \ref{def.QNI} is verified we will use the equivalent characterization of QNI in Lemma \ref{lem altdefQNI}.  

 Let $V, \alpha > 0$ be two constants, and let $d_1 > 0$ be an integer, chosen depending only on $f, \mu$ at the end of proof. 
We fix an arbitrary constant $\eps>0$.    Using Proposition \ref{prop.measurable}, we choose a compact set $\cP_1 \subset M$ with $ \mu(\cP_1) > 1-\eps/100$ which verifies the following properties:

\begin{enumerate}

\item\label{it.Holdercont} $W^{1}_1(x)$ and $W^{3}_1(x)$ vary H\"{o}lder continuously with respect to $x \in \cP_1$ in the smooth topology (see \cite[\S 8]{BP}); and the chart $\imath_x$ has uniformly bounded smooth norm for all $x \in \cP_1$;

\item   given $\nu>0$, we have that for large enough $j>0$ and for every $x \in \cP_1$ one has $\mu^i_x (W^{i,j}_1(x) \cap \cP_1) > (1- \frac{\nu}{10}) \mu^i_x(W^{i,j}_1(x)) $  for $i\in \{1,3\}$.  

\end{enumerate}

 Consider $\nu_n= 2^{-n}$ and let $\cQ_n$ be the set given by Proposition \ref{prop.estimatepoly2} for the values  $\nu_n$ so that $\mu(\cQ_n) > 1- (\eps/100)2^{-n}$. Consider $\cQ= \cap \cQ_n$ and $\cP_0= \cP_1 \cap \cQ$ which also verifies the previous properties (and $\mu(\cP_0) > 1-\eps$). Moreover, we know that given $\nu>0$ we know that if $x, f^{k}(x) \in \cQ$ then equation \eqref{eq:estimatepoly2} is verified for every rational function $\frac{Q}{P}$ of degree at most $d_1$ with $\delta$ depending only on $f, \ell, \mu$; and $c$ depending only on $f, \ell, \nu, \eps, \mu$.  

We fix some $\nu \in (0, 1)$ from now on. 
In the following, we say that a constant $C$ is uniform if $C > 0$ and it depends only on $f, \mu$ and the sets given above. We will use $c$ to denote a generic uniform constant which may vary from line to line.

We fix an arbitrary $x' \in \cP_0$. 

There is a uniform constant $r_0>0$ such that for $y' =  \Phi^3_{x'}(s) \in W^3_{r_0}(x') \cap \cP_1$, we may write
\begin{equation} \label{eq defQP}
 \imath_{x'}^{-1}(W^1_{loc}(y')) = \{ (t', \hat Q(t'), \hat P(t'))) \ : \ t' \in (-r_0,r_0)\} ,\end{equation} 
\noindent where $\hat Q$ and $\hat P$ are smooth functions (with uniformly bounded derivatives of any given order). Note that $\hat Q(0)=0$ and $\hat P(0)=s \in (-r_0,r_0)$.

The H\"{o}lder condition in (\ref{it.Holdercont}) ensures a property of uniform H\"{o}lder holonomies as in \eqref{eq:holderholon}. Since this is the non-uniform hyperbolic setting, we expand the argument. Note that condition (\ref{it.Holdercont}) says that there exist uniform constants $c_0>0$ and $\gamma_1 \in (0,1]$ so that for all  $y' =  \Phi^3_{x'}(s) \in W^3_{r_0}(x') \cap \cP_1$, we have
\begin{equation}\label{eq:holderPQ} |\hat P'(t')| < c_0 v^{\gamma_1} \text{ and } |\hat Q'(t')| < c_0 v^{\gamma_1} 
\end{equation} \noindent for all $t' \in (-r_0,r_0)$ where $v = \min \{ |\hat P(t')|+|\hat Q(t')| \ : \ t'\in (-r_0,r_0) \text{ s.t. } \Phi_y^1(t') \in \cP_1\} \leq |s| = |\hat P(0)| + |\hat Q(0)|$. Integrating, we get that: 
\begin{equation}\label{eq:holderPQ2} ||\hat P(t')|-|s|| < c_0 v^{\gamma_1}t'  \text{ and } |\hat Q(t')| < c_0 v^{\gamma_1}t'.  
\end{equation}

Now, choosing $\gamma_2 \gg \frac{1}{\gamma_1} + 1$ and some small $c_2>0$ we see that if there is some $t' \in (-r_0,r_0)$ so that $|\hat P(t')| + |\hat Q(t')| \leq c_2 s^{\gamma_2}$ then we will have that $v< s^{\gamma_2}$ and so get that  $|\hat P(0)| + |\hat Q(0)| < s$ which is a contradiction.  This shows that there are $c_1, \gamma_1,c_2, \gamma_2 >0$ so that for $t' \in (-r_0,r_0)$ we have (as in \eqref{eq:holderholon}): 

\begin{equation}\label{eq:pt'} 
 c_2 |s|^{\gamma_2} < |\hat P(t')| + |\hat Q(t')|   < c_1 |s|^{\gamma_1}.
\end{equation}

By making $r_0$ smaller if necessary, for any $t \in (-r_0, r_0)$, we denote $z' = \Phi^1_{x'}(t)$, and we have a well-defined $t'$ as the unique constant depending on $t$ and $y$ such that $(t', \hat P(t') )$ belongs to $\pi_{1,3}( \imath_{x'}^{-1}(W^3_{1}(z')) )$ where $\pi_{1,3}$ is the projection from $\R^3$ to its $1$st and $3$rd coordinates.  By the H\"older condition, we deduce that $|t'| \leq C|t|^{\gamma_3}$ for some $\gamma_3 > 0$ depending only on $f$ and $\mu$. 
We may write
 \begin{equation}\label{eq:W3}  
 \imath_{x'}^{-1}(W^3_{loc}(z')) = \{ (t + a(t)u +  e_{z'}(u),\cT^\ell_{x'}(t)u^{\ell+1}  +  \hat e_{z'}( u) , u) \ : \ u \in (-r_0,r_0)\}. \end{equation}
 Now assume that $t$ is chosen as that $z' \in \cP_1$. Then there is a uniform constant  $c_3>0$ such that $|a(t)| \leq c_3$, $| e_{z'}( u) | \leq c_3 u^2$ and $| \hat e_{z'}( u)  | \leq c_3 u^{\ell+2}$. 
Notice that we may deduce from the above bound that
 \begin{equation}\label{eq:tt}
 |t-t'| < c_3 |\hat P(t')| + c_3 |\hat P(t')|^2 \leq 2 c_3 |\hat P(t')|.
 \end{equation}

\begin{lemma} \label{lem distancebetweenleavesincoordinate}
There is a uniform constant $c_4 > 0$ such that we have
\begin{equation} \label{eq:distance1'} 
d(W^3_{1}(z'), W^1_{1}(y'))  \geq c_4 |\hat P(t')^{\ell+1} \cT^\ell_{x'}(t) - \hat Q(t')| - c_4^{-1} |\hat P(t')|^{\ell+2}.
\end{equation}
\end{lemma}
\begin{proof}
Since $\{ \imath_{x} \}$ is a family of $\ell$-good unstable coordinates, we can see that  
the tangent spaces of the curves $\imath_{x'}^{-1}(W^3_{loc}(z'))$ and $\imath_{x'}^{-1}(W^1_{loc}(y'))$ are both disjoint from 
a closed cone $\{ (v_1, v_2, v_3) : |v_1|+|v_3| \leq c |v_2| \}$  for some constant $c > 0$ independent of all choices of $x' \in \cP_0$, $y', z' \in \cP_1$ given above. This follows from the fact that the manifolds $W^1_{loc}(y')$ and $W^3_{loc}(z')$ have uniformly bounded derivatives because $y',z' \in \cP_1$. The choice of $t'$ is made so that when projecting along the second coordinate we get that the graphs $v \mapsto (v, \hat P(v))$ and $u \mapsto (t + a(t)u +  e_{z'}(u), u)$ intersect exactly at $v=t'$ and $u=\hat P(t')$. The distance between the second coordinates, for values of $v$ and $u$ close to $t'$ and $\hat P(t')$ can vary no more than by $c$ defined above, while the distance between the other coordinates can only increase. 

Thus, by matching the 1st and 3rd coordinates of the expressions in \eqref{eq defQP} and  \eqref{eq:W3}, we have 
\aryst
 d(W^3_{1}(z'), W^1_{1}(y'))  &\geq& c |\hat P(t')^{\ell+1} \cT^\ell_{x'}(t) + \hat e_{z'}(\hat P(t') ) - \hat Q(t')|.
\earyst
By \eqref{eq:W3} and the choices of $t, t'$, we deduce that 
 \ary\label{eq:distance1} 
 c |\hat P(t')^{\ell+1} \cT^\ell_{x'}(t) + \hat e_{z'}(\hat P(t') ) - \hat Q(t')|  
&\geq&c_4 |\hat P(t')^{\ell+1} \cT^\ell_{x'}(t) - \hat Q(t')| -  c_4^{-1} |\hat P(t')|^{\ell+2} \nonumber
 \eary
\noindent for some uniform constant $c_4 > 0$.  This concludes the proof.
\end{proof}

We first consider the case where $|\hat P(t')| < c_2|s|^{\gamma_2}/2$. In this case, by \eqref{eq:pt'}, we have 
\ary \label{eq lowerboundhatQ}
|\hat Q(t')| > c_2|s|^{\gamma_2}/2.
\eary
Then by Lemma \ref{lem distancebetweenleavesincoordinate}   and  by reducing the size of $r_0$ if necessary\footnote{Note that if $\ell=0$ we need to change slightly the constants for this to work and choose, for instance, $|\hat P(t')| < \frac{\hat D c_2 c_4^2 }{10}|s|^{\gamma_2}$ where $\hat D$ is a uniform bound for $\cT^0_x$. But the argument is the same: the point is to treat one the case when $\hat P$ is small (and therefore $\hat Q$ is big) and the other when $\hat P$ it is uniformly bounded from below, so that we can control the quantity in equation \eqref{eq:taylor}.},  we have
\ary \label{eq:distance111}
\ \ \mbox{RHS of } \eqref{eq:distance1'} &\geq& c_4 |\hat Q(t')| - c_4^{-1} |\hat P(t')^{\ell+1} \cT^\ell_{x'}(t)| - c_4^{-1} |\hat P(t')|^{\ell+2} \\
&\geq& c_4 c_2|s|^{\gamma_2}/2 - c c_4^{-1} (c_2|s|^{\gamma_2}/2)^{\ell+1} \nonumber \\
&\geq& c_4 c_2|s|^{\gamma_2}/4. \nonumber
\eary

Now it remains to consider the case where  $|\hat P(t')| \geq c_2 |s|^{\gamma_2}/2$. Then we have
\begin{equation}\label{eq:pt} 
 c_2 |s|^{\gamma_2}/2 \leq |\hat P(t')|  < c_1 |s|^{\gamma_1}.
\end{equation}

We let $d_1$ be large depending only on $f$, $\mu$ and $\ell$.
Fix some $s \in ( - r_0, r_0 )$. 
Denote
$r_s = |s|^{C_*}$
where
 \ary \label{eq defCstar}
 C_* = \frac{2 \gamma_2 (2 \ell + 3)}{ \gamma_3(d_1 + 1)}.
 \eary
 
Now we fix an arbitrary $t \in (-r_s, r_s)$ such that $z' \in \cP_1$.
Then we have $ |t'|  \leq Cr_s^{\gamma_3}$, and
\begin{align}
\mbox{RHS of } \eqref{eq:distance1'} \geq c_4 |\hat P(t')|^{\ell+1} \left| \cT^\ell_{x'}(t) - \frac{\hat Q(t')}{\hat P(t')^{\ell + 1}} \right|  - c_4^{-1} |\hat P(t')|^{\ell+2}.  
\end{align}
Since $y' \in \cP_1$, there exists a uniform constant $c_5>0$ so that Taylor's expansion gives $\hat Q(\tau)= Q(\tau) + q(\tau)$ and $\hat P^{\ell+1}(\tau)= P(\tau) + p(\tau)$ such that $|q(\tau)|, |p(\tau)| < c_5 |\tau|^{d_1+1}$, and $P,Q$ are  polynomials of degree $\leq d_1$. 
 By $C_* >  \frac{2\gamma_2(\ell + 1)}{ \gamma_3 (d_1 + 1)}$ and by reducing $r_0$ if necessary,  we deduce $|P(t')| \geq |\hat P(t')|^{\ell+1}/2 > 0$ by \eqref{eq:pt} and $|t'| \leq Cr_s^{\gamma_3}$. Enlarging $c_5$ if necessary, the function $\beta_5(t') := | \frac{\hat Q(t')}{\hat P(t')^{\ell+1}} - \frac{Q(t')}{P(t')} |$ satisfies that
 \begin{equation}\label{eq:taylor}
\beta_5(t') \leq  2c_1 c_5 \frac{ |t'|^{d_1+1}   |s|^{\gamma_1}  }{|\hat P(t')|^{ 2 (\ell+1 ) \clb}}.   
 \end{equation}
 
Denote $R(t) = \frac{Q(t)}{P(t)} \in {\cR}^{d_1}$. We have $R(0) = 0$.  By  \eqref{eq:distance1'}  and our choice of $Q,P$, we get 
\ary
  d(W^3_{1}(z'), W^1_{1}(y')) \geq c_4  |\hat P(t')|^{\ell+1} | \cT^\ell_{x'}(t) - R(t') | -  c_4^{-1} |\hat P(t')|^{\ell+2} - c_4^{-1} |\hat P(t')|^{\ell+1}   \beta_5(t).   \nonumber  
 \eary
 Then by \eqref{eq:taylor} and by reducing $c_4$ if necessary, the distance $d(W^3_{1}(z'), W^1_{1}(y'))$  is bounded from below by 
 \ary
&& c_4  |\hat P(t')|^{\ell+1} | \cT^\ell_{x'}(t) - R(t') | - 2c_1 c_5 c_4^{-1}  |\hat P(t')|^{- (\ell+1)}  |t'|^{d_1+1}  |s|^{\gamma_1}  \nonumber \\
&&  \clblue  -  \clb c_4^{-1} |\hat P(t')|^{\ell+2}    \nonumber  \\
&\geq& c_4   |\hat P(t')|^{\ell+1} | \cT^\ell_{x'}(t) - R(t') | - c c_4^{-1} |\hat P(t')|^{\ell+2}.  \label{eq:distance2} 
\eary
 The last inequality above follows from  \eqref{eq:pt}, \eqref{eq defCstar} and $ |t'| \leq Cr_s^{\gamma_3}$.

Recall that $\cQ= \cap \cQ_n$ is the set defined at the beginning and that $\nu > 0$ is a small constant also fixed at the beginning of the proof. We have the following.

\begin{claim}\label{l.bound}
There exist constants  $V_0, m_0 > 0$, $d_1 = O(\ell)$ and $\alpha_0,C_0>0$ such that the following is true. Given any $m \geq m_0$, denote by $I^{(m)}= (- (\lambda^{(m)}_{1, x'})^{-1}, (\lambda^{(m)}_{1, x'})^{-1} )$. Then, if $x' \in \cQ$ is such that $f^m(x') \in \cQ$, and  $s \in (-r_0, r_0)$  is such that
\begin{equation}\label{eq:conditionQNI}
\frac{ - \log s}{\log{ \lambda^{(m)}_{1, x'}}} \in  \left[\frac{3}{5} V_0, \frac{5}{3}V_0\right],
\end{equation}
and  $y' = \Phi^3_{x'}(s) \in \cP_1 $, there is a subset $U_{y'}$ of  $W^{1,m}_1(x')  =\Phi_{x'}^1(I^{(m)})$  such that $\mu^1_{x'}(U_{y'}) > (1-\nu) \mu^1_{x'}(W^{1,m}_1(x'))$, and for any $z' = \Phi^1_{x'}(t) \in U_{y'}$ we have
 \begin{equation}\label{eq:distanceQNI} 
 d(W^3_{1}(z'), W^1_{1}(y')) > C_0 e^{-\alpha_0 m}. 
 \end{equation}
 \end{claim}

\begin{proof}

Given  some $s \in (-r_0,r_0)$  with $\Phi^3_{x'}(s) \in \cP_1$ we can define the functions $\hat P$ and $\hat Q$ as in equation \eqref{eq defQP}. Note that the functions $\hat P$ and $\hat Q$ are well defined as longs as $y' \in \cP_1$ so the rest of the constructions can be made. 

We will fix $V_0 > \frac{2}{\gamma_1}$ and $d_1 > 10 V_0 \gamma_3^{-1} \gamma_2(\ell+1)$. Note that this will ensure that $C_\ast < \frac{3}{5 V_0}$ from our choice of $C_\ast$. 

We will consider $C_0$ sufficiently small and $\alpha_0, m_0$ sufficiently large verifying some conditions that will be explicit in the proof.  For a given $m>m_0$ and $s$ verifying  \eqref{eq:conditionQNI} and $\Phi^3_{x'}(s) \in \cP_1$, we consider $U_s$ to be the set of points  for  which \eqref{eq:distanceQNI} holds. 

We will divide the set $I^{(m)} = I^{(m)}_> \cup I^{(m)}_{<}$ where 
\begin{itemize}
\item $t' \in I^{(m)}_{>}$ if  $|\hat P(t')| > c_2|s|^{\gamma_2}/2$  (cf. equation \eqref{eq:pt}) and,
\item $t' \in I^{(m)}_{<}$ if  $|\hat P(t')| \leq c_2|s|^{\gamma_2}/2$. 
\end{itemize}

Note that if we have  $|\hat P(t')| \leq c_2|s|^{\gamma_2}/2$ (i.e. $t' \in I^{(m)}_{<}$), then by  \eqref{eq:distance1'} and \eqref{eq:distance111}, we have
\aryst
 d(W^3_{1}(z'), W^1_{1}(y'))  \geq c_2 c_4 |s|^{\gamma_2}/4. 
\earyst
We can then deduce \eqref{eq:distanceQNI} with appropriate $C_0$ , $\alpha_0$ and $m_0$ for all $t' \in I^{(m)}_{<}$. Thus, if $C_0$ is sufficiently small and $\alpha_0$ sufficiently large and $m$ sufficiently large, we can consider $I^{(m)}_{<}$ to be  fully contained in $U_s$.

We will now deal with $I^{(m)}_{>}$ and show that for an appropriate choice of $C_0$, $\alpha_0$, if $m$ is large we get that $U_s \cap I^{(m)}_{>}$ covers $I^{(m)}_{>}$ except for a subset whose measure is at most $\nu \hat \mu_{x'}^1(I^{(m)})$. 

There exists a uniform $\sigma>0$ (independent on $m$) so that for every subset $T \subset I^{(m)}$ with $\hat \mu_{x'}^1 (T) \geq (1-\nu/2) \hat \mu_{x'}^1(I^{(m)})$ it verifies that $T$ is $(d_1,\sigma,\nu/2)$-spread in $I^{(m)}$. (See the proof of Proposition \ref{prop.estimatepoly} for a similar argument.)

By  Proposition \ref{p.compactnessrational} (see Remark \ref{remark-rescale})  there exists $C:= C(d_1,\sigma, \nu/2) > 1$ such that if $\hat R \in \mathrm{Rat}^{d_1}$ and $D= \sup_{t \in I^{(m)}} |\hat R(t)|$ then $|\hat R'(\tau)| \leq C  D  \lambda^{(m)}_{1, x'}$ and $|\hat R(\tau)| > D/C$ for every $\tau \in I^{(m)} \setminus \cup I_i$ where $I_i \Subset \hat I_i$, $0 \leq i \leq 3d$, are open subintervals of $I^{(m)}$ such that  $\hat I_0, \cdots, \hat I_{3d}$ are mutually disjoint, whose union is of $\hat\mu^1_x$-measure at most $\frac{\nu}{2}\hat\mu^1_{x'}(I^{(m)})$.
We may choose the intervals so that   $\hat I_i$ contains the $\kappa|I^{(m)}|$-neighborhood of $I_i$, where $\kappa > 0$ depends only on $\mu$, $f$ and $\cQ$, but is independent of $i$, $x$ and $m$: 
the existence of such $\kappa$ is guaranteed by the fact that $\hat\mu^1_x|_{(-1,1)}$ depends continuous on $x \in \cQ$.

Let the rational function $R$ be constructed as before so that \eqref{eq:distance2} holds. Consider  $D = \sup_{t \in I^{(m)}} |R(t)|$ and $\hat D =  \sup_{t \in (-1,1)} |\cT_{x'}^\ell(t)|$.   

Assume first that $D \geq C \hat D +1$. For every $t' \in I^{(m)} \setminus \bigcup \hat I_i$ we have that $|R(t') - \cT_{x'}^\ell(t)| >1$.
We can without loss of generality assume that $t' \in I^{(m)}_{>}$ since we already know that $I^{(m)}_{<} \subset U_s$ for well-chosen values of the constants.   
Therefore by equation \eqref{eq:distance2} we have that: 
$$ d(W^3_1(z'), W^1_1(y') \geq \frac{c_4}{2} |\hat P(t')|^{\ell+1} \geq C_0 e^{-\alpha_0 m}, $$
where $C_0 \ll c_4 c_2^{\ell+1}$ and $\alpha_0$ and $m_0$ are large enough so that $|s|^{\gamma_2(\ell+1)} \geq e^{-\alpha_0 m}$ if $m \geq m_0$. So in this case it holds that $U_s$ contains $I^{(m)} \setminus \bigcup_i \hat I_i$. 

We can therefore assume from now on that $D \leq C \hat D +1$. 

By Proposition \ref{prop.estimatepoly2} and our choice of $\nu$, there exist $\delta = \delta(f, \ell, \mu) > 0$ and  a subset $U'_s \subset I^{(m)}$ such that  (as long as $m_0$ is sufficiently large) $\hat \mu^1_{x'}(U'_{s}) >(1-\nu/10) \mu_{x'}^1(I^{(m)})$  and if $t \in ( \Phi^1_{x'})^{-1}(U'_s)$, then  
$$ |\cT^\ell_{x'}(t) - R(t)| > c e^{-\delta m}. $$

We claim that $U_s$ contains $U'_s \cap (I^{(m)} \setminus \bigcup \hat I_i)$ for well chosen values of $C_0$, $\alpha_0$ and sufficiently large $m$. 

Fix an arbitrary  $t \in U'_s \cap (I^{(m)} \setminus \bigcup \hat I_i)$.
By \eqref{eq:tt}, \eqref{eq:pt} and \eqref{eq:conditionQNI}, we have 
\aryst
|t - t'| \leq 2c_3c_1 |s|^{\gamma_1} < 2c_3c_1 (\lambda_{1, x'}^{(m)})^{- \frac{3}{5}\gamma_1 V_0}.
\earyst
 By the choice of intervals $I_i \Subset \hat I_i$, and by letting $V_0$ be sufficiently bigger,   we see that $t$ and $t'$ belong to the same component of $I^{(m)} \setminus \bigcup_i I_i$ (in particular, $t' \in I^{(m)} \setminus \bigcup_i I_i$). Consequently we have 
\ary \label{eq smallRcase}
|R(t) - R(t')| \leq CD \lambda^{(m)}_{1, x'} |t-t'| \leq c_6 \lambda^{(m)}_{1, x'}   |\hat P(t')|
\eary
for some uniform constant $c_6 > 0$ (note that here we used that $D$ is uniformly bounded). 

 Putting together  \eqref{eq:distance2}, \eqref{eq smallRcase} and $t \in U'_s \cap (I^{(m)} \setminus \bigcup \hat I_i)$,  we see that there is a constant $\beta_7(t)$ and a uniform constant $c_7>0$ such that 
\ary\label{eq:distance3}
\begin{aligned}  d(W^3_{1}(z'), W^1_{1}(y')) \geq c_4 |\hat P(t')|^{\ell+1} | \cT^\ell_{x'}(t) - R(t) | - \beta_7(t) \end{aligned}
\eary
where 
\aryst
|\beta_7(t)| &\leq& c_4^{-1} (|\hat P(t')|^{\ell+1}|R(t) - R(t')| +  |\hat P(t')|^{\ell+2}) \\
&\leq&  c_7 \lambda^{(m)}_{1, x'}|\hat P(t')|^{\ell+2}.
\earyst

By  \eqref{eq:pt}, we have
\aryst
|\hat P(t')|^{\ell+1} > c_2^{\ell+1} |s|^{\gamma_2 (\ell+1)} \ \text{ and } \ |\hat P(t')| \leq c_1 |s|^{\gamma_1}. 
\earyst

By \eqref{eq:distance3} we deduce that 
\ary
  d(W^3_{1}(z'), W^1_{1}(y'))  &\geq& c_4 |\hat P(t')|^{\ell+1} \left( c e^{-\delta m}  - \frac{|\beta_7(t)|}{|\hat P(t')|^{\ell+1}}\right) \nonumber \\
 &\geq& c_4 |\hat P(t')|^{\ell+1} \left( c e^{-\delta m}  -  c_7 \lambda^{(m)}_{1, x'}  |\hat P(t') | \right) \nonumber \\
 &\geq& c_8 |s|^{\gamma_2(\ell + 1)}  \left( c_8 e^{-\delta m}  -  c_8^{-1} \lambda^{(m)}_{1, x'} |s|^{\gamma_1}  \right) \label{eq lowerboundfordw3w11}
  \eary
  for some uniform constant $c_8 > 0$.
  
We fix a large constant $m_0 > 0$  such that  for every $m > m_0$, and every $|s| < (\lambda^{(m)}_{1, x'})^{- 3 V_0 / 5}$,  we have 
\aryst
\lambda^{(m)}_{1, x'} |s|^{\gamma_1} < \frac{c_8^2}{2} e^{- \delta m}.
\earyst
Then for every $s$ satisfying \eqref{eq:conditionQNI}, we have
\aryst
  d(W^3_{1}(z'), W^1_{1}(y'))  &\geq& \frac{c_8^2}{2} e^{-\delta m} (\lambda^{(m)}_{1, x'})^{-   5 V_0 \gamma_2 (\ell + 1)/3}.
\earyst

\noindent which gives \eqref{eq:distanceQNI} for $C_0< \frac{c_8^2}{2}$ and $\alpha_0$ so that $e^{-\delta m} (\lambda^{(m)}_{1, x'})^{-   5 V_0 \gamma_2 (\ell + 1)/3} \geq e^{-\alpha_0 m}$. 

Notice that we have
\aryst
C_* \log |s| \geq -\frac{5 C_*}{3} V_0 \log  \lambda^{(m)}_{1, x'} \geq - \log \lambda^{(m)}_{1, x'}  \geq \log |t|,
\earyst
\noindent which is ensured by our choice of $V_0$ and $d_1$.    This gives us the hypothesis $|t| < r_s$, on which the estimate  \eqref{eq:distance2}  is based. 

This shows that $U_s$ contains $U'_s \cap (I^{(m)} \setminus \bigcup \hat I_i)$ for well chosen values of $C_0, \alpha_0,m_0$ and thus completes the proof.   
\end{proof}

Note that Claim \ref{l.bound} has put us under the conditions of Lemma \ref{lem altdefQNI} from which we can deduce that QNI is verified for $f$. Indeed, let $V_0, \alpha_0,$ be given by  Claim \ref{l.bound}. Then by slightly reducing the size of $\cP_0$ if necessary, and by letting integer $k_0 \geq m_0$ be sufficiently large depending only on $f, \mu$, we may assume that for any $x \in \cP_0$ and any $k > k_0$, we have 
$k^{-1}\log \lambda^{(k)}_{3, f^{-k}(x)} \in ( \frac{99}{100} \chi_3,  \frac{100}{99} \chi_3)$ and $k^{-1}\log \lambda^{(k)}_{1, x} \in ( \frac{99}{100} \chi_1,  \frac{100}{99} \chi_1)$. We set $V = - \chi_1 V_0/ \chi_3$ and  $\alpha = \alpha_0$. Then for any integers $k_1, k_2 \geq k_0$ such that $\frac{k_2}{k_1} \in (\frac{2}{3}V, \frac{3}{2}V)$ we choose  $S_x = W^{3,k_2}_1(x) \cap \cP_1$. Then for each $y = \Phi^3_{x}(s) \in S_x$, we have
\eqref{eq:conditionQNI} for $m = k_1$.
 We set $U_y =  \Phi^1_x(U_s)$ where $U_s$ is given by Claim \ref{l.bound}. Then we can see that the conditions of Lemma \ref{lem altdefQNI} is satisfied.

 
  \section{Compatibility of good charts or QNI: Proof of Theorem \ref{teo.maintechnical2}}\label{s.compatible}

Throughout this section, we let $\mu$ be a partially hyperbolic measure of $f$, which admits $L$-good stable charts $\{\imath_x\}_{x \in M}$ and $L$-good unstable charts $\{\imath'_x\}_{x\in M}$ for some large integer $L$, which will be determined later depending on $f, \mu$ and $\ell$. We fix a subset $\Omega \subset M$ with full measure so that $W^1_1(x)$, $W^3_1(x)$, $\imath_x$ and  $\imath'_x$ are defined for every $x \in \Omega$.

To facilitate the proof, we introduce the following notation. We denote by $T_0$ the hyperplane $\{(t_1, t_2, t_3) : t_2 = 0\}$. Given a function $\phi: (-1,1)^2 \to \R$, we denote by $\tau_\phi: \R^3 \to \R^3$ the diffeomorphism $\tau_\phi(x, y, z) = (x, y + \phi(x, z), z)$. 
We define $T_{\phi} = \tau_\phi(T_0)$.

Given $x \in \Omega$. We define 
\aryst
S_{1,x} = \imath_x^{-1} \big ( \bigcup_{\substack{t \in (-1,1) \\ \Phi^3_x(t) \in \Omega  }} W^1_1(\Phi^3_x(t)) \big ).
\earyst
In the following, we say that $S_{1,x}$ and $T_\phi$ (for some function $\phi$) are {\it tangent to order $L$ on a subset} $U \subset W^3_1(x)$ if there exists $C > 0$ depending on $f, \mu, x$ and $U$ such that for any $t$ with $\Phi^3_x(t) \in U$ we have 
\begin{eqnarray} \label{eq iotaxtangentalongu1x}
 \ \ \  \tau_{\phi}^{-1} \imath_x^{-1}(W^1_{loc}(\Phi^3_x(t))) = \{ (s, O(C|s|^{L}) ,t + O(C s))  : s \in (-1,1) \}.
\end{eqnarray}
Similarly, we define
\aryst
S'_{3,x} = (\imath'_x)^{-1}\big (\bigcup_{\substack{t \in (-1,1) \\ \Phi^1_x(t) \in \Omega  }} W^3_1(\Phi^1_x(t)) \big ),
\earyst
and say that $S'_{3,x}$ and  $T_\phi$  are tangent to order $L$ on a subset $U' \subset W^1_1(x)$ if there exists $C > 0$ depending on $f$, $\mu$, $x$ and $U'$ such that for any $t$ with $\Phi^1_x(t) \in U'$ we have
\begin{eqnarray} \label{eq iotaxtangentalongu1x2}
 \ \ \ \ \  \tau_{\phi}^{-1} (\imath'_x)^{-1}(W^3_{loc}(\Phi^1_x(t))) = \{ (t + O(C |s|), O(C|s|^{L}) ,s)  : s \in (-1,1) \}.
\end{eqnarray}

Given a $\mu$-typical $x \in \Omega$, the smooth surface  $(\imath'_x)^{-1} \circ \imath_x(T_0)$ contains a graph of a function  $\psi_x:   (- r_x, r_x)^2  \to \R$ for some $r_x \in (0, 1)$.   
 We also denote 
\aryst
I_x =  \{ (a,b) \in \N^2 : a + b \leq 2\ell \mbox{ and }  \partial_1^a \partial_3^b \psi_x(0,0) \neq 0 \}.
\earyst
By definition, $\partial^k_1 \psi_x(0,0) = \partial^k_3 \psi_x(0,0) = 0$ for every integer $k \geq 0$, and consequently we have
\ary \label{eq Ixstructure}
( \{(0, i) : i \geq 0 \}  \cup \{(i, 0) : i \geq 0 \}  ) \cap I_x = \emptyset.
\eary
We have the following.
\begin{lemma} \label{lem Ixisempty} 

For $\mu$-a.e. $x \in M$,
if $I_x  = \emptyset$ then we have $I_{f(x)} = \emptyset$, and  \eqref{eq:compatible} holds at $x$.
\end{lemma}
\begin{proof}
 Without loss of generality, we may assume that $x \in M$ satisfies that $I_x = \emptyset$ and $\psi_x$ is defined.
We can deduce \eqref{eq:compatible} from Taylor's expansion.

Assume to the contrary that $I_{f(x)} \neq \emptyset$.
  By letting $L > 2\ell$ among other things, and 
 by Lemma \ref{lem uniformboundab} and and by restricting $x$ to a $\mu$-conull subset, there exists $C_x > 0$ such that for every $\eps > 0$ there exist $t_1, t_3$ with $|t_1|, |t_3| < \eps$ satisfying
\aryst  
d(W^3_1(\Phi^1_x(t_1)), W^1_1(\Phi^3_x(t_3))) \leq C_x \eps^{2\ell +1}
\earyst
and
\aryst  
d(f(W^3_1(\Phi^1_x(t_1))),  f(W^1_1(\Phi^3_x(t_3))) ) \geq \eps^{2\ell }/C_x.
\earyst
We obtain a contradiction by letting $\eps$ be sufficiently small. Consequently, we deduce that $I_{f(x)} = \emptyset$.
\end{proof}

The main result of this section is the following.
 \begin{prop} \label{prop IxnonemptytoQNI}
Given an integer $\ell > 0$ large, there exists $L = L(\mu, f, \ell ) > 0$ such that the following is true. Assume that there is a set $\cP_0 \subset M$ with $\mu(\cP_0)  > 0$ such that for any $x \in \cP_0$ we have $I_x  \neq \emptyset$. Then $\mu$ has QNI. 
 \end{prop}

\begin{proof} 
By Lemma \ref{lem Ixisempty}, the set of $x$ such that $I_x \neq \emptyset$ is $f$-invariant. Then by ergodicity we may assume without loss of generality that $\mu(\cP_0) = 1$.

By Pesin's theory, there is a constant $\delta > 0$, depending only on $f, \mu$, such that 
for any $\varepsilon > 0$, there is a compact set $\cP_{\varepsilon} \subset \Omega$ with $\mu(\cP_{\varepsilon}) \geq 1 - \varepsilon$ such that $E^1$ and $E^3$ are uniformly $\delta$-H\"older continuous on $\cP_{\varepsilon}$.
 
 By  \eqref{eq Ixstructure},  
 there are numbers $ \frac{- 10 \chi_1 \ell^{2}  }{\chi_3  \delta^{2}}  > V >    \frac{- 10 \chi_1 }{\chi_3  \delta^{2}} $   and $K > 1$ such that for any $V' \in  (\frac{- \chi_3}{4 \chi_1}V\delta^2,  \frac{- 4 \chi_3}{ \chi_1}V\delta^{-2})$, for any $x \in \cP_0$, the set $\{ a + b V'	 :  (a,b) \in I_x \}$ admits a unique minimum $K(V', x) \leq K$. 
 By the choices of $V$ and $K$, we may assume that there exists a measurable positive function $x \mapsto c_x$ such that  for every $x \in \cP_0$,  for any $s_1, s_3 \in (- c_x, c_x) \setminus \{0\}$ with $\frac{\log |s_3|}{ \log |s_1|} \in (\frac{- \chi_3}{4 \chi_1}V\delta^2,  \frac{- 4 \chi_3}{ \chi_1}V\delta^{-2})$, we have 
\begin{eqnarray} \label{eq lowerboundofpsix}
|\psi_x(s_1, s_3)| \geq c_x |s_1|^K.
\end{eqnarray} 

We fix a small constant $\varepsilon  > 0$.  

Let $C_1 > 1$ be a large constant to be determined in due course.
By Lusin's theorem and by enlarging $C_1$ if necessary, we  may take a compact subset $B_0 \subset \cP_0$  with 
\begin{eqnarray}
\mu(B_0) > 1 - \eps/2, 
\end{eqnarray} 
satisfying the following properties:

\enmt
\item  we have
\begin{eqnarray}
r_x, c_x > C_1^{-1}, \quad x \in B_0;
\end{eqnarray}

\item the smooth norms of the charts $\imath_x$ and $\imath'_x$ are bounded by $C_1$ whenever $x \in B_0$;

\item for any $n \in \mathbb{Z}$, for any $i\in \{1,3\}$ and any $x \in B_0$ we have
\begin{eqnarray}  
C_1^{-1} e^{n(\chi_i - \epsilon)} < \| D_x f^n|_{E^i(x)} \|   < C_1 e^{n(\chi_i + \epsilon)};
\end{eqnarray}

\item $E^1$ and $E^3$ are uniformly $\delta$-H\"older continuous on $B_0$, with  $\delta$-H\"older norms bounded by $C_1$.
 
\eenmt

By Proposition \ref{prop.measurable}, there is a compact subset $B \subset B_0$ with \begin{eqnarray} \label{eq mumeasureofb}
 \mu(B) > 1 -  \eps, 
\end{eqnarray}
such that  the following holds:
for every $\nu > 0$, 
 there exist $m_0 > 0$ such that for every $x \in B$ and every $m > m_0$ such that $f^m(x) \in B_0$  there exist a subset $U^{1, m}_x \subset B_0 \cap W^{1, m}_{1}(x)$ such that 
\begin{eqnarray} \label{eq conditionalmeasureofintersectionwithb01}
\mu^1_x(U^{1, m}_x) > (1-\nu) \mu^1_x(W^{1 , m}_{1}(x)),
\end{eqnarray}
and $S'_{3,x}$, $T_0$  are tangent to order $L$ on $U^{1, m}_x$;
and a subset $U^{3, m}_x \subset B_0 \cap W^{3, m}_{1}(x)$ such that 
\begin{eqnarray} \label{eq conditionalmeasureofintersectionwithb03}
\mu^3_x(U^{3, m}_x) > (1-\nu) \mu^3_x(W^{3, m}_{1}(x)),
\end{eqnarray}
and  $S_{1,x}$, $T_0$ are tangent to order $L$ on $U^{3, m}_x$. Moreover, by Lemma \ref{lem uniformboundab}, we may assume that the implicit constants for the above tangencies are uniformly bounded.

Let us denote 
\begin{eqnarray}
S'_{1,x} = (\imath'_x)^{-1} \circ \imath_x(S_{1, x}). 
\end{eqnarray}
 Then $S'_{1,x}$ and $(\imath'_x)^{-1} \circ \imath_x(T_0) = T_{\psi_x}$ are  tangent  to order $L$ on $U^{3, m}_x$.  By  enlarging $C_1$ if necessary, we  may assume that 
 \enmt
\item[(v)] \eqref{eq iotaxtangentalongu1x} holds for $C = C_1$ whenever $x \in B$, $U = U^{3, m}_x$ and $\phi = \psi_x$; and  \eqref{eq iotaxtangentalongu1x2} holds for $C = C_1$ whenever $x \in B$, $U' = U^{1, m}_x$ and $\phi = 0$.
\eenmt
We may assume without loss of generality that $U^{3, m}_x$, resp. $U^{1, m}_x$, is disjoint from $W^{1, m}_{r}(x)$, resp. $W^{3, m}_{r}(x)$, for some $r = r(f, \mu, \nu, \varepsilon) > 0$.

Now take an arbitrary $x \in B$ and two large integers $k_1, k_2 > m_0$ such that $f^{k_1}(x), f^{-k_2}(x) \in B$ and
\begin{eqnarray} \label{eqk1k2}
\frac{k_2 }{k_1 } \in \left(\frac{2}{3}V, \frac{3}{2}V \right). 
\end{eqnarray}

Let us now suppose that $t_1, t_3  \in (- C_1^{-1},  C_1^{-1}) \setminus \{0\}$  satisfy that  
\begin{eqnarray} \label{eq t3inW3}
\Phi^3_x(t_3)  \in \widetilde U^3_x := U^{3, k_2}_{x} \subset W^{3, k_2 }_{1}(x)
\end{eqnarray} 
and 
\begin{eqnarray} \label{eq t1inW1}
\Phi^1_x(t_1) \in \widetilde U^1_x :=  U^{1, k_1}_{x} \subset W^{1,k_1}_{1}(x). 
\end{eqnarray}
By \eqref{eq conditionalmeasureofintersectionwithb01} and \eqref{eq conditionalmeasureofintersectionwithb03}, we have
\begin{eqnarray} \label{eq sxuxlarge}
\frac{\mu^1_x(\widetilde U^1_x)}{\mu^1_x(W^{1,k_1 }_{1}(x))}, \frac{\mu^3_x( \widetilde U^3_x)}{\mu^3_x(W^{3,k_2 }_{1}(x))} > 1- \nu. 
\end{eqnarray}
Thus, by \eqref{eqk1k2} and by enlarging $k_1,k_2$ if necessary, we may assume that for any $t_1, t_3$ satisfying \eqref{eq t3inW3} and \eqref{eq t1inW1}, the following also holds:
\begin{eqnarray} \label{eq logt3logt1} 
\frac{\log |t_3|}{\log |t_1|} \in (\frac{-3\chi_3}{5\chi_1}V, \frac{-5\chi_3}{3\chi_1}V).
\end{eqnarray}

Recall that  $\pi_{1,3}: (-1,1)^3 \to (-1,1)^2$ denotes the projection to the first and the third coordinates.
Consider the curves $\gamma_3 = (\imath'_x)^{-1}(W^3_1(\Phi^1_x(t_1)))$ and $\gamma_1 = (\imath'_x)^{-1} (W^1_1(\Phi^3_x(t_3)))$. 
By enlarging $k_1$ (and $k_2$ at the same time), we can ensure that $\pi_{1,3}(\gamma_3)$ and $\pi_{1,3}(\gamma_1)$ have a unique intersection $(s_1, s_3)$. In other words, there exist $r_1, r_3 \in \R$ such that
\begin{align}
(s_1, r_1, s_3) \in S'_{1,x},\quad (s_1, r_3, s_3) \in S'_{3,x}.
\end{align}
We denote
\begin{eqnarray} \label{eq r'1}
r'_1 = \psi_x(s_1, s_3).
\end{eqnarray}
By definition, we have $(s_1, r'_1, s_3) \in T_{\psi_x}$.
 By the tangency between $S'_{1,x}$ and $(\imath'_x)^{-1} \circ \imath_x(T_0) = T_{\psi_x}$  on  $\widetilde U^3_{x}$; and the tangency between $S'_{3, x}$ and $T_0$ on $\widetilde U^1_{x}$, we have  
\begin{eqnarray} 
|r'_1 - r_1| \leq C_1 |s_1|^{L}, \quad |r_3| \leq C_1  |s_3|^{L}. \label{eqr'r1r3}
\end{eqnarray}
Moreover,  by the $\delta$-H\"older continuity of $E^1$ and $E^3$ on $B_0$,  as in \eqref{eq:pt'} we have
\begin{align} \label{eq ratio t1t3}
C_1^{-1}|t_1|^{1/\delta} \leq |s_1| + |r_1| \leq C_1|t_1|^\delta, 
\quad 
|s_3| \leq C_1 |t_3|^\delta.
\end{align}

Let us first assume that $|r_1| \geq (2C_1)^{-1}|t_1|^{1/\delta}$. Then by  \eqref{eq ratio t1t3}, the second inequality in \eqref{eqr'r1r3} and a similar argument as in Lemma \ref{lem distancebetweenleavesincoordinate}, we deduce that
\aryst
  d(W^3_{1}(\Phi^1_x(t_1)), W^1_{1}(\Phi^3_x(t_3)))  &\geq& C^{-1}|r_1 - r_3| \\
  &\geq& C^{-1}|r_1| - C^{-  1}|r_3| \\
  &\geq& C^{-1}(2C_1)^{-1}|t_1|^{1/\delta}- 2C  C_1| t_3|^{L \delta}.
\earyst
By  \eqref{eq logt3logt1}, and by assuming that
\aryst
L >  \frac{- 4 \chi_1}{\chi_3 \delta^{2}  V},
\earyst
we have
\begin{eqnarray}
  d(W^3_{1}(\Phi^1_x(t_1)), W^1_{1}(\Phi^3_x(t_3))) > \frac{1}{2}C^{-1}C_1^{-1}|t_1|^{1/\delta}.
\end{eqnarray} 

Now we assume that $|r_1| < (2C_1)^{-1}|t_1|^{1/\delta}$. In this case we have that 
\aryst
 |s_1| >  (2C_1)^{-1}|t_1|^{1/\delta}.
\earyst
Then, together with \eqref{eq ratio t1t3} and \eqref{eq logt3logt1}, we deduce that
\begin{align}
\frac{\log |s_3|}{\log |s_1|} \in (\frac{- \chi_3}{4 \chi_1}V\delta^2,  \frac{- 4 \chi_3}{ \chi_1}V\delta^{-2}).
\end{align}
In particular, we have $|s_3| < |s_1|$.
By our choice of $V$, by \eqref{eq lowerboundofpsix}, \eqref{eq r'1}, and by enlarging $C_1$ if necessary, we have
\begin{align}
|r'_1| > C_1^{-1} |s_1|^K >  C^{-1}C_1^{- K -1}|t_1|^{K/\delta}.
\end{align}
Thus we have
\aryst
  d(W^3_{1}(\Phi^1_x(t_1)), W^1_{1}(\Phi^3_x(t_3)))  &\geq& C^{-1}|r_1 - r_3| \\
  &\geq& C^{-1}|r'_1| - C^{-1}|r'_1 - r_1| - C^{-  1}|r_3| \\
  &\geq& C^{-1}C_1^{-K-1} |t_1|^{K/\delta} - C C_1(|s_1|^{L} + |s_3|^{L}) \\
  &\geq& C^{-1}C_1^{-K-1} |t_1|^{K/\delta} - 2C  C_1| t_1|^{L \delta}.
\earyst
By assuming that 
\begin{eqnarray}
L > 2\delta^{-2} K,
\end{eqnarray}
we have that for any $x \in B$, for any sufficiently large $k_1,k_2$ satisfying \eqref{eqk1k2}, for any $t_1, t_3$ satisfying \eqref{eq t3inW3}, \eqref{eq t1inW1} and  \eqref{eq logt3logt1}, we have
\begin{eqnarray}
  d(W^3_{1}(\Phi^1_x(t_1)), W^1_{1}(\Phi^3_x(t_3))) > \frac{1}{2}C^{-1}C_1^{-K-1}|t_1|^{K/\delta}.
\end{eqnarray} 
By Lemma \ref{lem altdefQNI}, we see that $f$ has the QNI property.
\end{proof}

\begin{proof}[Proof of Theorem \ref{teo.maintechnical2}]
It suffices to combine Lemma \ref{lem Ixisempty} and Proposition \ref{prop IxnonemptytoQNI}.
\end{proof}

 \section{Continuous and uniform versions for partially hyperbolic diffeomorphisms}\label{s.uniformversions}
In this section we explain how to adapt the results in the previous sections to the case where the measure is supported in a (uniformly) partially hyperbolic set.  

Let $f: M \to M$ be a smooth diffeomorphism and $\Lambda \subset M$ a compact $f$-invariant subset. Assume that there is a continuous splitting of $T_\Lambda M = E^u \oplus E^c \oplus E^s = E^1 \oplus E^2 \oplus E^3$ and consider the functions $\lambda_{i,x}$ defined in equation \eqref{eq:exponentpointuni} which are continuous on $\Lambda$ and verify (for an appropriate metric) that  $|\lambda_x^1 |> |\lambda_x^2| > |\lambda_x^3|$ as well as $|\lambda_x^1|>1>|\lambda_x^3|$.

We will show the following. 
\begin{teo}\label{teo.QNIuniform}  
Let $f: M \to M$ be a smooth diffeomorphism of a closed 3 manifold $M$ and let $\Lambda \subset M$ be a compact $f$-invariant partially hyperbolic subset. Then the following dichotomy holds: 
\begin{itemize}
\item Either for every non degenerate $\mu$ with full support on $\Lambda$, $\mu$ has QNI (cf. Definition \ref{def.QNI}), or,
\item for every $\ell \geq 1$, the set $\Lambda$ is jointly integrable up to order $\ell$ (cf. Definition \ref{defi.jointorderell}). 
\end{itemize}
\end{teo}

Note that the second condition is independent of the measure, and forces every non-degenerate measure with full support on $\Lambda$ to not verify QNI. Also, while not obvious from the definition of the QNI property, our result implies that having this property for all non-degenerate invariant measures with full support  on some partially hyperbolic subset with good continuation properties (e.g. the whole manifold) is an open property   in the smooth topology.

\subsection{Proof of Theorem \ref{teo.QNIuniform}}  
 As for the measurable case, the proof has three stages\footnote{Note that whenever possible, we will use the results from previous sections, particularly \S \ref{s.computations} and \S \ref{s.compatible}. We note that in those sections, the fact that templates are measurable functions included an extra difficulty that here we could do without if we wanted to show the results here directly. We leave those simplifications to the interested reader.}: 
 
 \begin{itemize}
 \item First we show that if QNI is not verified, then there are $\ell$-good stable and unstable charts for all $\ell$. In this case,  the normal form coordinate depends continuously on its base point. We will need to check that  these $\ell$-good charts will  depend continuously on the base point. The proof mimics what is done in \S\ref{s.computations}.
 \item Then we show that if QNI is not verified, then, the approximations of the stable and unstable  Hopf brushes (cf. Remark \ref{rem-HB})  are at the same up to order $\ell$. This proof mimics the one done in \S\ref{s.compatible} and indeed, in this case, no continuity is needed.
 \item Finally, we show that this compatibility of charts implies that there is a  continuous  family of surfaces that approximates well the Hopf brushes up to order $\ell$. 
 \end{itemize}
 
 Let us give the main arguments and see how to adapt what has already been done:   

\begin{lemma} \label{lem lgoodcharts-unif}
Let $\mu$ be a measure of full support on $\Lambda$ and assume that $\Lambda$ does not admit $\ell$-good uniform unstable charts for some $\ell \geq 1$. Then, $\mu$ has QNI. 
\end{lemma}

\begin{proof}
	By Proposition \ref{p.CNFuniform},  $0$-good uniform unstable charts exist.
Let $0 \leq k < \ell$ be the largest number such that $\Lambda$ admits $k$-good uniform unstable charts. We claim that $\cT^k_x$ is not polynomial for $\mu$-a.e. $x \in \Lambda$, and $\mu$ has QNI.   

 Suppose that $\cT^k_x$ is polynomial for $\mu$-a.e. $x \in \Lambda$. Then we would have case (i) in  Proposition \ref{p.dichotomytech}.  As the stable and unstable manifolds of $f$ through $\Lambda$ have uniformly bounded smoothness, we deduce that in fact $\cT^k_x$ is polynomial for every $x \in \Lambda$.
 This would allow us to construct $(k+1)$-good uniform unstable charts using Proposition \ref{p.CNFuniform}, contradicting the choice of $k$.

Now, the rest of the proof of Proposition \ref{p.qniellgood} works verbatim. 
\end{proof}

  Assume that $\mu$ does not have QNI, then by Lemma \ref{lem lgoodcharts-unif}, 
 there are $\ell$-good stable charts and $\ell$-good unstable charts for every integer $\ell \geq 0$ which form a collection of  compatible good charts by Theorem \ref{teo.maintechnical}. 
 Moreover, under the hypothesis of Theorem \ref{teo.QNIuniform}, we may apply Lemma \ref{p.CNFuniform} to show that for each $\ell \geq 0$, the $\ell$-good stable (unstable) charts may be chosen to depend continuously on the base point.
 Then by the proof of Proposition \ref{prop.compatibleimpliesjoint}, 
 we see that $\Lambda$ is jointly integrable up to order $\ell$ for every $\ell \geq 0$. This concludes the proof of Theorem \ref{teo.QNIuniform}.

%
%

\appendix

\section{Discussion on the notion of QNI}\label{app.symmetry}

In this appendix we provide some alternative ways to understand the QNI property and prove Proposition \ref{prop.symmetry}.

\begin{proof}[Proof of Proposition \ref{prop.symmetry}] 
Assume $\mu$ has the QNI property for $f$ as stated in Definition \ref{def.QNI}. We wish to show that it also verifies the property for $f^{-1}$. For this, consider $\alpha>0$, $\eps>0$ and $\nu>0$ and we will consider the set $\cP$ given by the fact that $\mu$ has the QNI property and some value of $C = C(\nu,\eps)$ (which may differ from the one given for $\mu$) and $k_0$ as given for $\mu$. 

To get the result, it is enough to show that there is a function $\rho(\nu)$ such that $\rho(\nu) \to 0$ as $\nu\to 0$ so that if $k>k_0$ and $x, f^k(x),f^{-k}(x) \in \cP$ there is a subset $\hat U_x \subset W^{1,k}_1(x)$ with $\mu_x^1(\hat U_x)> (1-\rho(\nu)) \mu_x^1(W^{1,k}_1(x))$ with the property that given $z \in \hat U_x$ there is $\hat S_z \subset W^{3,k}_1(x)$ with  $\mu_x^3(\hat S_z) > (1-\rho(\nu)) \mu_x^3(W^{3,k}_1(x))$ so that if $y \in \hat S_z$ then 
\begin{equation}\label{eq:QNI2}
d(W^1_1(y), W^3_1(z)) > \hat C  e^{-\alpha k}.   
\end{equation}

Consider the set of pairs $(y,z) \in W^{3,k}_1(x) \times W^{1,k}_1(x)$ which verify equation \eqref{eq:QNI2}. It follows from the fact that $\mu$ verifies QNI that this set has measure larger than $(1-\nu)^2$ with respect to the probability measure $\frac{\mu^{3}_x}{\mu^3_x(W^{3,k}_1(x))} \times \frac{\mu^{1}_x}{\mu^1_x(W^{1,k}_1(x))}$ and thus, by Fubini's theorem it follows that considering $\rho(\nu) = 2 \sqrt{\nu}$ the result follows. 
\end{proof}

The following characterization of QNI is the one we establish to prove our main results.

\begin{lemma} \label{lem altdefQNI}
Assume that $\mu$ is a partially hyperbolic measure for a smooth diffeomorphism $f$ satisfying the following property. 
\begin{itemize}
\item there exist $V > 2$, $\alpha>0$ and, 
\item for every $\eps>0$, there exists a subset $\cP_0 \subset M$ of measure $\mu(\cP_0)>1-\eps$

and,
\item for every $\nu > 0$,  there exist $k_* = k_*(  \nu,\eps)$, and  a constant $C=C(\nu,\eps)$ so that 
  such that:
\end{itemize}
\qquad if $k_1, k_2 \geq k_*$ with $\frac{k_2}{k_1} \in (\frac{2}{3}V, \frac{3}{2}V)$ and  $x, f^{k_1}(x), f^{-k_2}(x) \in \cP_0$,  then
\begin{itemize}
\item
there is a subset $S_x \subset W^{3,k_2 }_1(x)$ with $\mu_x^3(S_x) > (1-\nu) \mu_x^3(W^{3,k_2  }_1(x))$ with the following property: 
\item  For all $y \in S_x$ there exists $U_y \subset  W^{1,k_1  }_1(x)$ with $\mu_x^1(U_y) > (1-\nu) \mu_x^1(W^{1,k_1 }_1(x))$ so that if $z \in U_y$ then 
\begin{equation}\label{eq:QNI3}
d(W^3_1(z), W^1_1(y) ) > C  e^{-\alpha k_1}. 
\end{equation}
\end{itemize}
Then $\mu$ has the QNI property (cf. Definition \ref{def.QNI}).
\end{lemma}

Note that the condition on $k_1,k_2$ says that equation \eqref{eq:QNI3} (up to possibly changing $\alpha$) is the same as asking that $d(W^1_1(y), W^3_1(z)) > C  e^{-\alpha \min \{ k_1, k_2\}}$ or other variations.

\begin{proof}
We fix some $\eps \in (0,1)$.
We set
\begin{eqnarray} \label{eq a0b0} 
a_0 = \frac{1 + \frac{2}{3}V }{\frac{2}{3}V  - 1}, \quad
b_0 = \frac{1 + \frac{3}{2}V }{\frac{3}{2}V  - 1}  \in ( a_0, 1).
\end{eqnarray}
We can without loss of generality suppose that $\cP_0$ satisfies that 
\begin{eqnarray} \label{eq mumeasureofb}
 \mu(\cP_0) > 1- \min( \frac{\varepsilon}{100},  \frac{b_0 - a_0}{4b_0}).
\end{eqnarray}

We define in the following way the set $\cP$ in Definition \ref{def.QNI}. 
Given a constant $N > 0$,
we let $\cP = \cP(\varepsilon, N)$ be the set of points $x \in \cP_0$ such that for every $k > N$, 
\begin{eqnarray}
\frac{1}{k} |\{ 0 \leq j \leq k-1 :  f^{j}(x) \in \cP_0  \}| > 1 - \min(\frac{\eps}{50},  \frac{b_0 - a_0}{2b_0}), \\
\frac{1}{k} |\{ - k \leq j \leq -1 :  f^{j}(x) \in \cP_0  \}| >  1 -  \min(\frac{\eps}{50},  \frac{b_0 - a_0}{2b_0}).
\end{eqnarray}
 Using Birkhoff's theorem we may assume by letting $N$ be sufficiently large that
\begin{eqnarray}
\mu(\cP) > 1 - \varepsilon.
\end{eqnarray}
We now show that the statement in Definiton \ref{def.QNI} is satisfied for this $\cP$.

Fix some $\nu \in (0, 1)$.
We let  $k_* = k_*(\eps, \nu)$ be given by the hypothesis of the lemma.  

Let us take some $x \in \cP$ and an integer $k \geq k_0$ such that $f^k(x) \in \cP$ and $f^{-k}(x) \in \cP$.

We denote $k' = (\frac{a_0+b_0}{2b_0})k$.
By the definition of $\cP$, we have
\aryst
|\{ a_0 k < l < b_0 k' : f^{l}(x) \in \cP_0 \}| &>&  (1 - \frac{b_0 - a_0}{2b_0} ) b_0k' - a_0 k - 1 > 0.
\earyst
Consequently, there exists some $j \in \{ a_0 k + 1, \cdots, b_0 k' \}$ such that $x' := f^j(x) \in \cP_0$.

Denote $k_2 = k  + j$ and $k_1 = k  - j$. Then we have
\aryst
f^{k_1}(x') \in \cP_0, \ f^{-k_2}(x') \in \cP_0 \ \mbox{ and } \
 \frac{k_2}{k_1} \in (\frac{2}{3}V, \frac{3}{2}V).
\earyst

By the hypothesis of the lemma,  there exists a subset $S' \in W^{3, k_2 }_{1}(x')$ with $\mu^3_{x'}(S') > (1- \nu)\mu^3_{x'}( W^{3, k_2 }_{1}(x'))$ such that for any $y \in S'$ there exists $U'_y \subset  W^{1, k_1 }_{1}(x')$ with $\mu^1_{x'}(U'_y) > (1- \nu)\mu^1_{x'}(W^{1, k_1 }_{1}(x'))$ such that if $z \in U'_y$, then
\begin{eqnarray}
d(W^1_1(y), W^3_1(z)) > Ce^{-\alpha k}.
\end{eqnarray}
We define $S_x = f^{-j}(S')$ and for each $y \in S_x$, define $U_y = f^{-j}(U'_{f^j(y)})$.
Notice that $\mu^3_{x'} = (f^{j})_* \mu^3_x$ and $\mu^1_{x'} = (f^{j})_* \mu^1_x$.
Then it is clear that the statement in Definition \ref{def.QNI} holds for $x$ by letting $\alpha$ be larger.
\end{proof}

We end this appendix by commenting the difference between our definition of QNI and that in \cite{Katz}. 

The only difference between the definitions is the choice of the notion of local stable/unstable manifolds. We have chosen to work with $W^i_1(x)$ (with $i\in \{1,3\}$) to be the unstable/stable manifold of lenght $1$ with respect to the normal form coordinates. Note that the Riemannian length of these manifolds is not continously variating as the normal form coordinates are just measurable, but they vary continuously in sets of arbitrarily large measure. To choose the scales, we have chosen to use $W^{1,k}_1(x) = f^{-k}(W^{1}_1(f^k(x)))$ and $W^{3,k}_1(x) = f^k(W^3_1(f^{-k}(x))$. In \cite{Katz} he first introduces a (sufficiently small) measurable partition $\cB$ of the lamination with a Markov property and defines $W^1_{loc}(x) = W^1_1(x) \cap \cB(x)$. Then, he takes $W^{1,k}_{loc}(x)$ to be $f^{-k}(W^1_{loc}(f^k(x))$ (a symmetric partition allows to define local stable manifolds). The definition of QNI in \cite{Katz} it is then identical to Definition \ref{def.QNI} where the sets $W^{1,k}_1(x)$ and $W^{3,k}_1(x)$ are replaced by $W^{1,k}_{loc}(x)$ and $W^{3,k}_{loc}(x)$. 

As it is usual, to see the equivalence, one considers large measure sets of points where the \lq boundary\rq  of the leaves $W^{1}_{loc}(x)$ and $W^{3}_{loc}(x)$ is 'far' from the center point $x$. In those sets, and for iterates which return to those sets there is an easy way to relate the sets $W^{i,k}_{loc}(x)$ and $W^{i,k}_1(x)$ and thus one can go from one definition to the other without difficulty.

\section{An application of Lusin's theorem} 

We have used the following general result repeatedly.

\begin{prop}\label{prop.measurable}
Let $\mu$ be a partially hyperbolic measure for a diffeomorphism $f$ of a $3$-dimensional closed manifold $M$. Assume that $c_1, \ldots, c_k$ are measurable functions with respect to $\mu$. Then, for every $\eps > 0$ and compact set $\cQ\subset M$ with $\mu(\cQ)>1-\eps/2$ there exists compact subsets $\cP \subset \cP_0 \subset \cP_1 \subset \cQ$ and $C,k_0>0$ such that $\mu(\cP)> 1-\eps$ and such that: 
\begin{enumerate}
\item\label{item.continuity} all functions $c_1, \ldots, c_k$ are uniformly continuous on $\cP_1$, 
\item \label{item.density}
 for every $\nu > 0$, 
 there exists an integer $m_0 > 1$ such that for every integer $m > m_0$, for every $x \in \cP_0$ and for both $i\in \{1,3\}$ we have that  $\mu^i_x(\cP_1 \cap W^{i, m}_{1}(x)) \geq (1- \nu ) \mu^i_x(W^{i, m}_{1}(x))$.
\end{enumerate}
\end{prop}

\begin{proof}
Item \eqref{item.continuity} is a standard application of Lusin's theorem.   Without loss of generality, let us assume that $\mu(\cP_1) > 1 - 2\eps/3$.

To see item \eqref{item.density}, we define for any integers $q, m \geq 2$ a subset of $\cP_1$ by the formula
\aryst
\cQ_{q, m} = \{ x \in \cP_1 : \mu^i_x(\cP_1 \cap W^{i, m'}_{1}(x)) \geq (1- q^{-1} ) \mu^i_x(W^{i, m'}_{1}(x)), \forall i \in \{1,3\}, m' \geq m \}.
\earyst
Fix an arbitrary integer $q \geq 2$. We have $\lim_{m \to \infty} \mu(\cP_1 \setminus \cQ_{q, m}) = 0$.
We choose some $m_q \geq 2$ such that  $\mu(\cP_1 \setminus \cQ_{q, m_q}) < \eps/(100q^2)$. Then we take
$\cP_0 = \cap_{q \geq 2} \cQ_{q, m_q}$.
It is clear that $\mu(\cP_0) > 1 - 3\eps/4$, and satisfies item \eqref{item.density}.

\end{proof}

\section{Some statements about cocycles}\label{ss.apcocycles}
Here we give some proofs of some results which are probably well-known but not available in the literature. The reason is that not many references deal with cocycles which are only smooth along unstable manifolds. We state a particular case since it is the one we will use, but of course it holds in more generality. We use the notation and definitions from \S\ref{s.cocyclenormalforms}. Note that this can be seen as just a generalization of the fact that Pesin unstable manifolds are smooth. Note that the following result is implicit in \cite[Remark 5.2(b)]{Ruelle}.
\begin{prop}\label{p.oseledetssmooth} 
Let $f: M \to M$ be a $C^{\infty}$ smooth diffeomorphism preserving an ergodic partially hyperbolic measure $\mu$. Let $\cE \to M$ be a (measurable) two-dimensional vector bundle over $(M,\mu)$ and let $A: \cE \to \cE$ be a vector bundle automorphism,  both of which   are smooth along unstable manifolds. Assume that the Lyapunov exponents of $A$ with respect to $\mu$ are  $\alpha > \beta$, corresponding to Oseledets subspaces $E_{\alpha}$ and $E_{\beta}$ respectively.  

Then there exists a family of smooth trivializations
 $\cY_0 = \{ \cY_{0,x}= (\xi_{0,x}, \xi_{0,x}^\perp) \}_{x \in M}$ such that for $\mu$-a.e. $x$,
\aryst
 A^{\cY_0} (x, t) = \begin{pmatrix} \alpha_x(t) & r^0_x(t) \\ 0 & \beta_x(t) \end{pmatrix}
\earyst
where   $\alpha_x, \beta_x, r_x: (-1,1) \to \RR$  are  smooth functions. Moreover, for $\mu$-a.e. $x$, we have that $\xi_{0,x}(0) \in E_{\alpha}(x)$.
 \end{prop}

Let us recall that the fact that $A$ is smooth along unstable manifolds implicitly requires the bundle $\cE$ to be defined and be smooth along unstable manifolds (see Remark \ref{rem.smoothunstables}). This means that for $\mu$-almost every $x \in M$, the bundle $\cE$ is defined over $W^1_1(x)$ and admits a smooth trivialization $\cY$ making $A^\cY(x, \cdot)$ smooth as a function from $(-1,1)$ to $\mathrm{GL}(2,\RR)$. 

%
%
%

We also have the following parallel version.

\begin{prop} \label{p.oseledetssmooth.uniform}
	Let $f : M \to M$ be a $C^{\infty}$ smooth diffeomorphism preserving a uniform partially hyperbolic set $\Lambda$. Let $\cE \to \Lambda$ be a two-dimensional vector bundle over $\Lambda$, and let $A : \cE \to \cE$ be a bundle automorphism, both of which are smooth along the unstable manifolds. Assume that $A|_{\cE}$ admits a continuous dominated splitting $\cE = E' \oplus E''$, i.e.,  $\| A |_{E'} \| > \| A|_{E''} \|$ pointwise.
	
	Then there exists a continuous family of smooth trivializations
	$\cY_0 = \{ \cY_{0,x}= (\xi_{0,x}, \xi_{0,x}^\perp) \}_{x \in \Lambda}$ such that for every $x \in \Lambda$,
	\aryst
	A^{\cY_0} (x, t) = \begin{pmatrix} \alpha_x(t) & r^0_x(t) \\ 0 & \beta_x(t) \end{pmatrix}
	\earyst
	where  $\alpha_x, \beta_x, r_x: (-1,1) \to \RR$  are  smooth functions with uniformly (in $x$) bounded smooth norms. Moreover,  we have that $\xi_{0,x}(0) \in E'(x)$.
\end{prop}

We will omit the proof of Proposition \ref{p.oseledetssmooth.uniform} since it is in close parallel with that of Propositon \ref{p.oseledetssmooth}: it is enough to check the uniformity of the estimates on various functions at each step of the construction in the proof of Propositon \ref{p.oseledetssmooth}. 


\begin{proof}
By hypothesis, there is a family of trivializations $\cY = \{  \cY_{x} = (\xi_x, \xi^\perp_x) \}_{x \in M}$ such that for $\mu$ almost every $x \in M$, $\xi_x, \xi_x^\perp : (-1,1) \to \cE$ are smooth maps so that $\xi_x(t), \xi_x^\perp(t) \in \cE_{\Phi^1_x(t)}$ are linearly independent. 
Moreover, we may assume without loss of generality that for $\mu$-a.e. $x \in M$, $\xi_x(0) \in E_{\alpha}(x)$ and $\xi_x^{\perp}(0) \in E_{\beta}(x)$. 

The restriction of the bundle map $A$ on $\cE|_{W^1_1(x)}$, seen under the basis $(\xi_x, \xi^\perp_x)$ and $(\xi_{f(x)}, \xi^\perp_{f(x)})$, is given by the matrix
\aryst
A^{\cY}(x, t) = 
\begin{bmatrix} 
\alpha_x(t) & r_x(t) \\
q_x(t) & \beta_x(t)
\end{bmatrix}.
\earyst
Here functions $\alpha_x$, $r_x$, $\beta_x$, $q_x$ are smooth. Moreover, we have $q_x(0) = 0$ by our choices of $\xi_x(0)$ and $\xi_x^{\perp}(0)$.

Let us define another family of trivializations $\hat\cY = \{ \hat\cY_x = ( \hat\xi_x, \xi_x^{\perp} ) \}_{x \in M}$ by setting
\aryst 
\hat\xi_x =  \eta_x \xi_x + p_x \xi_x^{\perp}
\earyst
where $p_x$ is a smooth function on $(-1,1)$ satisfying $p_x(0) = 0$; and $\eta_x$ is a non-vanishing smooth function on $(-1,1)$ satisfying $\eta_x(0) = 1$.
Then the  restriction of the bundle map $A$ on $\cE|_{W^1_1(x)}$, seen under the basis $(\hat\xi_x, \xi^\perp_x)$ and $(\hat\xi_{f(x)}, \xi^\perp_{f(x)})$, is given by the matrix
\aryst
&& A^{\hat\cY}(x, t) = \begin{bmatrix}
\eta_{f(x)}(\lambda_{1,x} t)^{-1} & 0 \\  - \eta_{f(x)}(\lambda_{1,x} t)^{-1} p_{f(x)}( \lambda_{1, x} t) & 1
\end{bmatrix}
 A^{\cY}(x, t) 
\begin{bmatrix}
\eta_x(t) & 0 \\ p_x(t) & 1
\end{bmatrix}
 =
 \begin{bmatrix}
 \hat \alpha_x(t) & \hat r_x(t) \\
\hat q_x(t) & \hat \beta_x(t)
 \end{bmatrix}
\earyst

\noindent where 
\begin{equation} \hat q_x(t)=    \eta_x(t) q_x(t) + p_x(t) \beta_x(t) - (\eta_{f(x)}^{-1} p_{f(x)})( \lambda_{1,x}t) ( \eta_x(t) \alpha_x(t) + p_x(t) r_x(t)).
\end{equation}
We will choose $\eta_x$ and $p_x$ such that for $\mu$-a.e. $x$ we have the equations
\ary
  \eta_x(t)  q_x(t) + p_x(t) \beta_x(t) &=& \alpha_x(0) p_{f(x)}( \lambda_{1,x}t), \label{eq equationforp} \\
\eta_x(t) \alpha_x(t) + p_x(t) r_x(t) &=& \alpha_x(0)  \eta_{f(x)}( \lambda_{1,x}t).  \label{eq equationforeta}
\eary

If we denote $\lambda^-_{1,m}(x) = (\lambda_{1, f^{-1}(x)} \cdots \lambda_{1, f^{-m}(x)})^{-1}$, we can solve the equation \eqref{eq equationforp} by setting
\aryst 
p_x(t) =  \sum_{n = 1}^{\infty} [\prod_{j = 1}^{n-1} \alpha_{f^{-j}(x)}(0)]^{-1}[\prod_{j = 1}^{n-1} \beta_{f^{-j}(x)}(\lambda^{-}_{1, j}(x) t)] \alpha_{f^{-n}(x)}(0)^{-1} (\eta_{f^{-n}(x)} q_{f^{-n}(x)})(\lambda^{-}_{1, n}(x) t).
\earyst
Notice that the above sum converges since $\alpha > \beta$, and
\aryst
\lim_{n \to \infty} n^{-1} \sum_{j = 1}^{n-1}\log  \alpha_{f^{-j}(x)}(0) = \alpha \ \mbox{and } \ \lim_{n \to \infty} n^{-1} \sum_{j = 1}^{n-1}\log  \beta_{f^{-j}(x)}(0) = \beta.
\earyst 
Then we can solve the equation \eqref{eq equationforeta} by setting
\aryst
\eta_{x}(t) =  \sum_{n = 1}^{\infty} [\prod_{j = 1}^{n-1} \alpha_{f^{-j}(x)}(0)]^{-1}[\prod_{j = 1}^{n-1} \alpha_{f^{-j}(x)}(\lambda^{-}_{1, j}(x) t)] \alpha_{f^{-n}(x)}(0)^{-1} ( p_{f^{-n}(x)}  r_{f^{-n}(x)} )(\lambda^{-}_{1, n}(x) t).
\earyst
Thus the equations \eqref{eq equationforp} and \eqref{eq equationforeta} are simultaneously solvable.  We see that $A^{\hat \cY}(x, \cdot)$ is of form 
$ \begin{bmatrix}
* & * \\ 0 & *
\end{bmatrix}. $ This concludes the proof.
\end{proof}

\end{document}